\documentclass{amsart}

\usepackage{amsmath,amssymb,amsthm,graphicx,pstricks,comment,amscd,amsfonts,latexsym}
\usepackage[all]{xy}

\usepackage{lscape}
\usepackage{hyperref}
\usepackage{tikz}
\usetikzlibrary{shapes}

\usepackage{epstopdf} 

\theoremstyle{plain} 
\newtheorem{theorem}{Theorem}[section]
\newtheorem{lemma}[theorem]{Lemma}

\newtheorem{proposition}[theorem]{Proposition}
\newtheorem{corollary}[theorem]{Corollary}

\newtheorem*{theostar}{Theorem}

\theoremstyle{definition} 
\newtheorem{example}[theorem]{Example}

\newtheorem{definition}[theorem]{Definition}
\newtheorem{remark}[theorem]{Remark}

\newtheorem{notation}[theorem]{Notation}
\newtheorem{assumption}[theorem]{Assumption}

\newcommand{\Ima}{\operatorname{\rm Im\,}}

\newcommand{\Hom}[1]{\operatorname{{\rm Hom}}_{#1}}

\newcommand{\Mod}{\mbox{{\rm Mod \!}}}

\newcommand{\perf}{\mbox{{\rm per \!}}}

\newcommand{\demo}[1]{\textsc{Proof.} #1 \hfill $\Box$ \bigskip}

\newcommand{\cB}{\mathcal{B}}
\newcommand{\cC}{\mathcal{C}}
\newcommand{\cD}{\mathcal{D}}

\newcommand{\cM}{\mathcal{M}}

\newcommand{\cP}{\mathcal{P}}
\newcommand{\cR}{\mathcal{R}}

\newcommand{\cT}{\mathcal{T}}

\newcommand{\bC}{\mathbb{C}}

\newcommand{\bZ}{\mathbb{Z}}

\newcommand{\ttau}{\widetilde{\tau}}
\newcommand{\tM}{{\tt M}^{\widetilde{\tau}}}
\newcommand{\tB}{{\tt B}^{\widetilde{\tau}}}
\newcommand{\tSigma}{\widetilde{\Sigma}}
\newcommand{\flip}{\mathfrak{f}}
\newcommand{\tgamma}{\widetilde{\gamma}}

\newcommand{\eg}[1]{^{ #1\!}}

\begin{document}
\title[The cluster category of a surface with punctures]{The cluster category of a surface with punctures via group actions}

\author{Claire Amiot}
\address{Universit\'e Grenoble Alpes, Institut Fourier, CS 40700, 38058 Grenoble cedex 09}
\email{Claire.Amiot@univ-grenoble-alpes.fr}

\author{Pierre-Guy Plamondon}
\address{Laboratoire de Math\'ematiques d'Orsay, Universit\'e Paris-Sud, CNRS, Universit\'e Paris-Saclay, 91405 Orsay, France}
\email{pierre-guy.plamondon@math.u-psud.fr}

\thanks{The authors are supported by the French ANR grant SC3A (ANR-15-CE40-0004-01)}
\dedicatory{Dedicated to Idun Reiten on the occasion of her 75th birthday.}

\date{\today}


\begin{abstract}
Given a certain triangulation of a punctured surface with boundary, we construct a new triangulated surface without punctures which covers it.
This new surface is naturally equipped with an action of a group of order two, and its quotient by this action recovers the original surface.
We show that the group acts on the quivers with potentials associated to the surfaces, and that their Ginzburg dg algebras are skew group
algebras of each other, up to Morita equivalence.
We then use these results to construct functors between the generalized cluster categories associated to the triangulations.
This allows us to give a complete description of the indecomposable objects of these categories in terms of curves on the surface,
when the surface has punctures and non-empty boundary.
\end{abstract}

\maketitle

\tableofcontents

\section*{Introduction}

 The cluster algebra $\mathcal{A}(Q)$ of a quiver $Q$ was defined by Fomin and Zelevinsky in their seminal paper \cite{FZ1}. It is a commutative algebra with a certain set of generators called cluster variables which can be computed by iterating mutations of the quiver $Q$. In the case where $Q$ is the adjacency quiver of a triangulation of an unpunctured marked surface $(\Sigma,\cM)$, the situation is especially nice, since the cluster variables are in natural bijection with the arcs on $(\Sigma, \cM)$ by \cite{FST}. When the triangulation comes from a surface with punctures, the situation is more complicated and a notion of tagged arcs is introduced in \cite{FST} in order to get an analogue bijection with cluster variables.
 
 A strong link between cluster algebras and representations of quivers was established via the construction of the cluster category, first associated to an acyclic quiver \cite{BMRRT} (and in \cite{CCS} in type $A_n$) and then via its generalized version $\cC_{(Q,S)}$ \cite{Ami09} associated to a quiver with potential in the sense of \cite{DWZ}. These are triangulated categories with a certain class of objects called cluster-tilting objects. In the case where the quiver with potential $(Q(\tau),S(\tau))$ is associated to a triangulation $\tau$ of a surface (as introduced in \cite{LF08}, and in \cite{ABCP} for unpunctured surfaces), the indecomposable summands of cluster-tilting objects in $\cC_{\tau}=\cC_{(Q(\tau),S(\tau))}$ are in bijection with tagged arcs (see \cite{BZ,QZ}). Moreover, a complete description of all indecomposable objects of the category $\cC_\tau$ in terms of homotopy classes of curves on $\Sigma$ is given in \cite{BZ} in the unpunctured case, and much representation-theoretic information, such as the Auslander-Reiten translation or components, can be recovered from operations on the surface and the arcs and curves on it.
 
 In this paper, we give a link between the punctured case and the unpunctured case, and use it to give a description of the cluster category of a triangulated punctured surface. 
 More precisely, given a certain triangulation $\tau$ of a punctured surface $(\Sigma, \cM,\cP)$, 
 we construct an unpunctured surface $(\tSigma, \widetilde{\cM})$ together with a triangulation $\ttau$ and triangle functors between the categories $\cC_\tau$ and $\cC_{\ttau}$. 
 The new surface $\tSigma$ comes naturally with an order two homeomorphism, and the surface $(\Sigma, \cM, \cP)$ can be recovered from $(\tSigma, \widetilde{\cM},\sigma)$ 
 via a bijection $\tSigma/\sigma\to \Sigma$. This gives a structure of orbifold to $\Sigma$, where (order 2) orbifold points are points in $\cP$. 
 As a main consequence, we use the description of \cite{BZ} of indecomposable objects in $\cC_{\ttau}$ to deduce a complete description of indecomposable objects in $\cC_\tau$ 
 in terms of the orbifold fundamental groupoid of $\Sigma$:
 \begin{theostar}[Corollaries \ref{cor::map-string-to-tagged-arcs} and \ref{cor::bands-final}]
  Let $(\Sigma, \cM, \cP)$ be a marked surface with non-empty boundary and possibly with punctures. Let $\tau$ be a triangulation of $\Sigma$ such that each puncture belongs to a self-folded triangle and such that no triangle shares a side with two self-folded triangles. 
  Then the indecomposable objects of the cluster category $\cC_\tau$ are in bijection with the following sets:
  \begin{itemize}
   \item $\pi_1^{\rm orb}(\Sigma, \cM; \cP')$ (see Definition \ref{defi::P'}),
   
   \item  $\big\{ \{\gamma, \gamma^{-1} \} \ | \ \gamma \in \pi_1^{\rm orb}(\Sigma, \cM), \gamma \neq \gamma^{-1} \big\}$,
          
   \item $\Big\{ [\gamma]\in \pi_1^{\rm orb,free}(\Sigma) |\ [\gamma]\neq [\gamma^{-1}]\Big\}\times k^*/\sim $,
   \item $\Big\{ [\gamma]\in \pi_1^{\rm orb,free}(\Sigma) |\ \gamma^2\neq 1 \textrm{ and }[\gamma]= [\gamma^{-1}]\Big\}\times k^*\backslash \{\pm 1\}/ \sim$,
   \item $\Big\{  [\gamma]\in \pi_1^{\rm orb,free}(\Sigma) |\ \gamma^2\neq 1 \textrm{ and }[\gamma]= [\gamma^{-1}]\Big\}\times (\bZ/2\bZ)^2$,
  \end{itemize}
  where $\sim$ is the equivalence relation given by $([\gamma], \lambda) \sim ([\gamma^{-1}], \lambda)$.
 \end{theostar}
 This result is a generalization of \cite[Theorem 1.1]{BZ}, which treats the case where the set of punctures of $\Sigma$ is empty. 
 
 Our construction of the surface $(\tSigma, \widetilde{\cM})$ can be seen as a generalization of that of \cite[Section 3.5]{FZ03} and \cite[Section 12.4]{FZ2}, who studied cluster algebras of type $D_n$ via symmetric pairs of diagonals of a $2n$-gon.
 
 \medskip
 The construction of the functors between $\cC_\tau$ and $\cC_{\ttau}$ goes back to the study of skew group algebras introduced by Reiten in Riedtmann in \cite{RR85}. 
 Given an algebra $\Lambda$ and a finite group $G$ acting on $\Lambda$ by automorphism, the authors in \cite{RR85} defined the skew group algebra $\Lambda G$ 
 and studied the different properties of the functors linking the categories $\mod \Lambda$ and $\mod \Lambda G$. 
 In the present paper, we adapt the situation in the context of $G=\bZ/2\bZ$ acting on a quiver with potential $(Q,S)$. 
 We let the group act on the corresponding Ginzburg dg algebra (as defined in \cite{G06}), and show that the resulting skew group dg algebra is Morita equivalent to the Ginzburg dg algebra of a quiver
 with potential $(Q_G, S_G)$ which we describe explicitly (see Theorem \ref{prop::action-on-Ginzburg}).
 In the case where $(Q(\tau),S(\tau))$ is the quiver with potential associated to a certain tagged triangulation $\tau$ of a punctured surface, 
 we have a natural action of $\bZ/2\bZ$ on $(Q(\tau),S(\tau))$ and we prove that the quiver with potential $(Q(\tau)_G,S(\tau)_G)$ arises as the quiver with potential of a triangulated unpunctured surface.
 
 The case of a free group action on a quiver with potential has recently been studied in \cite{PS17}, where functors between cluster categories are also obtained.  In our situation, however, the group action is never free, so other methods from skew group algebras need to be used.  Also, even though we use the orbifold structure on $\Sigma$ given by our group action, our results differ from the works on cluster algebras from orbifolds, see for instance \cite{FSTu}.  Indeed, while we study skew group algebras and objects which are ``simply-laced'' (such as quivers, skew-symmetric matrices, etc.), the results of \cite{FSTu} reflect a folding procedure via this group action and deals with ``non-simply-laced'' objects (such as valued quivers, skew-symmetrizable matrices, etc.).

 \medskip
 
 The paper is organized as follows. 
 In Section~\ref{section1}, we recall Reiten-Riedtmann's constructions on skew group algebras, and study in detail the case of an action of a group of order two.
 In Section~\ref{section1bis}, we extend these construction to dg algebras and prove that in our setting, skew group Ginzburg dg algebras are Morita equivalent to Ginzburg dg algebras. 
 From there, we deduce triangle functors between the corresponding cluster categories. 
 In Section~\ref{section2}, we apply the results of Section~\ref{section1bis} in the case where the quiver with potential $(Q(\tau),S(\tau))$ comes from a triangulation of a punctured surface $(\Sigma,\cM,\cP)$. 
 This allows us to construct a new surface $(\tSigma, \widetilde{\cM})$. 
 Section~\ref{section3} is devoted to the description of the indecomposable objects of $\cC_{\tau}$ in terms of curves on the surface $\tSigma$. 
 We use the orbifold structure of $(\Sigma, \cM, \cP)$ in Section~\ref{section4} to get a description of the indecomposable objects of $\cC_\tau$ in terms of curves on $\Sigma$. 
 Finally, Section~\ref{section5} is dedicated to examples.   
 
 \bigskip
 
 \subsection*{Conventions}
 We compose arrows of quivers from right to left, as for function.  If $\alpha$ is an arrow of a quiver, then $s(\alpha)$ is its source
 and $t(\alpha)$ is its target.  All modules over algebras are right modules.
 
\section{Skew group algebras}\label{section1}

\subsection{Definition for associative algebras}
In this section, we follow \cite[Introduction]{RR85}.

Let $k$ be a commutative ring, and $\Lambda$ be an Artin $k$-algebra.  Let $G$ be a finite group acting on $\Lambda$ by $k$-algebra automorphisms.

\begin{definition}\label{defi::skew}
 The \emph{skew group algebra} $\Lambda G$ is the $k$-algebra defined thus:
   \begin{enumerate}
    \item Its underlying $k$-module is $\Lambda\otimes_k kG$.  
    \item Multiplication is given by $(\lambda \otimes g)(\lambda' \otimes g') = \lambda \cdot g(\lambda')\otimes gg'$ and extended to all of $\Lambda G$ by distributivity.
   \end{enumerate}  
\end{definition}

There is a natural monomorphism of $k$-algebras
\begin{eqnarray*}
  \Lambda & \longrightarrow & \Lambda G \\
    \lambda & \longmapsto & \lambda \otimes 1.
\end{eqnarray*}

Note that the algebra $\Lambda G$ is not basic in general.
\subsection{The case of $G=\bZ/2\bZ$}\label{sect::cyclic}
In this paper, we will be concerned only with certain actions of the cyclic group of order $2$, so for the rest of the section, we fix $G=\bZ/2\bZ = \{1,\sigma \}$.

Let $k$ be a field whose characteristic is not equal to $2$.  Let $Q$ be a finite quiver, and $\Lambda = kQ/I$ be a quotient of the path algebra $kQ$ by an admissible ideal $I$ (admissible means that if $\cR$ is the ideal generated by the arrows of $Q$, then there exists an integer $m\geq 2$ such that $\cR^m \subseteq I \subseteq \cR^2$).

 We will assume that $G$ acts on $\Lambda$ in the following way:

\begin{assumption}\label{assu::action-on-quiver}
 The action of $G$ on $\Lambda = kQ/I$ is induced by an action of $G$ on the quiver $Q$, such that if two vertices $i$ and $j$ are fixed under this action, then all arrows from $i$ to $j$ are also fixed under this action.
\end{assumption}

This implies that vertices are sent to vertices and arrows are sent to arrows.  This assumption is strong: in general, the action of $G$ sends arrows to linear combinations of arrows.

Under this assumption, we will describe a basic algebra Morita equivalent to $\Lambda G$, applying the results of \cite{RR85}.

Let $Q_0 = V \coprod W$, where
\begin{itemize}
 \item $V$ is the set of vertices fixed by $G$;
 \item $W$ is the set of vertices not fixed by $G$.
\end{itemize}
Then a complete set of pairwise orthogonal primitive idempotents of $\Lambda G$ is given by the union of the sets

$$ 
  \{ \frac{1}{2}e_i\otimes (1+\sigma), \frac{1}{2}e_i\otimes (1-\sigma) \ \big| \ i\in V\},
$$
$$
  \{ \frac{1}{2}(e_j+e_{\sigma(j)})\otimes(1+\sigma) \ \big| \ j\in W \}, \textrm{ and}
$$
$$
  \{ \frac{1}{2}(e_j+e_{\sigma(j)})\otimes(1-\sigma) \ \big| \ j\in W \}.
$$
We will use the following notations:

\begin{definition}\label{defi::+-}

  \begin{itemize}
    \item For each $i\in V$, we put $e_i^\pm := \frac{1}{2}e_i\otimes (1\pm \sigma)$;
    \item For each $i\in W$, we put $e_i^\pm := \frac{1}{2} (e_i+e_{\sigma(i)})\otimes (1\pm   \sigma)$.
		\item For each arrow $\alpha\in Q_1$, we put $\alpha^\pm = e_{t(\alpha)}^\pm (\alpha\otimes 1) e_{s(\alpha)}^\pm$.
		\item For any path $w=\alpha_1 \cdots \alpha_m$, we put $w^\pm = \alpha_1^\pm \cdots \alpha_m^\pm$.
  \end{itemize}

\end{definition}

Note that $e_{i}^\pm = e_{\sigma(i)}^\pm$.  Moreover, it can be shown that for any $j\in W$, the indecomposable projective modules $e_j^+ \Lambda$ and $e_j^- \Lambda$ are isomorphic. Thus, if we let $o(W)$ denote a set of representatives of the $G$-orbits in $W$, and we put
$$
 \bar{e} = \sum_{i\in V} (e_i^+ + e_i^-) + \sum_{j\in o(W)} e_j^+,
$$
 then $\bar{e} \Lambda G \bar{e}$ is a basic algebra Morita-equivalent to $\Lambda$.

\begin{lemma}\label{lemm::arrows1234}
Let $\alpha:i\to j$ be an arrow in $Q$.

\begin{enumerate}

 \item\label{VV} If $i\in V$ and $j\in V$, then  
 $$
 \alpha^\pm := e_j^\pm(\alpha\otimes 1)e_i^\pm = \frac{1}{2}\alpha\otimes (1\pm \sigma).
 $$
 
 \item\label{VW} If $i\in V$ and $j\in W$, then  
 $$
 \alpha^\pm := e_j^+(\alpha\otimes 1)e_i^\pm = \frac{1}{4}(\alpha\pm\sigma(\alpha))\otimes (1\pm \sigma).
 $$

 \item\label{WV} If $i\in W$ and $j\in V$, then  
 $$
 \alpha^\pm := e_j^\pm(\alpha\otimes 1)e_i^+ = \frac{1}{4}(\alpha\pm\sigma(\alpha))\otimes (1 + \sigma).
 $$

 \item\label{WW} If $i\in W$ and $j\in W$, then  
 $$
 \alpha^+ := e_j^+(\alpha\otimes 1)e_i^+ = \frac{1}{4}(\alpha + \sigma(\alpha) ) \otimes (1 + \sigma).
 $$

\end{enumerate}

In particular, in cases (\ref{VW}) and (\ref{WV}), we have that $\alpha^\pm = \pm \sigma(\alpha)^\pm$, while in case (\ref{VV}), we have that $\alpha^\pm = \sigma(\alpha)^\pm$, and in case (\ref{WW}), we have $\alpha^+=\sigma(\alpha)^+$.
\end{lemma}

From there, an application of \cite[Section 2.4]{RR85} allows us to compute the Gabriel quiver $Q_G := Q_{\bar{e} \Lambda G \bar{e}}$ of $\bar{e} \Lambda G \bar{e}$.  

\begin{proposition}[\cite{RR85}]\label{prop::quiver-QG}
Under Assumption \ref{assu::action-on-quiver}, the Gabriel quiver $Q_G$ of $\bar{e} \Lambda G \bar{e}$ is defined as follows:
  \begin{itemize}
    \item The vertices of $Q_G$ correspond to the idempotents $e_i^\pm$ (for $i\in V$) and $e_j^+$ (for $j\in W$).  We denote them by $i^\pm$ and $j^+$, respectively.
    \item Let $o(Q_1)$ be a set of representatives of the $G$-orbits of arrows of $Q$.  Then for any arrow $\alpha\in o(Q_1)$,
      \begin{enumerate}
        \item if $i\in V$ and $j\in V$, then $\alpha^\pm$ corresponds to an arrow $i^\pm \to j^\pm$ in $Q_G$;
        
        \item if $i\in V$ and $j\in W$, then $\alpha^\pm$ corresponds to an arrow $i^\pm \to j^+$ in $Q_G$;
        
        \item if $i\in W$ and $j\in V$, then $\alpha^\pm$ corresponds to an arrow $i^+ \to j^\pm$ in $Q_G$;
        
        \item if $i\in W$ and $j\in W$, then $\alpha^+$ corresponds to an arrow $i^+ \to j^+$ in $Q_G$.
      \end{enumerate}
  \end{itemize}
\end{proposition}

Relations on the quiver of $\bar{e}\Lambda G \bar{e}$ can also be obtained from those on $Q$.

\begin{notation}
We will sometimes write $j$ instead of $j^+$ if $j\in W$ and $\alpha$ instead of $\alpha^+$ if both endpoints of $\alpha$ are in $W$.
\end{notation}

There is a natural application
\begin{eqnarray*}
 \iota: \Lambda & \longrightarrow &\bar{e} \Lambda G \bar{e} \\
          \lambda & \longmapsto &\bar{e} (\lambda \otimes 1) \bar{e}.
\end{eqnarray*}
which is not a morphism of algebras.  Nevertheless, we have the following

\begin{lemma}\label{lemm::iota-of-a-path}
Let $w$ be an element in the radical of $\Lambda$, and $\alpha \in Q_1$. Denote by $i$ the start of $\alpha$. Then we have 

\[\iota(\alpha w) = \begin{cases} \iota(\alpha)\iota(e_i w) & \textrm{if } i\in V,\\ 2\iota(\alpha)\iota(e_i w) &  \textrm{if } i\in W.\end{cases}\]
In particular, if $w=\alpha_1\ldots \alpha_r$ is a path in $Q$, then
\[ 
  \iota (w) = 2^s \iota(\alpha_1) \ldots \iota(\alpha_r),
\]
where $s$ is the number of arrows in $\{ \alpha_1, \ldots, \alpha_{r-1} \}$ whose 
starting point is in $W$ (recall that we compose arrows from right to left).
\end{lemma}

\demo{Since $\iota$ is linear, it is enough to show it for $w$ a path of length $\geq 1$. If $\alpha$ and $w$ do not compose, then the statement clearly holds. So assume $\alpha$ and $w$ compose and that $i\in V$, then we have:
 \begin{eqnarray*}
  \iota(\alpha)\iota(w) &=& \bar e (\alpha\otimes 1) \bar e (w \otimes 1) \bar e \\
	             &=& \bar e (\alpha\otimes 1) (e_i^+ + e_i^-) (w \otimes 1) \bar e\\
							 &=& \bar e (\alpha\otimes 1) (e_i \otimes 1) (w \otimes 1) \bar e\\
							 &=& \bar e (\alpha\otimes 1)(w \otimes 1) \bar e\\
							 &=& \bar e (\alpha w\otimes 1) \bar e\\
							 &=& \iota(\alpha w).
\end{eqnarray*}

If $i$ is in $W$ then we have \begin{eqnarray*}
  \iota(\alpha)\iota(w) &=& \bar e (\alpha\otimes 1) \bar e (w \otimes 1) \bar e \\
	             &=& \bar e (\alpha\otimes 1) (\frac{1}{2}(e_i + e_{\sigma(i)})\otimes (1+\sigma)) (w \otimes 1) \bar e\\
							 &=& \frac{1}{2}\bar e ((\alpha\otimes (1+\sigma))  (w \otimes 1) \bar e\\
							 &=& \frac{1}{2}\bar e (\alpha w\otimes 1 +\alpha \sigma w\otimes \sigma) \bar e\\
							 &=& \frac{1}{2}\iota(\alpha w) \textrm{ since }\alpha \textrm{ and } \sigma w \textrm{ do not compose}.
\end{eqnarray*}

}

\begin{remark}\label{remark::iota(alpha)}
The application $\iota:\Lambda \to \bar{e}\Lambda G\bar{e}$ defined above sends arrows $\alpha$ of $Q$ to either $\alpha^+ + \alpha^-$ if we are in situations (\ref{VV}),  (\ref{VW}) or  (\ref{WV}) of Lemma \ref{lemm::arrows1234}, and to $\alpha^+$ if we are in situation  (\ref{WW}).
\end{remark}

We end this subsection by giving an expression of $\iota\circ\sigma$ that will be useful in several computations.
\begin{lemma}\label{lemma::iota(sigma w)}
Let $w$ be a path from $i$ to $j$ in $Q$.  Then we have 
\[\iota(\sigma w)=\left\{\begin{array}{lr} e_j^+\iota(w)e_i^++e_j^-\iota(w)e_i^--e_j^+\iota(w)e_i^--e_j^-\iota(w)e_i^+ & \textrm{if }i,j\in V;\\
 \iota(w)e_i^+-\iota(w)e_i^- & \textrm{if }i\in V, j\in W;\\
 e_j^+\iota(w)-e_j^-\iota(w) & \textrm{if }i\in W, i\in V;\\
 \iota(w) & \textrm{if }i,j\in W.\end{array}\right.\]
\end{lemma}
\demo{ The proof is done by induction on the length of $w$. If $w=\alpha$ is an arrow, the statement follows directly from Proposition \ref{prop::quiver-QG} together with Remark \ref{remark::iota(alpha)}.

Assume the result holds for any path of length $r$ and prove it for a path $w'=\alpha w$ of length $r+1$. Denote by $i=s(w)$, $j=t(w)=s(\alpha)$ and $k=t(\alpha)$.  We have then eight cases to consider depending on wether $i$,$j$ and $k$ belong to $V$ or $W$. 

\medskip

\emph{Case 1: $i,j,k\in V$.}

\noindent
We have the following equalities:

\[\begin{array}{rcl} \iota(\sigma(\alpha)\sigma(w)) & = & \iota(\sigma(\alpha))\iota(\sigma (w)) \quad\textrm{ by Lemma \ref{lemm::iota-of-a-path}}\\
 & = & \iota(\alpha) \iota(\sigma(w)) \quad\textrm{ since }j,k\in V   \\
 & = & (\alpha^++\alpha^-)(e_j^+\iota(w)e_i^++e_j^-\iota(w)e_i^--e_j^+\iota(w)e_i^--e_j^-\iota(w)e_i^+ )\\ & = & \alpha^+(\iota(w)e_i^+-\iota(w)e_i^-) +\alpha^-(\iota(w)e_i^--\iota(w) e_i^{+}) \\ &=&e_k^+\alpha^{+}\iota(w)e_i^+-e_k^+\alpha^-\iota(w)e_i^--e_k^-\alpha^-\iota(w)e_i^++e_k^-\alpha^-\iota(w)e_i^- \\ & = &e_k^+\iota(\alpha)\iota(w)e_i^+-e_k^+\iota(\alpha)\iota(w)e_i^--e_k^-\iota(\alpha)\iota(w)e_i^++e_k^-\iota(\alpha)\iota(w)e_i^-. \end{array}\]
 
 \medskip
 
 \emph{Case 2: $i,j\in V$ and $k\in W$}
 
 \noindent
 We have the following equalities:

\[\begin{array}{rcl} \iota(\sigma(\alpha)\sigma(w)) & = & \iota(\sigma(\alpha))\iota(\sigma (w)) \quad\textrm{ by Lemma \ref{lemm::iota-of-a-path}}\\
 & = &(\alpha^+-\alpha^-)(e_j^+\iota(w)e_i^++e_j^-\iota(w)e_i^--e_j^+\iota(w)e_i^--e_j^-\iota(w)e_i^+ )\\
  & = & (\alpha^+-\alpha^-)\iota(w)e_i^+-(\alpha^++\alpha^-)\iota(w)e_i^-\\
  & = &\iota(\alpha w)(e_i^+-e_i^-).
\end{array}\]

\medskip
\emph{Case 3: $i\in V$, $j\in W$ and $k\in V$.}

\noindent
We have then the following equalities:

\[\begin{array}{rcl} \iota(\sigma(\alpha)\sigma(w)) & = & 2 \iota(\sigma(\alpha))\iota(\sigma (w)) \quad\textrm{ by Lemma \ref{lemm::iota-of-a-path}}\\ &=& 2(\alpha^+-\alpha^-)(\iota(w) e_i^+-\iota(w) e_i^-) \\ & = & 2(\alpha^+\iota(w) e_i^+-\alpha^-\iota(w) e_i^+-\alpha^+\iota(w)e_i^-+\alpha^-\iota(w)e_i^-).\end{array}\]

We get then the result since we have $$\iota(\alpha w)= 2(\alpha^+\iota(w) e_i^++\alpha^-\iota(w) e_i^++\alpha^+\iota(w)e_i^-+\alpha^-\iota(w)e_i^-).$$

\medskip
\emph{Case 4: $i\in V$ and $j,k\in W$}

\noindent
We have then the following equalities:

\[\begin{array}{rcl} \iota(\sigma(\alpha)\sigma(w)) & = & 2 \iota(\sigma(\alpha))\iota(\sigma (w)) \quad\textrm{ by Lemma \ref{lemm::iota-of-a-path}}\\ &=& 2\alpha^+(\iota(w)e_i^+-\iota(w)e_i^-) \\ &= & \iota(\alpha w) e_i^+-\iota(\alpha w) e_i^-.
\end{array}\]
 The remaining cases are either dual or very similar and are left to the reader.
}

\subsection{Group action on $\bar{e} \Lambda G \bar{e}$}
As before, we let $G= \{1, \sigma \}$ act on the quiver $Q$, inducing an action on $\Lambda$.  
It is proved in \cite{RR85} that the dual group $\hat{G}$ acts on $\Lambda G$, and that the resulting skew group algebra
$(\Lambda G)\hat{G}$ is Morita-equivalent to $\Lambda$.  It is also observed that $\hat{G}$ should act on $\bar{e} \Lambda G \bar{e}$.
In this section, we make this action precise.

Let $o(W)$ be a set of representatives of $G$-orbits of $W$.  Define
\[
  \varepsilon := \sum_{i\in V} (e_i^+ + e_i^-) + \sum_{j\in o(W)} (e_j - e_{\sigma(j)}) \otimes 1 \quad\in \Lambda G.
\]

Note that the first term can be written as $\sum_{i\in V} e_i \otimes 1$.  Note, also, that the second term depends on the choice of $o(W)$.

\begin{lemma}\label{lemm::epsilon-square}
  We have that $\varepsilon^2 = 1$.
\end{lemma}
\demo{
Since the $e_i^\pm$ ($i\in V$) are pairwise orthogonal primitive idempotents, we have that the square of the first term is
\[
  (\sum_{i\in V} (e_i^+ + e_i^-))^2 = \sum_{i\in V} (e_i^+ + e_i^-).
\]

Next, the square of the second term is
\begin{eqnarray*}
  (\sum_{j\in o(W)} (e_j - e_{\sigma(j)}))^2 \otimes 1 & = & \sum_{j\in o(W)} (e_j - e_{\sigma(j)})^2 \otimes 1 \\
                                                       & = & \sum_{j\in o(W)} (e_j + e_{\sigma(j)}) \otimes 1 \\
                                                       & = & \sum_{j\in o(W)} (e_j^+ + e_j^-). 
\end{eqnarray*}

Finally, the orthogonality of the $e_i$ ($i\in V$) and $e_j$ ($j\in W$) implies that the two terms of $\varepsilon$ are orthogonal to each other.

Therefore, $\varepsilon^2 = \sum_{i\in V} (e_i^+ + e_i^-) + \sum_{j\in o(W)} (e_j^+ + e_j^-) = 1$.
}

Let $E:\Lambda G \longrightarrow \Lambda G : x\mapsto \varepsilon x \varepsilon$ be the conjugation by $\varepsilon$.

\begin{lemma}\label{lemm::E-automorphism}
  The map $E$ defined above is an algebra automorphism of $\Lambda G$.  Morevoer,
  \begin{itemize}
    \item if $i\in V$, then $E(e_i^\pm) = e_i^{\pm}$;
    \item if $j\in W$, then $E(e_j^\pm) = e_j^{\mp}$.
  \end{itemize}
\end{lemma}
\demo{ Since $\varepsilon$ is invertible by Lemma \ref{lemm::epsilon-square}, then $E$ is an algebra automorphism.

Let $k\in V$.  Then
\begin{eqnarray*}
  E(e_k^\pm) & = & \varepsilon e_k^\pm \varepsilon \\ 
             & = & \Big(\sum_{i\in V} (e_i^+ + e_i^-) + \sum_{j\in o(W)} (e_j - e_{\sigma(j)}) \otimes 1 \Big)   e_k^\pm   \Big(\sum_{i\in V} (e_i^+ + e_i^-) + \sum_{j\in o(W)} (e_j - e_{\sigma(j)}) \otimes 1 \Big) \\
             & = & e_k^\pm,
\end{eqnarray*}
the last line being obtained by using orthogonality relations between the $e_i$'s and~$e_i^\pm$'s.

Now, let $k\in W$.  Let $\delta = \begin{cases} 1 & \textrm{  if $k\in o(W),$} \\ -1 & \textrm{  if $k \notin o(W).$} \end{cases}$ Then
\begin{eqnarray*}
  E(e_k^\pm) & = & \varepsilon e_k^\pm \varepsilon \\
             & = & \Big(\sum_{i\in V} (e_i^+ + e_i^-) + \sum_{j\in o(W)} (e_j - e_{\sigma(j)}) \otimes 1 \Big) \Big( \frac{1}{2} (e_k + e_{\sigma(k)}) \otimes (1\pm \sigma) \Big) \varepsilon \\
             & = & \frac{1}{2}\Big( \sum_{j\in o(W)} (e_j - e_{\sigma(j)}) \otimes 1 \Big) \Big( \frac{1}{2} (e_k + e_{\sigma(k)}) \otimes (1\pm \sigma) \Big) \varepsilon \\
             & = & \frac{1}{2}\Big( \delta(e_k - e_{\sigma(k)})\otimes (1\pm\sigma) \Big) \varepsilon \\
             & = & \delta \frac{1}{2} \Big( (e_k - e_{\sigma(k)})\otimes (1\pm\sigma) \Big)  \Big(\sum_{i\in V} (e_i^+ + e_i^-) + \sum_{j\in o(W)} (e_j - e_{\sigma(j)}) \otimes 1 \Big) \\
             & = & \delta \frac{1}{2} \delta \Big( (e_k + e_{\sigma(k)}) \otimes 1 \pm (-e_k - e_{\sigma(k)})\otimes \sigma \Big) \\
             & = & \frac{1}{2} \Big( (e_k + e_{\sigma(k)}) \otimes ( 1 \mp \sigma) \Big) \\
             & = & e_k^\mp.
\end{eqnarray*}
}

Let $\hat{G} = \{ 1, \hat{\sigma} \}$ be the dual group of $G$.  We know from \cite{RR85} that $\hat{G}$ acts on $\Lambda G$ by
\[
  \hat{\sigma} \star (\lambda \otimes h) := \lambda \otimes \hat{\sigma}(h)h.
\]
However, this action does not restrict to an action on $\bar{e}\Lambda G\bar{e}$ in general.  To obtain an action on $\bar{e}\Lambda G\bar{e}$, we need to twist by the automorphism $E$:

\begin{proposition}\label{prop::G-dual-action}
  The assignment 
  \[
    \hat{\sigma} \cdot (\lambda \otimes h) := E(\lambda \otimes \hat{\sigma}(h) h)
  \]
  defines an action of $\hat{G}$ on $\Lambda G$.  This action restricts to an action of $\hat{G}$ on $\bar{e} \Lambda G \bar{e}$.
\end{proposition}
\demo{The proposed action of $\hat{\sigma}$ is an automorphism of $\Lambda G$, since it is the composition of $E$ and of the action of $\hat{\sigma}$ by $\star$.
In order to prove the first claim, it suffices to prove that this automorphism is an involution.

By definition, $\hat{\sigma} \cdot (\lambda\otimes h) = E(\hat{\sigma} \star (\lambda \otimes h))$.  Hence
\begin{eqnarray*}
  \hat{\sigma}\cdot (\hat{\sigma}\cdot (\lambda \otimes h)) & = & \hat{\sigma}\cdot E(\hat{\sigma}\star (\lambda\otimes h)) \\
               & = & \hat{\sigma} \cdot \varepsilon (\hat{\sigma}\star (\lambda\otimes h)) \varepsilon \\
               & = & E\Big(\hat{\sigma} \star ( \varepsilon (\hat{\sigma}\star (\lambda\otimes h)) \varepsilon ) \Big) \\
               & = & E\Big( (\hat{\sigma} \star \varepsilon) (\hat{\sigma} \star ( \hat{\sigma} \star (\lambda\otimes h)) ) (\hat{\sigma}\star \varepsilon) \Big) \\
               & = & E\Big( (\hat{\sigma} \star \varepsilon) (\lambda \otimes h) (\hat{\sigma} \star \varepsilon) \Big) \\
               & = & E\Big( \varepsilon (\lambda\otimes h) \varepsilon \Big) \\
               & = & \varepsilon^2 (\lambda\otimes h) \varepsilon^2 \\
               & = & \lambda\otimes h.
\end{eqnarray*}
The first claim is proved.  The second claim follows from the following computations, using Lemma \ref{lemm::E-automorphism}:
\begin{itemize}
  \item if $i\in V$, then $\hat{\sigma}\cdot e_i^\pm = E(\hat{\sigma}\star e_i^\pm) = E(e_i^\mp) = e_i^\mp$;
  \item if $j\in W$, then $\hat{\sigma}\cdot e_j^\pm = E(\hat{\sigma}\star e_j^\pm) = E(e_i^\mp) = e_i^\pm$.
\end{itemize}
}

\begin{lemma}\label{lemm::G-dual-action-on-vertices}
Let $\hat{G}$ act on $\bar{e}\Lambda G \bar{e}$ as in Proposition \ref{prop::G-dual-action}.
  \begin{enumerate}
    \item The action of $\hat{G}$ on the idempotents is defined as follows.
          \begin{itemize}
            \item If $i\in V$, then $\hat{\sigma}\cdot e_i^\pm = e_i^\mp$.
            \item If $j\in W$, then $\hat{\sigma}\cdot e_j^+ = e_j^+$.
          \end{itemize}
    
    \item The action of $\hat{G}$ on the arrows is defined as follows.  Let $\alpha:i\to j$ be an arrow in $Q$.
          \begin{itemize}
            \item If $i\in V$ and $j\in V$, then $\hat{\sigma}\cdot \alpha^\pm = \alpha^\pm$.
            \item If $i\in V$ and $j\in W$, then $\hat{\sigma}\cdot \alpha^\pm = \delta_j \alpha^\pm$, where $\delta_j = \begin{cases} 1 & \textrm{ if $j\in o(W)$,} \\ -1 & \textrm{ else.} \end{cases}$
            \item If $i\in W$ and $j\in V$, then $\hat{\sigma}\cdot \alpha^\pm = \delta_i \alpha^\pm$, where $\delta_i = \begin{cases} 1 & \textrm{ if $i\in o(W)$,} \\ -1 & \textrm{ else.} \end{cases}$
            \item If $i\in W$ and $j\in W$, then $\hat{\sigma}\cdot \alpha^+ = \delta_i\delta_j \alpha^+$, where $\delta_i$ and $\delta_j$ are as above.
          \end{itemize}
  \end{enumerate}
\end{lemma}
\demo{
The first two equalities were obtained at the end of the proof of Proposition \ref{prop::G-dual-action}.  
The others are straightforward computations, of which we only write one instance. If $i\in V$ and $j\in W$, then
\begin{eqnarray*}
  \hat{\sigma} \cdot \alpha^\pm & = & \hat{\sigma}\cdot (e_j^+ (\alpha\otimes 1) e_i^\pm) \\
                                & = & (\hat{\sigma}\cdot e_j^+) (\hat{\sigma}\cdot (\alpha\otimes 1)) (\hat{\sigma}\cdot e_i^\pm) \\
                                & = & e_j^+ (\varepsilon(\alpha\otimes 1) \varepsilon) e_i^\mp \\
                                & = & e_j^+\big( \sum_{i\in V} e_i\otimes 1 + \sum_{j\in o(W)} (e_j - e_{\sigma(j)})\otimes 1 \big)  (\alpha\otimes 1) \varepsilon e_i^\mp \\
                                & = & e_j^+ (\delta_j\alpha \otimes 1) \varepsilon e_i^\mp \\
                                & = & \delta_j e_j^+ (\alpha\otimes 1) \big( \sum_{i\in V} e_i\otimes 1 + \sum_{j\in o(W)} (e_j - e_{\sigma(j)})\otimes 1 \big) e_i^\mp \\
                                & = & \delta_j e_j^+ (\alpha\otimes 1) e_i^\mp \\
                                & = & \delta_j \alpha^\mp.
\end{eqnarray*}
The other cases are computed in a similar fashion.
}

\subsection{Skew group algebra of $\bar{e}\Lambda G\bar{e}$}
Our next aim is to describe the skew group algebra of $\bar{e}\Lambda G \bar{e}$ under the action of $\hat{G}$.

\subsubsection{Admissible choice}
To do so, we will need an additional assumption using the following definition.

\begin{definition}\label{defi::admissible}
 Let $Q$ be a quiver with an action of $G=\{1, \sigma\}$ as above.  A set $o(W)$ of representatives of the orbits of the vertices in $W$
 is \emph{admissible} if any arrow having an endpoint in $o(W)$ has its other endpoint either in $V$ or in $o(W)$.  
 A set $o(Q_1)$ of representatives of the orbits of the arrows is \emph{admissible} with respect to $o(W)$ if the arrows of $o(Q_1)$ are
 precisely the arrows whose endpoints are in $o(W)$ or $V$.
\end{definition}

\begin{example}
 \begin{enumerate}
  \item Consider the quiver given by
  \[
  \scalebox{0.8}
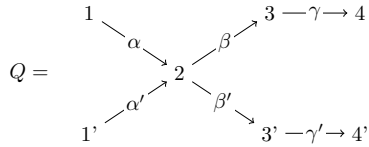
{
  \begin{tikzpicture}
  
   \node at (0,0) {$Q = $};
 
 \node (2) at (2.5 ,0) {2};
 \node (1) at (1, 1) {1};
 \node (1') at (1, -1) {1'};
 \node (3) at (4, 1) {3};
 \node (3') at (4, -1) {3'};
 \node (4) at (5.5, 1) {4};
 \node (4') at (5.5, -1) {4'};
 
 \draw[->] (1) -- (2) node[midway, fill=white, inner sep=1pt]{$\alpha$};
 \draw[->] (1') -- (2) node[midway, fill=white, inner sep=1pt]{$\alpha'$};
 \draw[->] (2) -- (3) node[midway, fill=white, inner sep=1pt]{$\beta$};
 \draw[->] (2) -- (3') node[midway, fill=white, inner sep=1pt]{$\beta'$}; 
 \draw[->] (3) -- (4) node[midway, fill=white, inner sep=1pt]{$\gamma$};
 \draw[->] (3') -- (4') node[midway, fill=white, inner sep=1pt]{$\gamma'$};
 
\end{tikzpicture}

}
  \]
  together with the action of $G$ sending $1$, $3$ and $4$ to $1'$, $3'$ and $4'$, respectively, and fixing $2$.
  Then the $o(W) = \{1,3,4\}$ and $o(Q_1)= \{\alpha, \beta, \gamma\}$ are admissible.

  \item Consider the quiver 
  \[
  \scalebox{0.8}
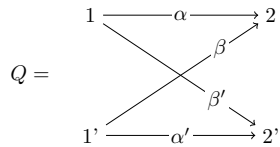
{
  \begin{tikzpicture}
  
   \node at (-1,1) {$Q = $};

 \node (1) at (0, 2) {1};
 \node (1') at (0, 0) {1'};
 \node (2) at (3, 2) {2};
 \node (2') at (3, 0) {2'};

 \draw[->] (1) -- (2) node[midway, fill=white, inner sep=1pt]{$\alpha$};
 \draw[->] (1') -- (2) node[near end, fill=white, inner sep=1pt]{$\beta$};
 \draw[->] (1) -- (2') node[near end, fill=white, inner sep=1pt]{$\beta'$};
 \draw[->] (1') -- (2') node[midway, fill=white, inner sep=1pt]{$\alpha'$}; 
 
\end{tikzpicture}

}
  \]
with action of $G$ sending $1$ and $2$ to $1'$ and $2'$, respectively, and $\alpha$ and $\beta$ to $\alpha'$ and $\beta'$, respectively.
Then there is no admissible choice of $o(W)$. 
If, nevertheless, we choose $o(W)=\{1, 2\}$ and $o(Q_1) = \{\alpha, \beta\}$,
then
\[
  \scalebox{0.8}{
  \begin{tikzpicture}
  
   \node at (-1,0) {$Q_G = $};

 \node (1) at (0, 0) {1};
 \node (1+) at (0, 0.2) {};
 \node (1-) at (0, -0.2) {};
 \node (2) at (3, 0) {2};
 \node (2+) at (3, 0.2) {};
 \node (2-) at (3, -0.2) {};

 \draw[->] (1+) -- (2+) node[midway, fill=white, inner sep=1pt]{$\alpha^+$};
 \draw[->] (1-) -- (2-) node[near end, fill=white, inner sep=1pt]{$\beta^+$};
 
\end{tikzpicture}

}
  \]
with $\hat{\sigma}(\alpha^+) = \alpha^+$ and $\hat{\sigma}(\beta^+) = -\beta^+$.
Thus Assumption \ref{assu::action-on-quiver} is not satisfied for $Q_G$.
 \end{enumerate}

\end{example}

\begin{corollary}\label{coro::G-dual-action}
 
  Assume that $o(W)$ and $o(Q_1)$ are admissible as in Definition \ref{defi::admissible}.  
  Then the action of the action of $\hat{G}$ on the algebra $\bar{e}\Lambda G \bar{e}$ is induced by an action of $\hat{G}$ on its Gabriel quiver $Q_G$ (see Proposition \ref{prop::quiver-QG}).
  Let 
  \[
   V_G := \{ i \ | \ i \in W \} \quad \textrm{and} \quad W_G := \{j^\pm \ | \ j\in V \}
  \]
  be the sets of fixed vertices and non-fixed vertices, respectively, by the action of $\hat{G}$ on $Q_G$.
  Then the choices $$o(W_G) = \{i^+ \ | \ i\in V \}\quad \textrm{and}$$ 
  $$o((Q_G)_1) = \{\alpha^+ \ | \ \alpha \textrm{ with at least one endpoint in $V$} \} \cup \{\beta^+ \ | \ \beta \textrm{ has both endpoints in $o(W)$} \}$$
  are admissible.
  
\end{corollary}
\demo{
For admissible choices of $o(W)$ and $o(Q_1)$, all the $\delta_i$ and $\delta_j$ of Lemma \ref{lemm::G-dual-action-on-vertices} are equal to $1$.  This proves the first claim.
The second claim follows from the description of the quiver $Q_G$ in Proposition \ref{prop::quiver-QG}.
}
\subsubsection{Quiver of $(\bar{e}\Lambda G\bar{e})\hat{G}$}
In view of Corollary \ref{coro::G-dual-action}, if there is an admissible choice of $o(W)$, then we can apply Proposition \ref{prop::quiver-QG} to the skew group algebra $(\bar{e}\Lambda G \bar{e})\hat{G}$.
We then get a set of pairwise primitive idempotents:
\begin{itemize}
  \item for $j\in W$, $(e_j^+)^\pm = \frac{1}{2}e_j^+ \otimes (1\pm \hat{\sigma})$, and
  \item for $i\in V$, $(e_i^\pm)^+ = \frac{1}{2}(e_i^\pm + e_i^\mp)\otimes (1 + \hat{\sigma})$.
\end{itemize}
Call $\tilde{e}$ the sum of these idempotents.
Then $\tilde{e} \big( (\bar{e}\Lambda G \bar{e})\hat{G} \big) \tilde{e}$ is a basic algebra.

The arrows of its Gabriel quiver $Q_{\hat{G}}$ can be described, again using Proposition \ref{prop::quiver-QG}.
First, we need to choose a set of representatives of the $\hat{G}$-orbits of $\{(e_i^\pm)^+ \ | \ i\in V\}$.
We choose $o(W_G) := \{ (e_i^+)^+ \ | \ i\in V\}$; this is admissible by Corollary \ref{coro::G-dual-action}.
We let $o((Q_G)_1) = \{\alpha^+ \ | \ \alpha \textrm{ with at least one endpoint in $V$} \} \cup \{\beta^+ \ | \ \beta \textrm{ has both endpoints in $o(W)$} \}$,
which is also admissible by Corollary \ref{coro::G-dual-action}.  We can also define
\[
 \iota_G: \bar{e}\Lambda G \bar{e} \to \tilde{e}\big( (\bar{e}\Lambda G \bar{e})\hat{G}\big) \tilde{e} : x\mapsto \tilde{e}(x\otimes 1)\tilde{e}.
\]

Then, for any arrow $\alpha:i\to j$ in the original quiver $Q$:
\begin{itemize}
  \item if $i\in V$ and $j\in V$, then we had two arrows $\alpha^\pm:i^\pm\to j^\pm$ in $Q_G$.  These become one $(\alpha^+)^+$ in $Q_{\hat{G}}$.
  \item if $i\in V$ and $j\in o(W)$, then we had two arrows $\alpha^\pm:i^\pm\to j$ in $Q_G$.  These become two arrows $(\alpha^+)^\pm : (e_i^+)^+ \to (e_j^+)^\pm$ in $Q_{\hat{G}}$.
  \item if $i\in o(W)$ and $j\in V$, then we had two arrows $\alpha^\pm:i\to j^\pm$ in $Q_G$.  These become two arrows $(\alpha^+)^\pm : (e_i^+)^\pm \to (e_j^+)^+$ in $Q_{\hat{G}}$.
  \item if $i\in o(W)$ and $j\in o(W)$, then we had one arrow $\alpha: i\to j$ in $Q_G$.  This becomes two arrows $(\alpha^+)^\pm : (e_i^+)^\pm \to (e_j^+)^\pm$ in $Q_{\hat{G}}$.
\end{itemize}

The following follows from these considerations.

\begin{proposition}\label{prop::isomosphim-quiver}
 Let $o(W)$ and $o(Q_1)$ be admissible as in Definition \ref{defi::admissible}.  Then the quivers $Q$ and $Q_{\hat{G}}$ are isomorphic,
 and an isomorphism is induced by:
 \begin{itemize}
  \item for each $i\in V$, $e_i\mapsto (e_i^+)^+$;
  \item for each $j\in o(W)$, $e_j\mapsto (e_j^+)^+$;
  \item for each $j\in W\setminus o(W)$, $e_j\mapsto (e_j^+)^-$;
  \item for each arrow $\alpha\in Q_1$, then $\alpha$ is sent to $(\alpha^+)^+$ if $\alpha\in o(Q_1)$ and to $(\alpha^+)^-$ otherwise.
\end{itemize}
 Moreover, this isomorphism is $G$-equivariant.
\end{proposition}

We will call $\xi: Q_{\hat{G}}\to Q$ the isomorphism inverse to the one described in Proposition \ref{prop::isomosphim-quiver}.
Then $\xi$ extends to an isomorphism from $kQ_{\hat{G}}$ to $kQ$.  

\subsubsection{Relations}
The rest of the section is devoted to showing that this isomorphism induces one between $\Lambda$ and $\tilde{e} \big( (\bar{e}\Lambda G \bar{e})\hat{G} \big) \tilde{e}$.
To this end, we will first need a technical lemma.

\begin{lemma}\label{lemm::relations-are-preserved}
 Assume that $\Lambda = kQ$, so that $\xi$ extends to an isomorphism from $kQ_{\hat{G}}$ to $kQ$.  
 Let $o(W)$ and $o(Q_1)$ be admissible as in Definition \ref{defi::admissible}.
 Let $w=\alpha_1 \cdots \alpha_m$ be a path in $Q$.  
 Let $p$ be the number of $\alpha_i$ with start in $V$ (resp. $W$) and end in $W$ (resp. $V$) if $s(w)$ is in $W$ (resp. in $V$).
 Then
 \[ 
  \xi \circ \iota_G \circ \iota (w) = \begin{cases}
                                       2^{m+p}w & \textrm{if $w$ does not have both its start and end in $W$,} \\
                                       2^{m+p-1}(w + \sigma(w)) & \textrm{if $w$ has both its start and end in $W$.}
                                      \end{cases}
 \]

\end{lemma}

\demo{ We prove the result as well as the following statement 
 \begin{equation*}\tag{$\dagger$}\label{addhyp} \textrm{``if }w\textrm{ is a path from }i\textrm{ to }j\textrm{ such that }i,j\in V\textrm{ then }\iota_G((e_j^+-e_j^-)\iota(w))=0\textrm{''}\end{equation*} by induction on the length $m$ of $w$.
 
Let $\sigma^{\mu}(\alpha):\sigma^{\mu}i\to \sigma^\mu(j)$ be an arrow in $Q$ with $\alpha\in o(Q_1)$ and $\mu=0$ or $1$. Then 
\begin{itemize}
\item if $i,j\in V$, $\xi\circ\iota_G\circ\iota (\alpha)=\xi\circ\iota_G(\alpha^++\alpha^-)=\xi(2(\alpha^+)^+)=2\alpha$;
\item if $i\in V$ and $j\in W$ then $\xi\circ\iota_G\circ\iota(\sigma^{\mu}\alpha)=\xi\circ\iota_G(\alpha^++(-1)^\mu\alpha^-)=(\alpha+\sigma\alpha)+(-1)^\mu(\alpha-\sigma\alpha)=2\sigma^{\mu}\alpha$;
\item if $i\in W$ and $j\in V$ the case is dual;
\item if $i,j\in W$, then $\xi\circ\iota_G\circ\iota(\sigma^{\mu}\alpha)=\xi\circ\iota_G(\alpha^+)=\alpha +\sigma \alpha=2^0(\sigma^{\mu}\alpha+\sigma^{\mu+1}\alpha).$ 
\end{itemize} 
Moreover if $\alpha$ is an arrow $i\to j$ with $i,j\in V$, then $\iota_G((e_j^+-e_j^-)\iota (\alpha))=\iota_G(\alpha^+-\alpha^-)=(\alpha^+)^+ -(\alpha^-)^+=\alpha-\alpha=0$, so $\eqref{addhyp}$ and the lemma hold for $r=1$. 
\medskip

Assume $w'=\sigma^{\mu}(\alpha)w$ is a path of length $r+1$, with $\alpha\in o(Q_1)$ and $\mu=0$ or $1$, and denote $i=s(w)$, $j=t(w)=s(\sigma^{\mu}\alpha)$ and $k=t(\sigma^{\mu}\alpha)$. Denote by $m$ and $p$ (resp. $m'$ and $p'$) the integers defined in the Lemma \ref{lemm::relations-are-preserved} for the path $w$ (resp. $w'$).

We first show $\eqref{addhyp}$ for $w'$. Assume $i$ and $k$ are in $V$. If $j$ is in $V$ then we have
$$\begin{array}{rcl}\iota_G((e_k^+-e_k^-)\iota(w')) &=& \iota_G((e_k^+-e_k^-)\iota(\alpha)\iota(w)) \textrm{ by Lemma \ref{lemm::iota-of-a-path}}  \\ & = & \iota_G(\alpha^+\iota(w)-\alpha^-\iota(w))\\ & = & \iota_G(\alpha^+)\iota_G(e_j^+ w)-\iota_G(\alpha^-)\iota_G(e_j^-\iota(w))\quad \textrm{by Lemma \ref{lemm::iota-of-a-path}}\\ & = & \alpha(\iota_G(e_j^+\iota(w)-e_j^-\iota(w)))=0 \textrm{ by induction hypotesis}\end{array}$$ 

If $j$ is in $W$ then we have
$$\begin{array}{rcl}\iota_G((e_k^+-e_k^-)\iota(w')) &=& 2 \iota_G((e_k^+-e_k^-)\iota(\sigma^{\mu}\alpha)\iota(w))\\ & = & 2\iota_G((\alpha^+-(-1)^{\mu}\alpha^-)\iota(w)) \\ &=& 2\iota_G(\alpha^+)\iota_G(e_j\iota(w))- (-1)^{\mu}\iota_G(\alpha^-)\iota_G(e_j\iota(w)) \\ & = & 2[(\alpha+\sigma \alpha)-(-1)^\mu (\alpha-\sigma\alpha)]\iota_G\circ\iota(w) \textrm{ since }\iota(w)=e_j\iota(w) \\ & = & 2^{m+p+2}(\sigma^{\mu+1}\alpha) w \textrm{ by induction }\\ & = & 0 \textrm{ since } \sigma^{\mu+1}\alpha \textrm{ and }w\textrm{ do not compose.}
\end{array}$$

\medskip

To prove the step for $w'=(\sigma^{\mu}\alpha) w$ we have to treat eight cases depending on wether $i$, $j$ and $k$ belong to $V$ or to $W$. 

\smallskip

\emph{Case 1: Assume that $i,j\in V$ and $k\in W$.} Then $p'=p$, and we have 
\[\begin{array}{rcl}\xi\circ\iota_G\circ\iota((\sigma^{\mu}\alpha) w) & = & \xi\circ\iota_G (\iota(\sigma^{\mu}\alpha)\iota(w) \\
 & = & \xi\circ\iota_G[(\alpha^{+}+(-1)^{\mu}\alpha^-)\iota(w)]\\ 
 & = & 2 \xi\circ\iota_G(\alpha^+)\xi\circ\iota_G(e_j^+\iota(w))+(-1)^{\mu}2 \xi\circ\iota_G(\alpha^-)\xi\circ\iota_G(e_j^-\iota(w))\\
 &=&  2 (\alpha+(\sigma\alpha))\xi\circ\iota_G(e_j^+\iota(w))+(-1)^{\mu} 2(\alpha-(\sigma\alpha))\xi\circ\iota_G(e_j^-\iota(w))\\
 & = & 2\alpha \xi\circ\iota_G(e_j^+\iota(w)+(-1)^{\mu}e_j^-\iota(w))+2(\sigma\alpha) \xi\circ\iota_G(e_j^+\iota(w)-(-1)^{\mu}e_j^-\iota(w))\\ & = & 2(\sigma^{\mu}\alpha) \xi\circ\iota_G(e_j^+\iota(w)+e_j^-\iota(w)) +2(\sigma^{\mu+1}\alpha) \xi\circ\iota_G(e_j^+\iota(w)-e_j^-\iota(w)) \\ & = & 2 (\sigma^{\mu} \alpha)\xi\circ\iota_G(\iota(w)) \quad \textrm{ by induction hypothesis }\eqref{addhyp}\\ & = & 2^{m+p+1}(\sigma^{\mu}\alpha) w \quad\textrm{ by induction hypothesis.} 
 \end{array}\]

\medskip
\emph{Case 2: $i\in V$, $j\in W$ and $k\in V$.} Then $p'=p+1$ and we have
\[\begin{array}{rcl}\xi\circ\iota_G\circ\iota((\sigma^{\mu}\alpha) w) & = & 2\xi\circ\iota_G((\alpha^++(-1)^{\mu}\alpha^-)\iota(w)) \\ & = & 2\xi\circ\iota_G(\alpha^+)\xi\circ\iota_G\circ\iota(w) +2 (-1)^{\mu}\xi\circ\iota_G(\alpha^{-})\xi\circ\iota_G\circ\iota(w) \textrm{ since }\iota(w)=e_j\iota(w) \\ &= & 2(\alpha+(\sigma\alpha)) 2^{m+p}w +2 (-1)^{\mu} (\alpha-(\sigma\alpha))2^{m+p}w\textrm{ by induction}\\ & = &
2^{m+p+2}(\sigma^{\mu}\alpha)w.\end{array}\]

\medskip

\emph{Case 3: $i\in W$, $j\in V$, $k\in W$.} Then $p'=p+1$ and we have 
\[\begin{array}{rcl}\xi\circ\iota_G\circ\iota((\sigma^{\mu}\alpha) w) & = & \xi\circ\iota_G((\alpha^++(-1)^{\mu}\alpha^-)\iota(w)) \\ 
& = & 2 (\alpha+(\sigma\alpha)) \xi\circ\iota_G( e_{j^{+}}\iota(w))+2 (-1)^{\mu} (\alpha-(\sigma\alpha))\xi\circ\iota_G (e_j^-\iota(w)) \\ 
& = & 2 (\sigma^{\mu}\alpha)\xi\circ\iota_G(e_j^+\iota(w)+e_j^-\iota(w)) +2(\sigma^{\mu+1}\alpha)\xi\circ\iota_G(e_j^+\iota(w)-e_j^-\iota(w)) \\ 
& = &  2 (\sigma^{\mu}\alpha)\xi\circ\iota_G(\iota(w)) +2(\sigma^{\mu+1}\alpha)\xi\circ\iota_G(\iota(\sigma w)) \textrm{ by Lemma \ref{lemma::iota(sigma w)}} \\ & = & 2^{m+p+1} ((\sigma^{\mu}\alpha) w +(\sigma^{\mu +1}\alpha)\sigma w)).\end{array}\]

\medskip

\emph{Case 4: $i,j\in W$ and $k\in V$.} Then $p'=p$ and we have 
\[\begin{array}{rcl}\xi\circ\iota_G\circ\iota((\sigma^{\mu}\alpha) w) & = & 2 \xi\circ\iota_G((\alpha^++(-1)^{\mu}\alpha^-)\iota(w)) \\
& = & 2(\alpha+(\sigma\alpha))\xi\circ\iota_G(\iota(w))+(-1)^{\mu}2(\alpha-(\sigma\alpha))\xi\circ\iota_G(\iota(w)) \\
 & = & 4.2^{m+p-1}(\sigma^{\mu}\alpha)(w+\sigma w) \\ & = & 2^{m+p+1}(\sigma^{\mu}\alpha)w \textrm{ since }\sigma^{\mu\alpha}\textrm{ and }w\textrm{ do not compose.}
\end{array}\]

The other cases are very similar and are left to the reader.

}

\begin{theorem}
  Let $o(W)$ and $o(Q_1)$ be admissible as in Definition \ref{defi::admissible}.
  Then the algebras $\Lambda$ and $\tilde{e} \big( (\bar{e}\Lambda G \bar{e})\hat{G} \big) \tilde{e}$ are isomorphic.
\end{theorem}
\demo{

By Proposition \ref{prop::isomosphim-quiver}, they have the same Gabriel quiver.  
Let $R$ be a set of relations, that is, a set of generators of the ideal $I$ such that $\Lambda = kQ/I$.
Since $G$ acts on $\Lambda$, we can assume that $R$ is closed under the action of $G$.
Then for any $\rho\in R$, $\iota_G\iota(\rho)$ is a relation for $\tilde{e} \big( (\bar{e}\Lambda G \bar{e})\hat{G} \big) \tilde{e}$.

Let $\zeta$ be the automorphism of $kQ$ sending each arrow $\alpha$ to either $\frac{1}{4}\alpha$ if $\alpha$ has its starting point in $V$ 
and its ending point in $W$, and to $\frac{1}{2}\alpha$ otherwise.  
Then, applying Lemma \ref{lemm::relations-are-preserved}:
\begin{itemize}
 \item if $\rho$ has its starting and ending points in $V$, then $\zeta\xi\iota_G\iota(\rho) = \rho$;
 \item the same holds if $\rho$ has its starting point in $W$ and its ending point in $V$;
 \item if $\rho$ has its starting point in $V$ and its ending point in $W$, then $\zeta\xi\iota_G\iota(\rho) = \frac{1}{2}\rho$;
 \item if $\rho$ has both its start and ending points in $W$, then $\zeta\xi\iota_G\iota(\rho) = \frac{1}{2}(\rho + \sigma(\rho))$.
 But if $i$ is the starting point of $\rho$ and $j$ its ending point, then $e_j(\rho + \sigma(\rho))e_i = \rho$ is a relation.
\end{itemize}


Therefore, the relations of $\tilde{e} \big( (\bar{e}\Lambda G \bar{e})\hat{G} \big) \tilde{e}$ are scalar multiple of those of $\Lambda$.
The two algebras are thus isomorphic.

}

\section{Application to generalized cluster categories}\label{section1bis}

\subsection{Extension to the dg setting}
In this section, we extend the previous notions to the case of differential graded (=dg) algebras.

A \emph{dg $k$-algebra} is a graded $k$-algebra $\Gamma$ together with a \emph{differential} $d$, that is, a degree-$1$ $k$-linear map from $\Gamma$ to itself satisfying the Leibnitz rule
$$ d(uv) = d(u)v + (-1)^{\deg u}ud(v)$$
for all homogenous elements $u$ and $v$.  Morphisms of dg algebras are degree-$0$ algebra morphisms which commute with the differentials. For more on dg algebras (and categories), we refer the reader to \cite{K06}.

Let $\Gamma$ be a dg algebra.  Let $G$ be a finite group acting on $\Gamma$ by dg algebra automorphisms.  

\begin{definition}
 The \emph{skew group dg algebra} $\Gamma G$ is the algebra defined as follows.
 \begin{itemize}
  \item As a $k$-module, $\Gamma G = \Gamma \otimes_k kG$.
  \item Multiplication is given by $$(x\otimes g)(y\otimes h) = x\cdot g(y) \otimes gh, $$ for all $x,y\in \Gamma$ and $g,h\in G$.
  \item The differential is defined by $$ d(x\otimes g) = d(x)\otimes g.$$
 \end{itemize}
\end{definition}

\begin{proposition}
 The algebra $\Gamma G$ is a dg algebra.
\end{proposition}
\demo{We only need to check the Leibnitz rule.  Let $x\in \Gamma^i$, $y\in \Gamma^j$, $g,h\in G$.  Then
  \begin{eqnarray*}
   d((x\otimes g) (y\otimes h)) & = & d(x(g(y))\otimes gh) \\
                                & = & d(x(g(y)))\otimes gh \\
                                & = & d(x)g(y)\otimes gh + (-1)^i xd(g(y))\otimes gh \\
                                & = & (d(x)\otimes g)(y\otimes h) + (-1)^i x\cdot g(d(y)) \otimes gh \\
                                & = & d(x\otimes g) (y\otimes h) + (-1)^i (x\otimes g)(d(y)\otimes h) \\
                                & = & d(x\otimes g) (y\otimes h) + (-1)^i (x\otimes g)d(y\otimes h).
  \end{eqnarray*}

}

\begin{corollary}\label{coro::action-on-H0}
 If $G$ acts on $\Gamma$ as above, then $G$ also acts on $H^0\Gamma$, and $(H^0\Gamma)G = H^0(\Gamma G)$.
\end{corollary}
\demo{Since the action of $G$ commutes with the differential of $\Gamma$, it preserves its image and kernel.  Thus $G$ acts on $H^0\Gamma$.
Moreover, if $d_{G}$ is the differential of $\Gamma G$, then it is clear that $\ker d^0_G = (\ker d^0)\otimes_k kG$ 
and $\Ima d^{-1}_G = (\Ima d^{-1}) \otimes_k kG$.
This implies that $(H^0\Gamma)G = H^0(\Gamma G)$.
}

\subsection{Functors}
Let, as before, $G$ be a finite group acting on a dg algebra $\Gamma$.  Then $\Gamma$ is a subalgebra of $\Gamma G$.  This induces an exact functor between the module categories
$$ F:\Mod(\Gamma) \longrightarrow \Mod(\Gamma G) $$
defined by $F=?\otimes_\Gamma \Gamma G$.  The image of a module $M$ by this functor can be seen as a $k$-module (or even a $\Gamma$-module) as $\oplus_{g\in G} M$.

The functor $F$ has an adjoint $F'$ sending a $\Gamma G$-module to its restriction to $\Gamma$.  We see that $FF'(M)=\oplus_{g\in G} M$.

Let $\cD\Gamma$ be the derived category of $\Gamma$, $\perf\Gamma$ be the perfect derived category (the full subcategory of $\cD \Gamma$ generated by $\Gamma$ and stable under taking direct summands) and $\cD_{fd}\Gamma$ be the full subcategory of $\cD \Gamma$ whose objects are those whose homology is of finite total dimension over $k$.

\begin{proposition}\label{prop::derived-functors}
 The derived functor $LF:\cD\Gamma \to \cD\Gamma G$ restricts to functors
 $$ LF: \perf\Gamma \longrightarrow \perf\Gamma G$$
 and
 $$ LF: \cD_{fd}\Gamma \longrightarrow \cD_{fd}\Gamma G.$$
\end{proposition}
\demo{The first restriction is because $F\Gamma = \Gamma G$.  The second comes from the fact that $F$ is exact and sends a module $M$ to $\oplus_{g\in G} M$.
}

\subsection{Actions of $G=\bZ/2\bZ$ on (complete) path dg algebras}
Let $Q$ be a finite graded quiver.  Let $\Gamma_{gr}$ be the graded path algebra of $Q$, and let $\widehat{\Gamma}_{gr}$ be the complete graded path algebra of $Q$.  Let $\Gamma$ be a dg algebra with underlying graded algebra $\Gamma_{gr}$ and differential $d$, and let $\widehat{\Gamma}$ be its completion. 

Let $G$ be a finite group which acts on $Q$ and which preserves degrees.  This action induces an action of $G$ on $\Gamma_{gr}$ and $\widehat{\Gamma}_{gr}$ by automorphisms of graded algebras.  Assume further that this action of $G$ commutes with $d$, so that $G$ acts on the dg algebras $\Gamma$ and $\widehat{\Gamma}$.

Assume now that $G=\{1, \sigma\}$.  As in Section \ref{sect::cyclic}, we partition $Q_0$ into subsets $Q_0 = V \coprod W$.

Then

\begin{proposition}
The dg algebra $\bar e \Gamma G \bar e$ is isomorphic to the dg algebra whose underlying graded graded algebra is the graded path algebra of the quiver described in Section \ref{sect::cyclic}, with differential $d$ such that $d(i^\pm)=(d(i)^\pm)$ and $d(\alpha^{\pm}) = (d(\alpha))^\pm$. 

Its completion is isomorphic to $\bar e \widehat{\Gamma} G \bar e$.
\end{proposition}

\subsection{The case of Ginzburg dg algebras}\label{subs::caseGinzburg}
We follow \cite{G06} and \cite{Ami09}.  Let $(Q,S)$ be a quiver with potential.  We define its complete Ginzburg dg algebra $\widehat{\Gamma}=\widehat{\Gamma}_{Q,S}$ as follows.

Let $\bar Q$ be the graded quiver whose vertices set is that of $Q$ and whose arrows set contains
\begin{itemize}
 \item for every arrow $\alpha:i\to j$ in $Q$, an arrow $\alpha:i\to j$ of degree $0$;
 \item for every arrow $\alpha:i\to j$ in $Q$, an arrow $\bar\alpha:j\to i$ of degree $-1$; and
 \item for every vertex $i$ of $Q$, a loop $t_i:i\to i$ of degree $-2$.
\end{itemize}

Then, as a graded algebra, $\widehat{\Gamma}$ is the complete path algebra of $\bar Q$, that is, for every integer $m$,
\[
 \widehat{\Gamma}^m = \prod_{w \textrm{ path of degree $m$}} kw.
\]
The differential of $\widehat{\Gamma}$ is the continuous map defined as follows on arrows, and extended by linearity and the Leibnitz rule:
for any arrow $\alpha$ of $Q$, $d(\alpha) = 0$ and $d(\bar\alpha) = \partial_\alpha S$,
and for any vertex $i$ of $Q$, $d(t_i) = e_i\big( \sum_{\alpha\in Q_1} (\alpha\bar\alpha - \bar\alpha\alpha) \big) e_i$.

Assume that $G=\{1,\sigma\}$ acts on $Q$, and that this action is such that $S$ and $\sigma(S)$ are cyclically equivalent.  This implies that for any arrow $a$ of $Q$, the cyclic derivative $\partial_{\sigma(a)}(S)= \sigma(\partial_a(S))$.

Before we go on, we need a version of the application $\iota:\Lambda \to \bar e \Lambda G \bar e$ which is well-defined on potentials, that is,
which is invariant under cyclic permutations. For any cyclic path $w$, we define
\[
 \iota'(w) = \begin{cases}
              \iota(w) & \textrm{if $w$ has its starting and ending points in $V$}, \\
              2\iota(w) & \textrm{if $w$ has its starting and ending points in $W$}.
             \end{cases}
\]
Thanks to Lemma \ref{lemm::iota-of-a-path}, $\iota'$ is well-defined on cyclic paths up to cyclic equivalence;
it also extends naturally to linear combinations of path.

\begin{theorem}\label{prop::action-on-Ginzburg}
  \begin{enumerate}
    \item\label{dg-action} The action of $G$ on $(Q,S)$ induces an action on $\widehat{\Gamma}$ by dg automorphisms.
    
    \item\label{is-Ginzburg} The dg algebra $\bar e \widehat{\Gamma}G \bar e$ is isomorphic to the complete Ginzburg dg algebra of the quiver with potential $(Q_G,S_G)$, where
      \begin{itemize}
        \item $Q_G$ is the quiver described in Section \ref{sect::cyclic};
        \item $S_G = \iota'(S)$ in $\bar e \widehat{\Gamma}G \bar e$.
      \end{itemize}
  \end{enumerate}
\end{theorem}
\demo{ Point (\ref{dg-action}) is proved by observing that the action of $G$ on $(Q,S)$ extends to an action on the dg algebra $\Gamma$, and then by completing to get an action on $\widehat{\Gamma}$.

To prove point (\ref{is-Ginzburg}), first note that the dg algebra $\bar e \widehat{\Gamma}G \bar e$ has an underlying graded algebra isomorphic to that of the Ginzburg dg algebra of $(Q_G,S_G)$.  Thus we only need to check that the differentials coincide.  We do this by computing the action of the differentials on the arrows in both cases.

The arrows of degree $0$ are all sent to $0$ in both cases.

\bigskip

To deal with the arrows of degree $-1$, we treat four cases.  Let $\alpha:i\to j$ be an arrow in $Q$.  
We can identify $(\bar\alpha)^\pm$ with $\overline{\alpha^\pm}$.
We need to compare $d_G((\bar\alpha)^\pm)$ and $\partial_{\alpha^\pm} ( \iota'(S))$.

\medskip

\emph{Case 1: $i,j\in V$.}  Let $w=\alpha_1\cdots\alpha_m$ be a term in $S$ which involves $\alpha$.  
Up to cyclic permutation, we can assume that $w$ starts and ends in a vertex in $V$.

We have that
\begin{eqnarray*}
 d_G((\bar\alpha)^\pm)
  &=& d_G(e_i^\pm \iota(\bar\alpha) e_j^\pm) \\
  &=& e_i^\pm d_G(\alpha\otimes 1) e_j^\pm \\
  &=& e_i^\pm (d(\alpha) \otimes 1) e_j^\pm \\
  &=& e_i^\pm (\partial_{\alpha}S \otimes 1) e_j^\pm \\
  &=& e_i^\pm \iota(\partial_{\alpha}S) e_j^\pm. \\
\end{eqnarray*}
Moreover, $\partial_{\alpha^\pm}(\iota'w) = e_i^\pm \iota(\partial_{\alpha}w) e_j^\pm$.
Summing over all terms $w$ of $S$, this shows that $d_G((\bar\alpha)^\pm) = \partial_{\alpha^\pm}(\iota'(S))$.

\medskip
\emph{Case 2: $i\in V, j\in W$.}  Let $w$ be as in Case 1.  

We can show as in Case 1 that $d_G((\bar \alpha)^\pm) = e_i^\pm(\iota(\partial_\alpha(S)))e_j^+$.

To compute $\partial_{\alpha^\pm}(\iota'w)$, note that $\alpha^\pm$ appears in $\iota'(w)$ whenever $\alpha$ or $\sigma(\alpha)$ appears in $w$;
this is because $\iota(\alpha) = (\alpha^+ + \alpha^-)$ and $\iota(\sigma(\alpha)) = (\alpha^+ - \alpha^-)$.  Hence
\begin{eqnarray*}
 \partial_{\alpha^\pm}(\iota'w)
   &=& 2 e_i^\pm (\iota(\partial_\alpha w)) e_j^+ \pm 2 e_i^\pm (\iota(\partial_{\sigma\alpha} w)) e_j^+,
\end{eqnarray*}
where the factor $2$ appears because of the definition of $\iota'$.  According to Lemma \ref{lemma::iota(sigma w)}, 
\[
 \iota(\partial_{\sigma\alpha} w) = \iota(\sigma(\partial_\alpha \sigma w)) = e_i^+ \iota(\partial_\alpha \sigma w) - e_i^- \iota(\partial_\alpha \sigma w),
\]
Hence $\partial_{\alpha^\pm}(\iota'w) = 2 e_i^\pm \big( \iota(\partial_\alpha w) + \iota(\partial_\alpha \sigma w) \big) e_j^+ $.

Since $S$ is $\sigma$-invariant, $\sigma w$ is also a term in $S$.  This shows that $d_G((\bar \alpha)^\pm) = \frac{1}{4} \partial_{\alpha^\pm}(\iota'(S))$.

\medskip
\emph{Case 3: $i\in W, j\in V$.} This case is dual to Case 2, and we omit it.

\medskip
\emph{Case 4: $i,j\in W$.}  Let $w$ be as before.  
As in the previous cases, we can show that $d_G((\bar \alpha)^+) = e_i^+(\iota(\partial_\alpha(S)))e_j^+$.

To compute $\partial_{\alpha^+}(\iota'w)$, note that $\alpha^+$ appears in $\iota'(w)$ whenever $\alpha$ or $\sigma(\alpha)$ appear in $w$.
Then
\begin{eqnarray*}
 \partial_{\alpha^+}(\iota'w)
   &=& 4 e_i^+ (\iota(\partial_\alpha w)) e_j^+ + 4 e_i^+ (\iota(\partial_{\sigma\alpha} w)) e_j^+,
\end{eqnarray*}
where the factor $4$ appears because of the factor $2$ in the definition of $\iota'$.
Again using Lemma \ref{lemma::iota(sigma w)}, we have that
\[
 \iota(\partial_{\sigma\alpha} w) = \iota(\sigma(\partial_\alpha \sigma w)) =  \iota(\partial_\alpha \sigma w),
\]
so that $\partial_{\alpha^+}(\iota'w) = 8e_i^+ (\iota(\partial_\alpha w)) e_j^+$.

Since $S$ is $\sigma$-invariant, the term $\sigma(w)$ also appears in it.  
Therefore $d_G((\bar \alpha)^+) = \frac{1}{8} \partial_{\alpha^+}(\iota'(S))$.

\bigskip

Finally, we deal with arrows of degree $-2$.  Let $i$ be a vertex of $Q$.

\medskip

\emph{Case a: $i\in V$.}  The following computation is sufficient:
\begin{eqnarray*}
 d_G(t_i^\pm) 
   &=& d_G(e_i^\pm t_i^\pm e_i^\pm) \\
   &=& e_i^\pm d_G(t_i \otimes 1) e_i^\pm \\
   &=& e_i^\pm \iota(d(t_i)) e_i^\pm \\
   &=& e_i^\pm \iota(e_i \sum_{\alpha\in Q_1} (\alpha \bar\alpha - \bar \alpha \alpha)e_i) e_i^\pm \\
   &=& e_i^\pm \big( \sum_{\substack{\alpha \in Q_1 \\ (s(\alpha), t(\alpha)) \in V^2}} (\alpha^\pm \overline{\alpha}^\pm - \overline{\alpha}^\pm \alpha^\pm ) +
                     2 \sum_{\substack{\alpha^\pm \in Q_1 \\ (s(\alpha), t(\alpha)) \notin V^2}} (\alpha^\pm \overline{\alpha}^\pm - \overline{\alpha}^\pm \alpha^\pm ) \big) e_i^\pm \\
   &=& e_i^\pm \big( \sum_{\substack{\alpha^\pm \in (Q_G)_1 \\ (s(\alpha), t(\alpha)) \in V^2}} (\alpha^\pm \overline{\alpha}^\pm - \overline{\alpha}^\pm \alpha^\pm ) +
                     4 \sum_{\substack{\alpha^\pm \in (Q_G)_1 \\ (s(\alpha), t(\alpha)) \notin V^2}} (\alpha^\pm \overline{\alpha}^\pm - \overline{\alpha}^\pm \alpha^\pm ) \big) e_i^\pm.
\end{eqnarray*}

\medskip
\emph{Case b: $i\in W$.}  In this case, a similar computation yields
\begin{eqnarray*}
 d_G(t_i^\pm) 
   &=& e_i^+ \Big(   2\sum_{\substack{\alpha^\pm \in (Q_G)_1 \\ (s(\alpha), t(\alpha)) \in W^2}} (\alpha^+ \overline{\alpha}^+ - \overline{\alpha}^+ \alpha^+ ) \\
    &&  \quad       + \sum_{\substack{\alpha^\pm \in (Q_G)_1 \\ (s(\alpha), t(\alpha)) \notin W^2}} (\alpha^+ \overline{\alpha}^+ - \overline{\alpha}^+ \alpha^+ ) \\
    &&  \quad        + \sum_{\substack{\alpha^\pm \in (Q_G)_1 \\ (s(\alpha), t(\alpha)) \notin W^2}} (\alpha^- \overline{\alpha}^- - \overline{\alpha}^- \alpha^- ) \Big) e_i^+.
\end{eqnarray*}

\bigskip
To finish the proof, define an automorphism $\zeta$ of the path algebra  $k\bar{Q}_G$ as follows. 
Firstly, $\zeta$ is the identity on vertices.  Secondly, for every arrow $\alpha$ of $Q$, $\zeta$ sends $\alpha$ to itself.  Thirdly, 
\[
 \zeta(\overline{\alpha}^\pm) = \begin{cases}
                              \overline{\alpha}^\pm & \textrm{if $(s(\alpha), t(\alpha)) \in V\times V$,} \\
                              4\overline{\alpha}^\pm & \textrm{if $(s(\alpha), t(\alpha)) \in (V\times W) \cup (W\times V)$,} \\
                              8\overline{\alpha}^\pm & \textrm{if $(s(\alpha), t(\alpha)) \in W\times W$.}
                            \end{cases}
\]
Finally, for every vertex $i$ of $Q$, 
\[
 \zeta(t_i^\pm) = \begin{cases}
                    t_i^\pm & \textrm{if $i\in V$}, \\
                    4 t_i^\pm & \textrm{if $i \in W$}.
                  \end{cases}
\]
Then $\zeta(\Gamma_{Q_G,S_G})$ is isomorphic to $\bar{e}\hat{\Gamma}G\bar{e}$ and we get the result.

%
%
%
%

}

\begin{corollary}
  The algebra $\bar e (\cP(Q,S)G) \bar e$ is a Jacobian algebra.
\end{corollary}
\demo{It is proved in \cite[Lemma 2.8]{KY11} that $H^0 \widehat{\Gamma}$ is isomorphic to $\cP(Q,S)$.
The result is then an application of Theorem \ref{prop::action-on-Ginzburg} and Corollary \ref{coro::action-on-H0}.
}

\begin{corollary}\label{cor::functors cluster}
  The functors of Proposition \ref{prop::derived-functors} induce functors $$F:\cC(\widehat{\Gamma}_{Q,S})\to \cC(\widehat{\Gamma}_{Q_G, S_G})\quad \textrm{and}\quad F':\cC(\widehat{\Gamma}_{Q_G,S_G})\to \cC(\widehat{\Gamma}_{Q, S})$$ between generalized cluster categories.
\end{corollary}

\section{Group action on cluster categories associated with surfaces}\label{section2}

We follow \cite[Section 2]{FST}. In the rest of the paper $\Sigma$ is a closed connected oriented surface with non empty boundary. Let $\cM$ be a finite set of marked points on the boundary of $\Sigma$ such that there is at least one marked point on each boundary component of $\Sigma$. Let $\cP$ be a finite set of marked points in the interior of $\Sigma$, called punctures.   We assume that:
\begin{itemize}
\item the set of punctures $\cP$ is non empty,
\item $(\Sigma,\cM,\cP)$ is not a once-punctured monogon.
\end{itemize}
We denote by $g$ the genus of $\Sigma$, by $b$ the number of boundary components and by $p$ the number of punctures.

The aim in this section is to construct a new marked surface $(\widetilde{\Sigma},\widetilde{\cM})$ without punctures together with triangle functors between the associated cluster categories using group actions and the results of Section \ref{section1}.

\subsection{$\bZ/2\bZ$-action on the quiver of a triangulation}\label{subsectionQ(tau)G}

Let $\tau$ be an ideal triangulation of $(\Sigma,\cM,\cP)$ of $\Sigma$ (in the sense of \cite[Def 2.6]{FST}) such that each puncture belongs to a self-folded triangle and such that no triangle shares a side with two self-folded triangles. (These kinds of triangulations are called skewed-gentle by \cite{GLFS} and are also considered in \cite{QZ}).

Then there are exaclty six different types of triangles in $\tau$ which are not self-folded (boundary segments are here colored in gray):

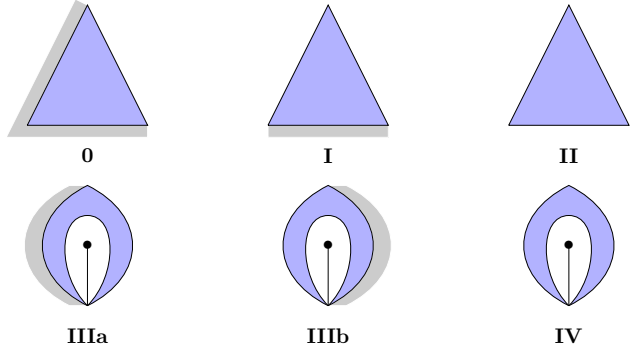
\[\scalebox{0.8}{
\begin{tikzpicture}[>=stealth,scale=1]

\draw[white,fill=gray!40] (0,0)--(2,0)--(2,-0.2)--(-0.35,-0.2)--(0.8,2.1)--(1,2)--(0,0);
\draw[fill=blue!30] (0,0)--(2,0)--(1,2)--(0,0);

\node at  (1,-0.5) {$\mathbf 0$};

\draw[white,fill=gray!40] (4,0)--(6,0)--(6,-0.2)--(4,-0.2)--(4,0);
\draw[fill=blue!30] (4,0)--(6,0)--(5,2)--(4,0);

\node at (5,-0.5) {$\mathbf I$};

\draw[fill=blue!30] (8,0)--(10,0)--(9,2)--(8,0);
\node at (9,-0.5) {$\mathbf{II}$};

\draw[white,fill=gray!40] (1,-3)..controls (0,-2.5) and (0,-1.5).. (1,-1)--(0.7,-1)..controls (-0.3,-1.5) and (-0.3,-2.5).. (0.7,-3)--(1,-3);
\draw[fill=blue!30] (1,-3)..controls (0,-2.5) and (0,-1.5).. (1,-1)..controls  (2,-1.5) and (2,-2.5)..(1,-3); ; 
\draw[fill=white] (1,-3)..controls (0.5,-2.5) and (0.5,-1.5)..(1,-1.5).. controls (1.5,-1.5) and (1.5,-2.5).. (1,-3);
\node at (1,-2) {$\bullet$};
\draw (1,-2)--(1,-3);
\node at (1,-3.5) {$\mathbf{IIIa}$};

\begin{scope}[xshift=4cm]
\draw[white,fill=gray!40] (1,-3)..controls (2,-2.5) and (2,-1.5).. (1,-1)--(1.3,-1)..controls (2.3,-1.5) and (2.3,-2.5).. (1.3,-3)--(1,-3);
\draw[fill=blue!30] (1,-3)..controls (0,-2.5) and (0,-1.5).. (1,-1)..controls  (2,-1.5) and (2,-2.5)..(1,-3); ; 
\draw[fill=white] (1,-3)..controls (0.5,-2.5) and (0.5,-1.5)..(1,-1.5).. controls (1.5,-1.5) and (1.5,-2.5).. (1,-3);
\node at (1,-2) {$\bullet$};
\draw (1,-2)--(1,-3);
\node at (1,-3.5) {$\mathbf{IIIb}$};
\end{scope}
\begin{scope}[xshift=8cm]

\draw[fill=blue!30] (1,-3)..controls (0,-2.5) and (0,-1.5).. (1,-1)..controls  (2,-1.5) and (2,-2.5)..(1,-3); ; 
\draw[fill=white] (1,-3)..controls (0.5,-2.5) and (0.5,-1.5)..(1,-1.5).. controls (1.5,-1.5) and (1.5,-2.5).. (1,-3);
\node at (1,-2) {$\bullet$};
\draw (1,-2)--(1,-3);
\node at (1,-3.5) {$\mathbf{IV}$};
\end{scope}

\end{tikzpicture}}\]

Therefore the adjacency quiver $Q(\tau)$  (as defined in \cite{FST}) is built by gluing blocks corresponding to each kind of triangle. 

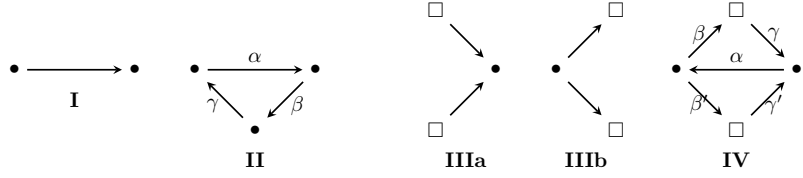
\[\scalebox{0.8}{
\begin{tikzpicture}[>=stealth,scale=1]
\node (I1) at (0,0) {$\bullet$};
\node (J1) at (2,0) {$\bullet$}; 
\draw[thick, ->] (I1)--(J1);
\node at (1,-0.5) {$\mathbf{I}$};

\begin{scope}[xshift=3cm, yshift=-1cm]
\node (I1) at (0,1) {$\bullet$};
\node (J1) at (2,1) {$\bullet$}; 
\node (K1) at (1,0) {$\bullet$};
\draw[thick, ->] (I1)--node [yshift=2mm]{$\alpha$}(J1);\draw[thick, ->] (J1)--node [xshift=2mm, yshift=-1mm]{$\beta$}(K1);\draw[thick, ->] (K1)--node [xshift=-2mm,yshift=-1mm]{$\gamma$}(I1);
\node at (1,-0.5) {$\mathbf{II}$};
\end{scope}

\begin{scope}[xshift=7cm]
\node (I1) at (0,1) {$\square$};
\node (I2) at (0,-1) {$\square$};
\node (J) at (1,0) {$\bullet$};
\draw[thick, ->] (I1)--(J);\draw[thick, ->] (I2)--(J);

\node at (0.5,-1.5) {$\mathbf{IIIa}$};

\end{scope}

\begin{scope}[xshift=8cm]
\node (I1) at (2,1) {$\square$};
\node (I2) at (2,-1) {$\square$};
\node (J) at (1,0) {$\bullet$};
\draw[thick, ->] (J)--(I1);\draw[thick, ->] (J)--(I2);

\node at (1.5,-1.5) {$\mathbf{IIIb}$};

\end{scope}

\begin{scope}[xshift=12cm]
\node (I1) at (0,1) {$\square$};
\node (I2) at (0,-1) {$\square$};
\node (J) at (1,0) {$\bullet$};
\node (K) at (-1,0) {$\bullet$};
\draw[thick, ->] (I1)--node[yshift=1mm,xshift=1mm]{$\gamma$}(J);\draw[thick, ->] (I2)--node[yshift=-1mm,xshift=1mm]{$\gamma'$}(J);\draw[thick, ->] (K)--node[yshift=1mm,xshift=-1mm]{$\beta$}(I1);\draw[thick, ->] (K)--node[yshift=-1mm,xshift=-1mm]{$\beta'$}(I2);\draw[thick, ->](J)--node[yshift=2mm]{$\alpha$}(K);
\node at (0,-1.5) {$\mathbf{IV}$};

\end{scope}

\end{tikzpicture}}\]

Note that two blocks can only be glued by identifying two vertices of type $\bullet$ and that one block cannot be glued to itself (see \cite[section 13]{FST} for more details on block decompositions). 

The potential $S(\tau)$ defined in \cite{LF08} associated to $\tau$ is then 

\[ S(\tau)=\sum_{\textrm{blocks of type II}}\gamma\beta\alpha+\sum_{\textrm{blocks of type IV}}(\gamma\beta\alpha+\gamma'\beta'\alpha).\]

We consider the action of $G=\bZ/2\bZ$ on $Q(\tau)$ as the unique one exchanging vertices $i$ and $i'$ in the blocks of type IIIa, IIIb and IV. The potential $S(\tau)$ is clearly $G$-invariant. Applying Theorem \ref{prop::action-on-Ginzburg} we get that $Q(\tau)_G$ is obtained from $Q(\tau)$ by replacing:

\begin{itemize}
\item each block of type I by two blocks  \[\scalebox{0.8}{
\begin{tikzpicture}[>=stealth,scale=1]
\node (I1) at (0,0) {$i^+$};
\node (J1) at (2,0) {$j^+$}; 
\draw[thick, ->] (I1)--node[yshift=2mm]{$\alpha^+$}(J1);

\begin{scope}[xshift=3cm]
\node (I1) at (0,0) {$i^-$};
\node (J1) at (2,0) {$j^-$}; 
\draw[thick, ->] (I1)--node[yshift=2mm]{$\alpha^-$}(J1);
\end{scope}

\end{tikzpicture}}\]

 \item each block of type II by two blocks
 \[\scalebox{0.8}{
\begin{tikzpicture}[>=stealth,scale=1]

\node (I1) at (0,1) {$i^+$};
\node (J1) at (2,1) {$j^+$}; 
\node (K1) at (1,0) {$k^+$};
\draw[thick, ->] (I1)--node [yshift=2mm]{$\alpha^+$}(J1);\draw[thick, ->] (J1)--node [xshift=2mm, yshift=-1mm]{$\beta^+$}(K1);\draw[thick, ->] (K1)--node [xshift=-2mm,yshift=-1mm]{$\gamma^+$}(I1);

\begin{scope}[xshift=3cm]
\node (I1) at (0,1) {$i^-$};
\node (J1) at (2,1) {$j^-$}; 
\node (K1) at (1,0) {$k^-$};
\draw[thick, ->] (I1)--node [yshift=2mm]{$\alpha^-$}(J1);\draw[thick, ->] (J1)--node [xshift=2mm, yshift=-1mm]{$\beta^-$}(K1);\draw[thick, ->] (K1)--node [xshift=-2mm,yshift=-1mm]{$\gamma^-$}(I1);
\end{scope}
\end{tikzpicture}}\]
 \item each block of type IIIa (resp. IIIb) by a block (resp. a block)
  \[\scalebox{0.8}{
\begin{tikzpicture}[>=stealth,scale=1]

\node (I1) at (0,1) {$i^+$};
\node (I2) at (0,-1) {$i^-$};
\node (J) at (1,0) {$j$};
\draw[thick, ->] (I1)--node[xshift=2mm,yshift=2mm]{$\alpha^+$}(J);\draw[thick, ->] (I2)--node[xshift=2mm,yshift=-1mm]{$\alpha^-$}(J);

\node at (-1,0) {\textrm{resp.}};

\begin{scope}[xshift=-5cm]
\node (I1) at (2,1) {$j^+$};
\node (I2) at (2,-1) {$j^-$};
\node (J) at (1,0) {$i$};
\draw[thick, ->] (J)--node[xshift=-2mm,yshift=1mm]{$\alpha^+$}(I1);\draw[thick, ->] (J)--node[xshift=-2mm,yshift=-1mm]{$\alpha^-$}(I2);

\end{scope}

\end{tikzpicture}}\]
 \item and each block of type IV by a block 
 
   \[\scalebox{0.8}{
\begin{tikzpicture}[>=stealth,scale=1]
\node (I+) at (0,1) {$i^+$};
\node (J) at (1,0) {$j$};
\node (I-) at (0,-1) {$i^-$};
\node (K+) at (2,1) {$k^+$};
\node (K-) at (2,-1) {$k^-.$};

\draw[thick, ->](I+)--node[xshift=-2mm,yshift=-1mm]{$\beta^+$}(J);
\draw[thick, ->](J)--node[xshift=2mm,yshift=-1mm]{$\gamma^+$}(K+);
\draw[thick, ->](K+)--node[yshift=2mm]{$\alpha^+$}(I+);
\draw[thick, ->](I-)--node[xshift=-2mm,yshift=1mm]{$\beta^-$}(J);
\draw[thick, ->](J)--node[xshift=2mm,yshift=1mm]{$\gamma^-$}(K-);
\draw[thick, ->](K-)--node[yshift=-2mm]{$\alpha^-$}(I-);

\end{tikzpicture}}\]

\end{itemize}

We obtain the following description for the potential.

\begin{proposition}\label{prop_potential}
The potential $S(\tau)_G$ defined in Theorem \ref{prop::action-on-Ginzburg} is
$$S(\tau)_G=\sum_{\textrm{blocks of type II}}(\gamma^+\beta^+\alpha^+ + \gamma^-\beta^-\alpha^-) +4 \sum_{\textrm{blocks of type IV}}(\gamma^+\beta^+\alpha^+ +\gamma^-\beta^-\alpha^-).$$
\end{proposition}

\begin{proof} By Theorem \ref{prop::action-on-Ginzburg} the potential $S(\tau)_G$ is defined to be 
$S(\tau)_G=\iota' (S(\tau))$, where $\iota'$ is as in Section \ref{subs::caseGinzburg}. 

Now if $\gamma\beta\alpha$ is a $3$-cycle corresponding to a block of type III in $Q(\tau)$ we compute
\begin{eqnarray*}\iota'(\gamma\beta\alpha) & = & \iota(\gamma\beta\alpha) \\
  &=& (\gamma^++\gamma^-)(\beta^++\beta^-)(\alpha^++\alpha^-)\\ & =& \gamma^+\beta^+\alpha^++\gamma^-\beta^-\alpha^-\end{eqnarray*}
since the arrows $\gamma^{\pm}$ and $\beta^{\mp}$  do not compose.

If $\gamma\beta\alpha+\gamma'\beta'\alpha$ is the potential associated with a block of type IV then we compute

\begin{eqnarray*}\iota'(\gamma\beta\alpha+\gamma'\beta'\alpha) & = & 2\iota(\gamma\beta\alpha+\gamma'\beta'\alpha) \\
  &=& 2\Big( (\gamma^++\gamma^-)(\beta^++\beta^-)(\alpha^++\alpha^-) +(\gamma^++\gamma^-)(\beta^+-\beta^-)(\alpha^+-\alpha^-) \Big)  \\
  & =& 2\Big( \gamma^+\beta^+\alpha^+ + \gamma^+\beta^+\alpha^- + \gamma^-\beta^-\alpha^+  + \gamma^-\beta^-\alpha^-+  \\ 
  &&  \hspace{3cm} \gamma^+\beta^+\alpha^+ - \gamma^+\beta^+\alpha^- -\gamma^-\beta^-\alpha^+ + \gamma^-\beta^-\alpha^- \Big) \\ 
  & = & 4\gamma^+\beta^+\alpha^++ 4 \gamma^-\beta^-\alpha^-.\end{eqnarray*}
This finishes the proof.
\end{proof}

\subsection{Gluing a Riemann surface along boundary segments}\label{subsection_gluing}

In this subsection we collect topological basic ingredients that will be useful in the construction of the new surface $\widetilde{\Sigma}$ associated to $\Sigma$ and $\tau$.

\medskip

Let $\Sigma'$ be an oriented Riemann surface with non empty boundary. Let $I^+ = [A^+,B^+]$ and $I^-=[A^-,B^-]$ be disjoint segments on the boundary of $\Sigma'$. The orientation of $\Sigma'$ induces an orientation on $I^+$ and $I^-$ and let $\varphi^\pm$ be a homemorphism $\varphi^\pm:I^\pm \to [0,1]$ respecting the orientation. Let $I:[0,1]\to [0,1]$ the homeomorphism $t\mapsto (1-t)$. Consider the following composition:
\[\xymatrix{\Psi: I^+\ar[r]^-{\varphi^+} &  [0,1]\ar[r]^{I} & [0,1]\ar[r]^-{(\varphi^-)^{-1}} & I^-} \] 
We consider the quotient $\Sigma'':=\Sigma'/\Psi$ made from $\Sigma'$ by identifying the segments $I^+$ and $I^-$ via the map $\Psi$, and denote by $\pi:\Sigma'\to\Sigma''$ the natural projection map. 
We say that a subset $U$ of $\Sigma'$ satisfies the \emph{gluing condition} if 
\[ \Psi(U\cap I^+)=U\cap I^-.\] 
   
\begin{proposition}\label{proposition_gluing} The quotient $\Sigma'':=\Sigma'/\Psi$ is an oriented Riemann surface with non empty boundary. Moreover the open sets of $\Sigma''$ are of the form $\pi(U)$ where $U$ is an open set of $\Sigma'$ satifying the gluing condition.
\end{proposition}

\begin{proof}
By definition of the topology of the quotient $\Sigma''$, if $U$ is an open set in $\Sigma'$ then $\pi(U)$ is open if and only if $\pi^{-1}(\pi(U))=U$. Now for $x$ a point in  $\Sigma'$, it is clear from the definition of $\Sigma''$ that $\pi^{-1}(\pi(x))=\{ x\}$ if and only if $x$ is not in $I^+\cup I^-$. Moreover if $x\in I^+$ then we have  $\pi^{-1}(\pi(x^+))=\{x, \Psi(x)\}$. Thus we get the following equality for any $U$ open set of $\Sigma'$:
\[\pi^{-1}(\pi(U))=U\cup\Psi(U\cap I^+)\cup \Psi^{-1}(U\cap I^-).\]
Then $\pi(U)$ is open in $\Sigma'$ if and only if $U$ satisfies the gluing condition.

\medskip

Now we prove that $\Sigma''$ has a structure of Riemann surface. Denote by $\mathbb D$ the open unitary disc in $\mathbb C$  and by $\mathbb D^+$ the half disc $\mathbb D\cap \{z\in \bC | {\rm Im}(z)\geq 0\}$. Since $\Sigma'$ is a Riemann surface, for each $x$ in the interior (resp. on the boundary) of $\Sigma'$ there exists a basis $\cB^x=(V)$ of neighborhoods of $x$ together with homeomorphisms $\omega:V\to \mathbb D$ (resp. $\omega:V\to \mathbb D^+$) with holomorphic transition maps. 

Now for $x\in \Sigma'$ let us describe a basis $\cB=(V)$ of neighborhoods of $\pi(x)$ together with homeomorphisms $\omega:V\to \mathbb D\  (\ \textrm{or } \mathbb D^+)$.

Assume first that $x$ is not in $I^+\cup I^-$. The point $\pi(x)$ is on the boundary of $\Sigma''$ if and only if $x$ is on the boundary of $\Sigma'$. Then there exists a basis $\cB=(V)\subset \cB^x$ of neighborhoods of $x$ such that for any $V\in \cB$, the open $V$ does not intersect $I^+\cup I^-$. Then the restriction of $\pi$ on $V$ is the identity, so $(\pi(V))_{V\in \cB}$ is a basis of neighborhoods of $\pi(x)$ in $\Sigma''$, and $\omega\circ \pi^{-1}:\pi(V)\to \mathbb D$ (resp.  $\mathbb D^+$ if $x\in \partial \Sigma'$) is an homeomorphism.

If $x^+$ is in the interior of $I^+$ (hence on the boundary of $\Sigma'$). Then we have $\pi^{-1}(\pi(x^+))=\{x^+, x^-\}$ with $x^-=\Psi(x^+)$ and $\pi(x^+)$ is in the interior of $\Sigma''$. There exists a basis $\cB^+\subset \cB^{x^+}$ of neighborhoods of $x^+$ and $\cB^-\subset \cB^{x^-}$ such that: for any $V^\pm\in \cB^\pm$ we have the inclusion  $V^\pm\cap \partial \Sigma'\subset I^\pm$. Let $V^+\in \cB^+$ and $V^-\in \cB^-$ such that $\Psi(V^+\cap I^+)=V^-\cap I^-$. The homeomorphism $\Psi: V^+\cap I^+\to V^-\cap I^-$ induces a homeomorphism $[-1,1]\to[-1,1]$ that allows us to glue $\omega^+(V^+)$ to $-\omega^-(V^-)$ along the diameter into the unitary disc $\mathbb D$ of $\cB$. Hence we obtain homeomorphisms $\omega:\pi(V^+\cup V^-)\to \mathbb D$ and $\pi(V^+\cup V^-)$ is a basis of neighborhood of $\pi(x^+)$ in $\Sigma''$.

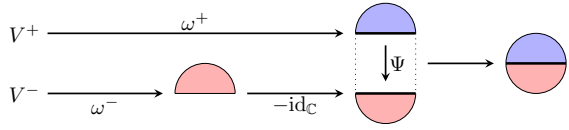
 \[\scalebox{0.8}{
\begin{tikzpicture}[>=stealth,scale=1]
\node (V+) at (0,0) {$V^+$};
\node (V-) at (0,-1) {$V^-$};
\draw[thick, ->] (V+)--node [yshift=2mm]{$\omega^+$}(5.3,0);
\draw[thick, ->] (V-)--node [yshift=-2mm]{$\omega^-$}(2.3,-1);
\draw[fill=blue!30] (6.5,0) arc (0:180:0.5);
\draw[very thick] (5.5,0)--(6.5,0);
\draw[fill=red!30] (3.5,-1) arc (0:180:0.5);
\draw (2.5,-1)--(3.5,-1);
\draw[fill=red!30] (6.5,-1) arc (0:-180:0.5);
\draw[very thick] (5.5,-1)--(6.5,-1);
\draw[thick, ->] (3.7,-1)--node [yshift=-2mm]{$-\textrm{id}_{\bC}$}(5.3,-1);
\draw[dotted] (5.5,0)--(5.5,-1);\draw[dotted] (6.5,0)--(6.5,-1);
\draw[thick, ->] (6,-0.2)--node [xshift=2mm]{$\Psi$} (6,-0.8);

\draw[thick, ->] (6.7,-0.5)--(7.8,-0.5);
\draw[fill=blue!30] (9,-0.5) arc (0:180:0.5);
\draw[fill=red!30] (8,-0.5) arc (180:360:0.5);
\draw[very thick] (9,-0.5)--(8,-0.5);

\end{tikzpicture}}\]

If $x^+$ is an endpoint of $I^+$ say $A^+$. Then we have $\pi^{-1}(\pi(A^+))=\{A^+, B^-\}$ since $\Psi(A^+)=B^-$ and $\pi(A^+)$ is on the boundary of $\Sigma''$. There exists a basis $\cB^+\subset \cB^{A^+}$  of neighborhoods of $A^+$ such that: for any $V^+\in \cB^+$  we have the inclusion $(\omega^+)^{-1}([0,1])\subset I^+$. Similarly there exists a basis $\cB^-\subset\cB^{B^-}$  of neighborhoods of $B^-$ such that: for any $V^-\in \cB^-$  we have the inclusion $-(\omega^-)^{-1}([0,1])\subset I^-$.

 Let $V^+\in \cB^+$ and $V^-\in \cB^-$ such that $\Psi(V^+\cap I^+)=V^-\cap I^-$. Then the homeomorphism $\Psi: V^+\cap I^+\to V^-\cap I^-$ induces a homeomorphism $[0,1]\to[0,1]$ that allows us to glue $\omega^+(V^+)$ to $-\omega^-(V^-)$ along the segment $[0,1]$. Hence we obtain homeomorphism $\omega$ from $\pi(V^+\cup V^-)\to \mathbb D^+$ and $\pi(V^+\cup V^-)$ gives a basis of neighborhood of $\pi(A^+)$ in $\Sigma''$.   
 
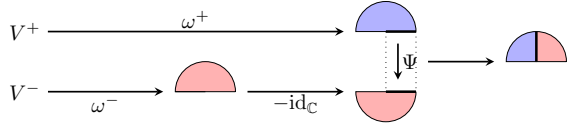
  \[\scalebox{0.8}{
\begin{tikzpicture}[>=stealth,scale=1]
\node (V+) at (0,0) {$V^+$};
\node (V-) at (0,-1) {$V^-$};
\draw[thick, ->] (V+)--node [yshift=2mm]{$\omega^+$}(5.3,0);
\draw[thick, ->] (V-)--node [yshift=-2mm]{$\omega^-$}(2.3,-1);
\draw[fill=blue!30] (6.5,0) arc (0:180:0.5)--(6.5,0);
\draw[very thick] (6,0)--(6.5,0);
\draw[fill=red!30] (3.5,-1) arc (0:180:0.5)--(3.5,-1);
\draw (2.5,-1)--(3,-1);
\draw[fill=red!30] (6.5,-1) arc (0:-180:0.5)--(6.5,-1);
\draw[very thick] (6,-1)--(6.5,-1);
\draw[thick, ->] (3.7,-1)--node [yshift=-2mm]{$-\textrm{id}_{\bC}$}(5.3,-1);
\draw[dotted] (6,0)--(6,-1);\draw[dotted] (6.5,0)--(6.5,-1);
\draw[thick, ->] (6.2,-0.2)--node [xshift=2mm]{$\Psi$} (6.2,-0.8);

\draw[thick, ->] (6.7,-0.5)--(7.8,-0.5);
\draw[fill=red!30] (9,-0.5) arc (0:90:0.5)--(8.5,-0.5)--(9,-0.5);
\draw[fill=blue!30] (8.5,0) arc (90:180:0.5)--(8.5,-0.5);
\draw[very thick] (8.5,0)--(8.5,-0.5);

\end{tikzpicture}}\]

Finally note that the map $\Psi$ respects the orientation since we have the following picture.

  \[\scalebox{0.8}
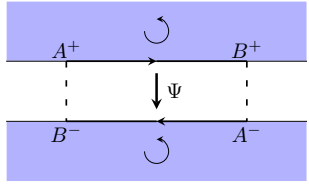
{
\begin{tikzpicture}[>=stealth,scale=1]

\draw[white,fill=blue!30] (-1,0) rectangle (4,1);
\begin{scope}[xshift=1.5cm, yshift=0.5cm,scale=0.2]
\draw[->] (-1,0) arc (-180:90:1);
\end{scope}

\draw[white,fill=blue!30] (-1,-1) rectangle (4,-2);
\begin{scope}[xshift=1.5cm, yshift=-1.5cm,scale=0.2]
\draw[->] (-1,0) arc (-180:90:1);
\end{scope}

\draw (-1,0)--(4,0);
\draw (-1,-1)--(4,-1);

\draw[thick,->] (0,0)--(1.5,0);
\draw[thick] (1.5,0)--(3,0);

\draw[thick] (0,-1)--(1.5,-1);
\draw[thick,<-] (1.5,-1)--(3,-1);
\draw[loosely dashed,thick] (0,0)--(0,-1);
\draw[loosely dashed,thick] (3,0)--(3,-1);

\node at (0,0.2) {$A^+$};
\node at (3,0.2) {$B^+$};
\node at (0,-1.2) {$B^-$};
\node at (3,-1.2) {$A^-$};

\draw[very thick, ->] (1.5,-0.2)--node [xshift=3mm]{$\Psi$} (1.5,-0.8);

\end{tikzpicture}}\]

\end{proof}

\subsection{Construction of a new surface $\widetilde{\Sigma}$}\label{subs::construction}

First we construct a surface $\Sigma^+$ with a triangulation $\tau^+$. For each $P\in \cP$ denote by $i_P$ the folded side of the folded triangle containing $P$. By definition of $\tau$, the other endpoint of $i_P$ is on the boundary of $\Sigma$. The surface $\Sigma^+$ and the triangulation $\tau^+$ are obtained by cutting $\Sigma$ and $\tau$ along all the arcs $i_P$. In other words each block of type IIIa, IIIb and IV is replaced as in the following picture:

  \[\scalebox{0.8}
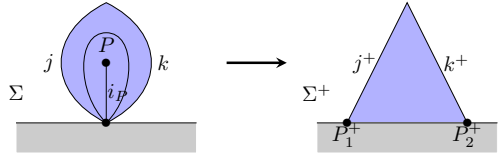
{
\begin{tikzpicture}[>=stealth,scale=1]

\node at (-0.5,-2.5) {$\Sigma$};
\draw[fill=blue!30] (1,-3)..controls (0,-2.5) and (0,-1.5).. node[xshift=-2mm]{$j$}(1,-1)..controls  (2,-1.5) and (2,-2.5) ..node[xshift=2mm]{$k$}(1,-3); 
\draw (1,-3)..controls (0.5,-2.5) and (0.5,-1.5)..(1,-1.5).. controls (1.5,-1.5) and (1.5,-2.5).. (1,-3);
\node at (1,-2) {$\bullet$};
\draw (1,-2)--node[xshift=2mm]{$i_P$}(1,-3);
\node at (1,-1.7) {$P$};
\draw[white, fill=gray!40] (-0.5,-3) rectangle (2.5,-3.5);
\node at (1,-3) {$\bullet$};
\draw (-0.5,-3)--(2.5,-3);

\draw[very thick,->] (3,-2)--(4,-2);

\begin{scope}[xshift=5cm]
\node at (-0.5,-2.5) {$\Sigma^+$};
\draw[fill=blue!30] (0,-3)--node[xshift=-2mm]{$j^+$}(1,-1)--node[xshift=2mm]{$\ \ k^+$}(2,-3)--(0,-3);  
\draw[white, fill=gray!40] (-0.5,-3) rectangle (2.5,-3.5);
\node at (0,-3) {$\bullet$};\node at (2,-3) {$\bullet$};
\draw (-0.5,-3)--(2.5,-3);
\node at (0,-3.2) {$P_1^+$};
\node at (2,-3.2) {$P_2^+$};
\end{scope}

\end{tikzpicture}}\]

If several punctures are linked to the same marked point, since the folded sides are compatible arcs, the folded sides linked to the same marked point can be ordered.  Then we cut $\Sigma$ following the order as shown in the following picture: 

  \begin{equation}\label{severalpunctures}\scalebox{0.8}
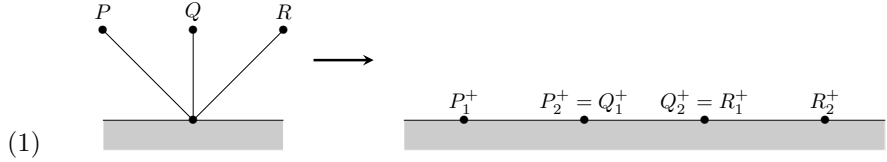
{
\begin{tikzpicture}[>=stealth,scale=1]
\draw[white,fill=gray!40] (0,0) rectangle (3,-0.5);
\node at (0,1.5) {$\bullet$};
\node at (1.5,1.5) {$\bullet$};
\node at (3,1.5) {$\bullet$};
\node at (0,1.8) {$P$};
\node at (1.5,1.8) {$Q$};
\node at (3,1.8) {$R$};
\node at (1.5,0) {$\bullet$};
\draw (0,0)--(3,0);
\draw (0,1.5)--(1.5,0)--(1.5,1.5);
\draw (1.5,0)--(3,1.5);

\draw[very thick,->] (3.5,1)--(4.5,1);

\draw[white, fill=gray!40] (5,0) rectangle (13,-0.5);
\draw (5,0)--(13,0);
\node at (6,0) {$\bullet$};\node at (8,0) {$\bullet$};\node at (10,0) {$\bullet$};\node at (12,0) {$\bullet$};
\node at (6,0.3) {$P^+_1$};\node at (8,0.3) {$P^+_2=Q^+_1$};\node at (10,0.3) {$Q^+_2=R^+_1$};\node at (12,0.3) {$R^+_2$};

\end{tikzpicture}}\end{equation}

Define $\cM^+$ as the set of marked points of $\Sigma$ which are not linked to a puncture, and let $\tau^+$ the union of arcs $j^+$ where $j$ is an arc in $\tau$ which is not a side of a self-folded triangle. The next result is immediate to check.

\begin{lemma} The marked surface $(\Sigma^+,\cM^+\cup\{P_1^+,P_2^+,P\in \cP\})$ is oriented and $\tau^+$ is a triangulation.
\end{lemma}

For each segment $[P_1^+,P_2^+]$, fix a homeomorphism $\varphi_P:[P_1^+,P_2^+]\to [0,1]$ and denote by $P^+:=\varphi_P^{-1}(\frac{1}{2})$. 
Fix another copy $\Sigma^-$ of $\Sigma^+$ and a homeomorphism $S:\Sigma^+\sqcup\Sigma^-\to\Sigma^+\sqcup \Sigma^-$  exchanging $\Sigma^+$ and $\Sigma^-$.   
\begin{definition}
From $\Sigma$ and $\tau$, we define $\widetilde{\Sigma}$ as the quotient $\Sigma^+\sqcup \Sigma^-/(\Psi_P, P\in \cP)$, where $\Psi_P$ is the following composition:
\[\xymatrix{\Psi_P: [P_1^+;P_2^+]\ar[r]^-{\varphi_P} &  [0,1]\ar[r]^{I} & [0,1]\ar[r]^-{\varphi_P^{-1}} & [P_1^+,P_2^+]\ar[r]^{S} & [P_1^-,P_2^-]}. \] 

\end{definition}

The main result of the section is the following.

\begin{theorem}\label{theorem sigma tilde} The quotient $\widetilde{\Sigma}$ is an oriented Riemann surface with non empty boundary. The set $ \widetilde{\tau}=\tau^+\cup\tau^-\cup\{ [P_1^+,P_2^+],P\in \cP\}$ is a triangulation of $(\widetilde{\Sigma},\widetilde{\cM})$ where  
$$\widetilde{\cM}=\cM^+\cup \cM^-\cup \{P_1^\pm,P_2^{\pm}, P\in \cP\}/(\Psi_P,P\in \cP).$$

Moreover the quiver with potential $(Q(\widetilde{\tau}),S(\widetilde{\tau}))$ is (naturally) right equivalent to $(Q(\tau)_G,S(\tau)_G)$. 
\end{theorem}

\begin{proof}
We first prove that $\widetilde{\Sigma}$ is an oriented Riemann surface with non empty boundary.
The quotient $\widetilde{\Sigma}$ is constructed by gluing iteratively segments as in subsection~\ref{subsection_gluing}. By Proposition \ref{proposition_gluing} it is enough to check that at each step, the segments $[P_1^+,P_2^+]$ and $[P_1^-,P^-_2]$ that are glued along $\Psi_P$ do not have any intersection. Denote by $\Sigma'$ the quotient $\Sigma^+\sqcup \Sigma^-/(\Psi_Q,Q\neq P)$. The segment $[P_1^+,P_2^+]$ intersects $\Sigma^-$ in the quotient $\Sigma'$ if and only if there exists $Q\neq P$ in  $\cP$ such that the intersection $[P_1^+,P_2^+]\cap[Q_1^-,Q^-_2]$ is non empty. This situation only occurs if $P$ and $Q$ are linked via self-folded sides to the same marked point on the boundary (cf picture \eqref{severalpunctures}). Then we have either $[P_1^+,P_2^+]\cap[Q_1^+,Q^+_2]=\{P_2^+\}=\{Q_1^+\}$ or  $[P_1^+,P_2^+]\cap[Q_1^+,Q^+_2]=\{P_1^+\}=\{Q_2^+\}$. In the first case we have  $[P_1^+,P_2^+]\cap[Q_1^-,Q^-_2]=\{P_2^+\}=\{Q_2^-\}$ while  we have $[P_1^-,P_2^-]\cap[Q_1^-,Q^-_2]=\{P_2^-\}=\{Q_1^-\}$. These two sets are disjoint since a boundary component of $\Sigma^+$ is not covered by segments $[Q_1^+,Q_2^+]$. The other case is similar. Therefore for any $Q\neq P$ we have $$[P_1^+,P_2^+]\cap[Q_1^+,Q^+_2]\neq [P_1^-,P_2^-]\cap[Q_1^-,Q^-_2].$$ This implies that $[P_1^+,P^+_2]\cap[P^-_1,P^-_2]=\emptyset$ and proves that $\tSigma$ is an oriented Riemann surface.

\medskip

By the previous lemma $\tau^+\cup \tau^-$ is a triangulation of $\Sigma^+\sqcup\Sigma^-$, in this triangulation the segments $[P_1^+,P_2^+]$ and $[P_1^-,P_2^-]$ are boundary segments. Hence when gluing along $\Psi_P$, these segments become internal arcs.

Let $\Delta$ be an internal triangle of $\tau$ which is not self-folded  and not adjacent to a self-folded (that is a block of type II). It corresponds to a $3$-cycle $i\to j\to k\to i$ in $Q(\tau)$.  The triangle $\Delta$ gives rise to two disjoint internal triangles in $\widetilde{\Sigma}$, one in $\Sigma^+$ and one in $\Sigma^-$. These triangles give rise to two disjoint $3$-cycles  $i^\pm\to j^\pm\to k^\pm\to i^\pm$ in $Q(\widetilde{\tau})$. The same hold for triangles of type I. 

If $B$ is a block of type IV that is a self-folded triangle (around $P$) together with the triangle adjacent to it, then it gives rise to an internal triangle whose sides are $[P_1^\pm,P_2^\pm]$, $j^\pm$ and $k^\pm$ in $\Sigma^\pm$. So in $Q(\widetilde{\tau})$ we obtain the following picture:

  \[\scalebox{0.8}
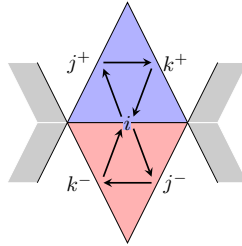
{
\begin{tikzpicture}[>=stealth,scale=1]
\draw[white,fill=gray!40] (0,0)--(-0.5,1)--(-1,1)--(-0.5,0)--(0,0);
\draw[white,fill=gray!40] (2,0)--(2.5,1)--(3,1)--(2.5,0)--(2,0);
\draw (0,0)--(-0.5,1);\draw (2,0)--(2.5,1);
\draw[fill=blue!30](0,0)--node[xshift=-3mm]{$j^+$}(1,2)--node[xshift=3mm]{$k^+$}(2,0)--node[fill=blue!30,inner sep=0pt]{$i$}(0,0);
\draw[thick,->] (0.9,0.1)--(0.6,0.9);\draw[thick,->] (0.6,1)--(1.4,1);\draw[thick,->] (1.4,0.9)--(1.1,0.1);

\begin{scope}[rotate=180,xshift=-2cm]
\draw[white,fill=gray!40] (0,0)--(-0.5,1)--(-1,1)--(-0.5,0)--(0,0);
\draw[white,fill=gray!40] (2,0)--(2.5,1)--(3,1)--(2.5,0)--(2,0);
\draw (0,0)--(-0.5,1);\draw (2,0)--(2.5,1);
\draw[fill=red!30](0,0)--node[xshift=3mm]{$j^-$}(1,2)--node[xshift=-3mm]{$k^-$}(2,0)--node[fill=blue!30,inner sep=0pt]{$i$}(0,0);
\draw[thick,->] (0.9,0.1)--(0.6,0.9);\draw[thick,->] (0.6,1)--(1.4,1);\draw[thick,->] (1.4,0.9)--(1.1,0.1);
\end{scope}

\end{tikzpicture}}\]

The same hold for blocks of types IIIa and IIIb. Therefore by subsection \ref{subsectionQ(tau)G} we conclude that $Q(\widetilde{\tau})=Q(\tau)_G$. The right equivalence for the associated potential directly follows from Proposition \ref{prop_potential}.

\end{proof}

By Corollary \ref{coro::G-dual-action} we obtain a natural action of $\bZ/2\bZ$ on $(Q(\ttau),S(\ttau))$. This is the unique action fixing the vertices of $Q(\ttau)$ corresponding to `defolded sides' and exchanging vertices $i^+$ and $i^-$. We denote both by $\sigma$ this automophism of $Q(\ttau)$ and the corresponding automorphism of $Q(\tau)$ (the one exchanging $i$ and $i'$). 

Denote by $\mathcal{C}_{\tau}$ (resp. $\mathcal{C}_{\widetilde{\tau}}$) the cluster category associated with the quiver with potential $(Q(\tau),S(\tau))$ (resp. $(Q(\widetilde{\tau}),S(\widetilde{\tau}))$).  Then a direct consequence of Theorem~\ref{theorem sigma tilde} together with Corollary \ref{cor::functors cluster} is the following.

\begin{corollary}\label{corollary functor cluster categories}

There exist triangle functors $F:\mathcal{C}_{\widetilde{\tau}}\to \cC_{\tau}$ and $F': \cC_{\tau}\to \cC_{\widetilde{\tau}}$ commuting with the action of $\sigma$ and satisfying the following properties:

\begin{enumerate}
\item For any object $X$ in $\cC_{\widetilde{\tau}}$ (resp. $\cC_{\tau}$) we have $F'\circ F(X)\simeq X\oplus X^{\sigma}$ (resp. $F\circ F'(X)\simeq X\oplus X^\sigma$).
\item If $X$ is an indecomposable object of $\cC_{\widetilde{\tau}}$ (resp. $\cC_{\tau}$) such that $X^\sigma \not\simeq X$, then $FX$ (resp. $F'X$) is an indecomposable object of $\cC_\tau$ (resp. $\cC_{\widetilde{\tau}}$) and $FX\simeq F(X^\sigma)$ (resp. $F'X\simeq F'(X^\sigma)$). 
\item If $X$ is an indecomposable object of $\cC_{\widetilde{\tau}}$ (resp. $\cC_{\tau}$) such that $X^\sigma \simeq X$, then there exists an indecomposable object $Y$ in $\cC_\tau$ (resp. $\cC_{\widetilde{\tau}}$) such that $FX\simeq Y\oplus Y^\sigma$ (resp. $F'X=Y\oplus Y^\sigma$).
 
\end{enumerate}

\end{corollary}

\begin{remark}
If $\tau$ and $\tau'$ are different triangulation of $(\Sigma, \cP,\cM)$ there exists a triangle equivalence between the corresponding cluster categories $\cC_{\tau}$ and $\cC_{\tau'}$, but this equivalence is not canonical (see \cite[Appendix]{CS16}) making dangerous the writing of $\cC_{\Sigma}$ instead of $\cC_{\tau}$. Moreover the surface $\tSigma$ (see next subsection for an example) and the functors $F$ and $F'$ depend on the choice of $\tau$, so we use the notation $\cC_{\tau}$ and $\cC_{\ttau}$ to emphasize this fact. 
\end{remark}


\subsection{Example}

Let $(\Sigma,\cM,\cP)$ be a cylinder with two punctures $\cP=\{P,Q\}$ and two marked points $\cM=\{A,B\}$. Assume first that  the two self-folded triangles of the triangulation $\tau$ are attached to the same boundary component.
Cutting the surface along the folded sides and along an arc $[A,B]$, we obtain $\Sigma$ as the following polygon with identification of sides.

  \[\scalebox{0.7}
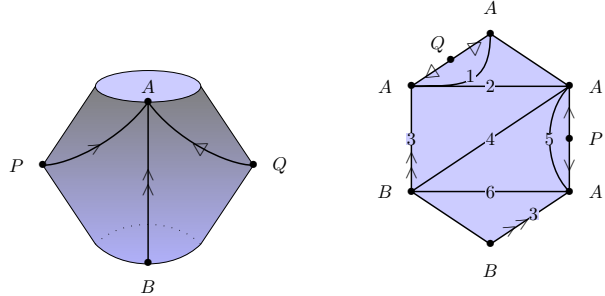
{
\begin{tikzpicture}[>=stealth,scale=1]

\shadedraw[bottom color=blue!30] (-1,0)..controls (-0.5,-0.5) and (0.5,-0.5)..(1,0)--(2,1.5)--(1,3)--(-1,3)--(-2,1.5)--(-1,0);
\draw[loosely dotted] (-1,0)..controls (-0.5,0.5) and (0.5,0.5) ..(1,0);
\draw[fill= blue!20] (0,3) ellipse (1 and 0.3);
\node at (0,-0.35) {$\bullet$};
\node at (0,2.7) {$\bullet$};
\node at (2,1.5) {$\bullet$};
\node at (-2,1.5) {$\bullet$};
\draw[thick] (-2,1.5)..controls (-1.5,1.5) and (-0.5,2)..node [rotate=30] {$\bf{>}$} (0,2.7);
\draw[thick] (2,1.5)..controls (1.5,1.5) and (0.5,2)..node [rotate=-60] {$\bf{\triangle}$} (0,2.7);
\draw[thick] (0,-0.35)--node[rotate=90]{$\bf{>>}$}(0,2.7);
\node at (0,3) {$A$};
\node at (0,-0.8) {$B$};
\node at (2.5,1.5) {$Q$};
\node at (-2.5,1.5) {$P$};

\begin{scope}[xshift=5cm,yshift=1cm]
\draw[thick, fill=blue!20](0,0)--(1.5,-1)--node[fill=blue!20, inner sep=0pt,xshift=2pt,yshift=2pt]{$3$}(3,0)--(3,2)--(1.5,3)--(0,2)--node[fill=blue!20, inner sep=0pt]{$3$}(0,0);
\node at (0,0) {$\bullet$};\node at (1.5,-1) {$\bullet$};\node at (3,0) {$\bullet$};\node at (3,2) {$\bullet$};\node at (1.5,3) {$\bullet$};\node at (0,2) {$\bullet$};\node at (0.75,2.5) {$\bullet$};
\node at (3,1) {$\bullet$};

\node[rotate=90] at (0,0.5) {$\bf{>>}$};
\node[rotate=30] at (2,-0.66){$\bf{>>}$};
\node[rotate=5] at (0.4,2.25) {$\bf{\triangle}$};\node[rotate=185] at (1.2,2.75) {$\bf{\triangle}$};
\node[rotate=-90] at (3,0.5) {$>$};
\node[rotate=90] at (3,1.5){$>$};
\node at (-0.5,0) {$B$};
\node at (1.5,-1.5) {$B$};
\node at (-0.5,2) {$A$};
\node at (0.5,2.8) {$Q$};
\node at (1.5,3.5) {$A$};
\node at (3.5,2) {$A$};
\node at (3.5,0){$A$};
\node at (3.5,1){$P$};
\draw[thick] (3,0)--node[fill=blue!20, inner sep=0pt]{$6$}(0,0)--node[fill=blue!20, inner sep=0pt]{$4$}(3,2)--node[fill=blue!20, inner sep=0pt]{$2$}(0,2)..controls (1,2) and (1.5,2)..node[fill=blue!20, inner sep=0pt,yshift=2pt]{$1$}(1.5,3);
\draw[thick] (3,0).. controls (2.5,0.5) and (2.5,1.5).. node[fill=blue!20, inner sep=0pt]{$5$}(3,2);

\end{scope}

\end{tikzpicture}}
\]

The surfaces $\Sigma^+$ and $\Sigma^-$ are then given by the following polygons with identification:

  \[\scalebox{0.7}
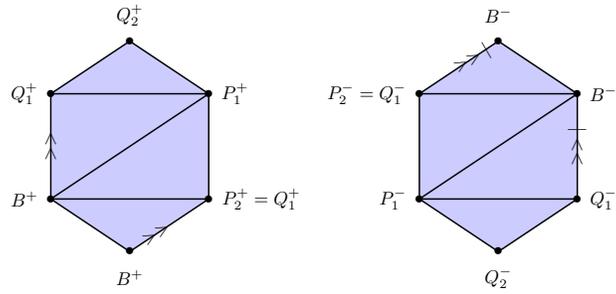
{
\begin{tikzpicture}[>=stealth,scale=1]

\draw[thick, fill=blue!20](0,0)--(1.5,-1)--(3,0)--(3,2)--(1.5,3)--(0,2)--(0,0);
\node at (0,0) {$\bullet$};\node at (1.5,-1) {$\bullet$};\node at (3,0) {$\bullet$};\node at (3,2) {$\bullet$};\node at (1.5,3) {$\bullet$};\node at (0,2) {$\bullet$};

\node[rotate=90] at (0,1) {$\bf{>>}$};
\node[rotate=30] at (2,-0.66){$\bf{>>}$};

\node at (-0.5,0) {$B^+$};
\node at (1.5,-1.5) {$B^+$};
\node at (-0.5,2) {$Q_1^+$};
\node at (1.5,3.5) {$Q_2^+$};
\node at (3.5,2) {$P_1^+$};
\node at (4,0){$P_2^+=Q_1^+$};

\draw[thick] (3,0)--(0,0)--(3,2)--(0,2);

\begin{scope}[rotate=180, xshift=-10cm,yshift=-2cm]
\draw[thick, fill=blue!20](0,0)--(1.5,-1)--(3,0)--(3,2)--(1.5,3)--(0,2)--(0,0);
\node at (0,0) {$\bullet$};\node at (1.5,-1) {$\bullet$};\node at (3,0) {$\bullet$};\node at (3,2) {$\bullet$};\node at (1.5,3) {$\bullet$};\node at (0,2) {$\bullet$};

\node[rotate=90] at (0,1) {$\bf{>>|}$};
\node[rotate=30] at (2,-0.66){$\bf{>>|}$};

\node at (-0.5,0) {$B^-$};
\node at (1.5,-1.5) {$B^-$};
\node at (-0.5,2) {$Q_1^-$};
\node at (1.5,3.5) {$Q_2^-$};
\node at (3.5,2) {$P_1^-$};
\node at (4,0){$P_2^-=Q_1^-$};

\draw[thick] (3,0)--(0,0)--(3,2)--(0,2);

\end{scope}

\end{tikzpicture}}\]

Hence the surface $\tSigma$ is a sphere with four boundary components and is given by the following polygon with identification:

  \[\scalebox{0.7}
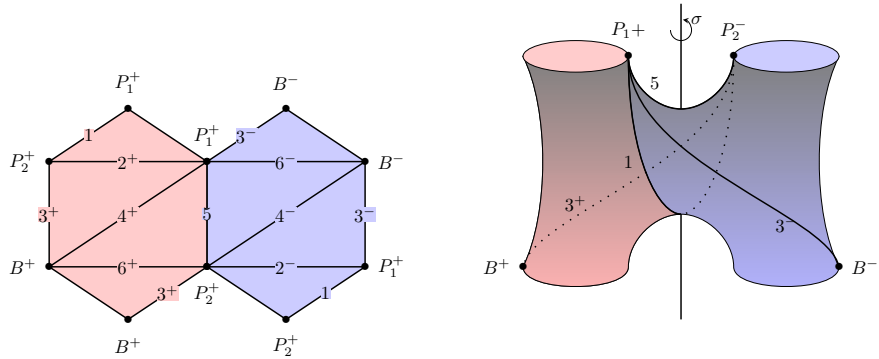
{
\begin{tikzpicture}[>=stealth,scale=1]

\draw[thick, fill=red!20](0,0)--(1.5,-1)--node[fill=red!20, inner sep=0pt]{$3^+$}(3,0)--node[fill=red!20, inner sep=0pt]{$5$}(3,2)--(1.5,3)--node[fill=red!20, inner sep=0pt]{$1$}(0,2)--node[fill=red!20, inner sep=0pt]{$3^+$}(0,0);
\node at (0,0) {$\bullet$};\node at (1.5,-1) {$\bullet$};\node at (3,0) {$\bullet$};\node at (3,2) {$\bullet$};\node at (1.5,3) {$\bullet$};\node at (0,2) {$\bullet$};

\node at (-0.5,0) {$B^+$};
\node at (1.5,-1.5) {$B^+$};
\node at (-0.5,2) {$P_2^+$};
\node at (1.5,3.5) {$P_1^+$};
\node at (3,2.5) {$P_1^+$};
\node at (3,-0.5){$P_2^+$};

\draw[thick] (3,0)--node[fill=red!20, inner sep=0pt]{$6^+$}(0,0)--node[fill=red!20, inner sep=0pt]{$4^+$}(3,2)--node[fill=red!20, inner sep=0pt]{$2^+$}(0,2);

\begin{scope}[rotate=180, xshift=-6cm,yshift=-2cm]
\draw[thick, fill=blue!20](0,0)--(1.5,-1)--node[fill=blue!20, inner sep=0pt]{$3^-$}(3,0)--node[fill=blue!20, inner sep=0pt]{$5$}(3,2)--(1.5,3)--node[fill=blue!20, inner sep=0pt]{$1$}(0,2)--node[fill=blue!20, inner sep=0pt]{$3^-$}(0,0);

\node at (0,0) {$\bullet$};\node at (1.5,-1) {$\bullet$};\node at (3,0) {$\bullet$};\node at (3,2) {$\bullet$};\node at (1.5,3) {$\bullet$};\node at (0,2) {$\bullet$};

\node at (-0.5,0) {$B^-$};
\node at (1.5,-1.5) {$B^-$};
\node at (-0.5,2) {$P_1^+$};
\node at (1.5,3.5) {$P_2^+$};

\draw[thick] (3,0)--node[fill=blue!20, inner sep=0pt]{$6^-$}(0,0)--node[fill=blue!20, inner sep=0pt]{$4^-$}(3,2)--node[fill=blue!20, inner sep=0pt]{$2^-$}(0,2);

\end{scope}

\begin{scope}[xshift=9cm]

\draw[thick] (3,5)--(3,-1);
\node at (3.3,4.7) {$\sigma$};

\draw[->] (2.8,4.5) arc (-180:80:0.2);
\shadedraw[bottom color=blue!30] (0,0)..controls (0.5,1) and (0.5,3)..(0,4)--(2,4)..controls (2,3.5) and (2.5,3)..(3,3)..controls (3.5,3) and (4,3.5).. (4,4)--(6,4)..controls (5.5,3) and (5.5,1)..(6,0)..controls (6,-0.5) and (4,-0.5)..(4,0)..controls (4,0.5) and (3.5,1)..(3,1).. controls (2.5,1) and ( 2,0.5)..(2,0)..controls (2,-0.5) and (0,-0.5)..(0,0);

\shadedraw[bottom color=red!30] (0,0)..controls (0.5,1) and (0.5,3)..(0,4)--(2,4).. controls (2,2) and (2.6,1)..(3,1).. controls (2.5,1) and ( 2,0.5)..(2,0)..controls (2,-0.5) and (0,-0.5)..(0,0);


\draw[fill=red!20] (1,4) ellipse (1 and 0.3);
\draw[fill=blue!20] (5,4) ellipse (1 and 0.3);

\draw[thick] (2,4)..controls (2,3.5) and (2.5,3)..(3,3)..controls (3.5,3) and (4,3.5).. (4,4);
\draw[thick] (2,4).. controls (2,2) and (2.6,1)..(3,1);
\draw[thick, loosely dotted] (4,4).. controls (4,2) and (3.4,1)..(3,1);

\draw[thick] (2,4)..controls (2,2) and (5.8,1)..(6,0);
\draw[thick,loosely dotted] (4,4)..controls (4,2) and (0.2,1)..(0,0);
\node at (2.5,3.5) {$5$};
\node at (2,2) {$1$};
\node at (1,1.2) {$3^+$};
\node at (5,0.8) {$3^-$};

\node at (0,0) {$\bullet$};
\node at (-0.5,0) {$B^+$};
\node at (2,4) {$\bullet$};
\node at (2,4.5) {$P_1+$};
\node at (4,4) {$\bullet$};
\node at (4,4.5) {$P_2^-$};
\node at (6,0) {$\bullet$};
\node at (6.5,0) {$B^-$};

\end{scope}

\end{tikzpicture}}\]

The quivers $Q(\tau)$ and $Q(\ttau)$ have the following shape:

  \[\scalebox{0.7}
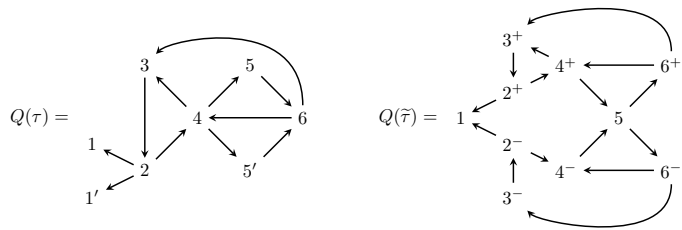
{
\begin{tikzpicture}[>=stealth,scale=1]

\node at (0,0) {$Q(\tau)=$};
\node (A1) at (1,-0.5) {$1$};
\node (A1') at (1,-1.5) {$1'$};
\node (A2) at (2,-1) {$2$};
\node (A3) at (2,1) {$3$};
\node (A4) at (3,0) {$4$};
\node (A5) at (4,1) {$5$};
\node (A5') at (4,-1) {$5'$};
\node (A6) at (5,0) {$6$};

\draw[thick,->] (A2)--(A1);\draw[thick,->] (A2)--(A1');\draw[thick,->] (A3)--(A2);\draw[thick,->] (A2)--(A4);\draw[thick,->] (A4)--(A3);\draw[thick,->] (A4)--(A5);\draw[thick,->] (A5)--(A6);\draw[thick,->] (A4)--(A5');\draw[thick,->] (A5')--(A6);\draw[thick,->] (A6)--(A4);
\draw[thick,->] (A6) ..controls (5,2) and (2.5,1.5).. (A3);

\begin{scope}[xshift=7cm]
\node at (0,0) {$Q(\ttau)=$};
\node (B1) at (1,0) {$1$};\node (B2+) at (2,0.5) {$2^+$};\node (B2-) at (2,-0.5) {$2^-$};\node (B3+) at (2,1.5) {$3^+$};\node (B3-) at (2,-1.5) {$3^-$};\node (B4+) at (3,1) {$4^+$};\node (B4-) at (3,-1) {$4^-$};\node (B5) at (4,0) {$5$};\node (B6+) at (5,1) {$6^+$};\node (B6-) at (5,-1) {$6^-$};

\draw[thick,->] (B2+)--(B1);\draw[thick,->] (B2-)--(B1);\draw[thick,->] (B3+)--(B2+);\draw[thick,->] (B2+)--(B4+);\draw[thick,->] (B4+)--(B3+);\draw[thick,->] (B4+)--(B5);\draw[thick,->] (B5)--(B6+);\draw[thick,->] (B4-)--(B5);\draw[thick,->] (B5)--(B6-);\draw[thick,->] (B6-)--(B4-);
\draw[thick,->] (B6+) ..controls (5,2.5) and (2.5,2).. (B3+);\draw[thick,->] (B6-) ..controls (5,-2.5) and (2.5,-2).. (B3-);\draw[thick,->] (B6+)--(B4+);\draw[thick,->] (B2-)--(B4-);\draw[thick,->] (B3-)--(B2-);

\end{scope}

\end{tikzpicture}}\]

Now for the same surface $(\Sigma,\cP,\cM)$ let $\tau$ be a triangulation such that the two self folded triangles are linked to different boundary components. 

  \[\scalebox{0.7}
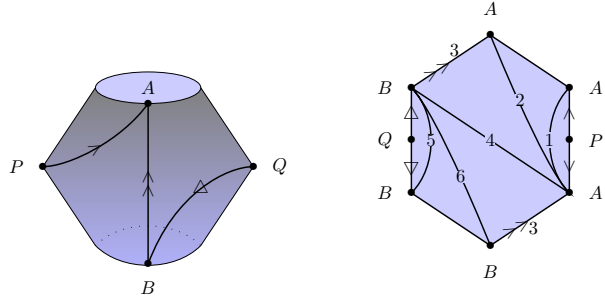
{
\begin{tikzpicture}[>=stealth,scale=1]

\shadedraw[bottom color=blue!30] (-1,0)..controls (-0.5,-0.5) and (0.5,-0.5)..(1,0)--(2,1.5)--(1,3)--(-1,3)--(-2,1.5)--(-1,0);
\draw[loosely dotted] (-1,0)..controls (-0.5,0.5) and (0.5,0.5) ..(1,0);
\draw[fill= blue!20] (0,3) ellipse (1 and 0.3);
\node at (0,-0.35) {$\bullet$};
\node at (0,2.7) {$\bullet$};
\node at (2,1.5) {$\bullet$};
\node at (-2,1.5) {$\bullet$};
\draw[thick] (-2,1.5)..controls (-1.5,1.5) and (-0.5,2)..node [rotate=30] {$\bf{>}$} (0,2.7);
\draw[thick] (2,1.5)..controls (1.5,1.5) and (0.5,1)..node  {$\bf{\triangle}$} (0,-0.35);
\draw[thick] (0,-0.35)--node[rotate=90]{$\bf{>>}$}(0,2.7);
\node at (0,3) {$A$};
\node at (0,-0.8) {$B$};
\node at (2.5,1.5) {$Q$};
\node at (-2.5,1.5) {$P$};

\begin{scope}[xshift=5cm,yshift=1cm]
\draw[thick, fill=blue!20](0,0)-- (1.5,-1)--node[xshift=2pt,yshift=-5pt]{$3$}(3,0)--node {$\bullet$}(3,2)--(1.5,3)--node [xshift=2pt,yshift=6pt]{$3$}(0,2)--node {$\bullet$}(0,0);
\node at (0,0) {$\bullet$};\node at (1.5,-1) {$\bullet$};\node at (3,0) {$\bullet$};\node at (3,2) {$\bullet$};\node at (1.5,3) {$\bullet$};\node at (0,2) {$\bullet$};

\node[rotate=30] at (0.5,2.33) {$\bf{>>}$};
\node[rotate=30] at (2,-0.66){$\bf{>>}$};
\node at (0,1.5) {$\bf{\triangle}$};\node[rotate=180] at (0,0.5) {$\bf{\triangle}$};
\node[rotate=-90] at (3,0.5) {$>$};
\node[rotate=90] at (3,1.5){$>$};
\node at (-0.5,0) {$B$};
\node at (1.5,-1.5) {$B$};
\node at (-0.5,2) {$B$};
\node at (-0.5,1) {$Q$};
\node at (1.5,3.5) {$A$};
\node at (3.5,2) {$A$};
\node at (3.5,0){$A$};
\node at (3.5,1){$P$};
\draw[thick] (0,0).. controls (0.5,0.5) and (0.5,1.5).. node[fill=blue!20, inner sep=0pt]{$5$}(0,2)..controls (0.5,1.5) and (1.5,-1)..node[fill=blue!20, inner sep=0pt]{$6$}(1.5,-1);
\draw[thick](0,2)--node[fill=blue!20, inner sep=0pt]{$4$}(3,0)..controls (2.5,0.5) and (1.5,3) ..node[fill=blue!20, inner sep=0pt,yshift=2pt]{$2$}(1.5,3);
\draw[thick] (3,0).. controls (2.5,0.5) and (2.5,1.5).. node[fill=blue!20, inner sep=1pt]{$1$}(3,2);

\end{scope}

\end{tikzpicture}}
\]

The surface $\tSigma$ is then a torus with two boundary components given by the following polygon with identification:

  \[\scalebox{0.7}
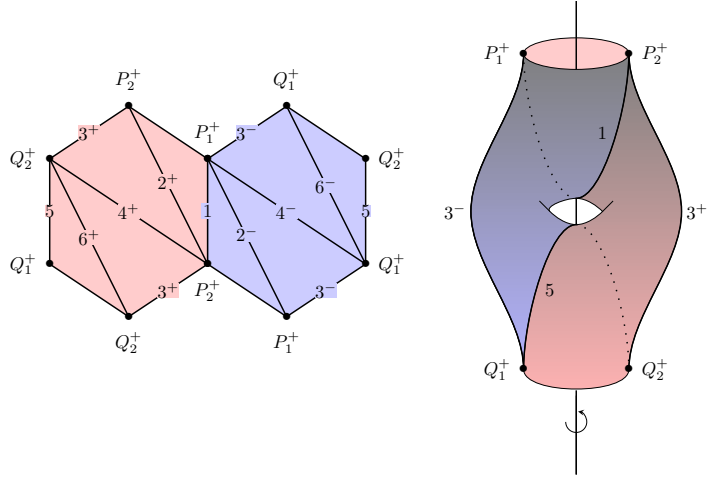
{
\begin{tikzpicture}[>=stealth,scale=1]

\begin{scope}[xshift=0cm,yshift=0cm]
\draw[thick, fill=red!20](0,0)-- (1.5,-1)--node[fill=red!20, inner sep=0pt]{$3^+$}(3,0)--node [fill=red!20, inner sep=1pt]{$1$}(3,2)--(1.5,3)--node [fill=red!20, inner sep=0pt]{$3^+$}(0,2)--node [fill=red!20, inner sep=0pt]{$5$}(0,0);
\node at (0,0) {$\bullet$};\node at (1.5,-1) {$\bullet$};\node at (3,0) {$\bullet$};\node at (3,2) {$\bullet$};\node at (1.5,3) {$\bullet$};\node at (0,2) {$\bullet$};

\node at (-0.5,0) {$Q_1^+$};
\node at (1.5,-1.5) {$Q_2^+$};
\node at (-0.5,2) {$Q_2^+$};
\node at (1.5,3.5) {$P_2^+$};
\node at (3,2.5) {$P_1^+$};
\node at (3,-0.5){$P_2^+$};

\draw[thick] (0,2)--node[fill=red!20, inner sep=0pt]{$6^+$}(1.5,-1);
\draw[thick](0,2)--node[fill=red!20, inner sep=0pt]{$4^+$}(3,0)--node[fill=red!20, inner sep=0pt,yshift=2pt]{$2^+$}(1.5,3);

\end{scope}

\begin{scope}[rotate=180, xshift=-6cm,yshift=-2cm]
\draw[thick, fill=blue!20](0,0)-- (1.5,-1)--node[fill=blue!20, inner sep=0pt]{$3^-$}(3,0)--node [fill=blue!20, inner sep=1pt]{$1$}(3,2)--(1.5,3)--node [fill=blue!20, inner sep=0pt]{$3^-$}(0,2)--node [fill=blue!20, inner sep=0pt]{$5$}(0,0);
\node at (0,0) {$\bullet$};\node at (1.5,-1) {$\bullet$};\node at (3,0) {$\bullet$};\node at (3,2) {$\bullet$};\node at (1.5,3) {$\bullet$};\node at (0,2) {$\bullet$};

\node at (-0.5,0) {$Q_2^+$};
\node at (1.5,-1.5) {$Q_1^+$};
\node at (-0.5,2) {$Q_1^+$};
\node at (1.5,3.5) {$P_1^+$};

\draw[thick] (0,2)--node[fill=blue!20, inner sep=0pt]{$6^-$}(1.5,-1);
\draw[thick](0,2)--node[fill=blue!20, inner sep=0pt]{$4^-$}(3,0)--node[fill=blue!20, inner sep=0pt,yshift=2pt]{$2^-$}(1.5,3);

\end{scope}

\begin{scope}[xshift=9cm,yshift=-2cm]

\draw (1,-2)--(1,-0.5);

\shadedraw[bottom color=red!30] (0,0).. controls (0,-0.5) and (2,-0.5)..(2,0).. controls (2,1) and (3,2)..(3,3)..controls (3,4) and (2,5)..(2,6)--(0,6)..controls (0,5) and (-1,4)..(-1,3)..controls (-1,2) and (0,1)..(0,0);

\shadedraw[bottom color=blue!30] (0,0).. controls (0,1) and (0.5,2.75)..(1,2.75)--(1,3.25)..controls (1.5,3.25) and (2,5)..(2,6)--(0,6)..controls (0,5) and (-1,4)..(-1,3)..controls (-1,2) and (0,1)..(0,0);

\draw[fill=red!20] (1,6) ellipse (1 and 0.3);
\draw[fill=white] (0.5,3)..controls (0.6,2.9) and (0.8,2.75)..(1,2.75)..controls (1.2,2.75) and (1.4,2.9).. (1.5,3)..controls (1.4,3.1) and (1.2,3.25)..(1,3.25)..controls (0.8,3.25) and (0.4,3.1)..(0.5,3);

\draw[thick] (0,0).. controls (0,1) and (0.5,2.75)..(1,2.75);
\draw[thick, loosely dotted] (1,2.75).. controls (1.5,2.75) and (2,1)..(2,0);
\draw[thick,loosely dotted] (0,6).. controls (0,5) and (0.5,3.25)..(1,3.25);
\draw[thick] (2,6)..controls (2,5) and (1.5,3.25)..(1,3.25);

\draw[thick](2,0).. controls (2,1) and (3,2)..(3,3)..controls (3,4) and (2,5)..(2,6);
\draw[thick](0,6)..controls (0,5) and (-1,4)..(-1,3)..controls (-1,2) and (0,1)..(0,0);
\node at (-1.3,3) {$3^-$};
\node at (3.3,3) {$3^+$};
\node at (0.5,1.5) {$5$};
\node at (1.5,4.5) {$1$};

\draw(0.5,3)--(0.3,3.2);\draw(1.5,3)--(1.7,3.2);

\node at (0,0){$\bullet$};
\node at (2,0) {$\bullet$};
\node at (0,6){$\bullet$};
\node at (2,6) {$\bullet$};
\draw[thick] (1,-2)--(1,-0.4);
\draw[thick] (1,2.75)--(1,3.25);
\draw[thick] (1,5.7)--(1,7);

\draw[->] (0.8,-1) arc (-180:80:0.2);

\node at (-0.5,0) {$Q_1^+$};
\node at (-0.5,6) {$P_1^+$};
\node at (2.5,6) {$P_2^+$};
\node at (2.5,0) {$Q_2^+$};

\end{scope}

\end{tikzpicture}}
\]

The quivers associated with $\tau$ and $\ttau$ are respectively:

  \[\scalebox{0.7}
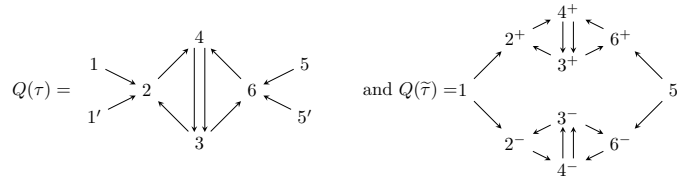
{
\begin{tikzpicture}[>=stealth,scale=1]

\node at (0,0) {$Q(\tau)=$};
\node (A1) at (1,0.5) {$1$}; \node (A1') at (1,-0.5) {$1'$};
\node (A2) at (2,0) {$2$};\node (A3) at (3,-1) {$3$};\node (A4) at (3,1) {$4$};
\node (A6) at (4,0) {$6$}; \node (A5) at (5,0.5) {$5$};\node (A5') at (5,-0.5) {$5'$};

\draw[->](A1)--(A2);\draw[->](A1')--(A2);\draw[->](A2)--(A4);\draw[->](A6)--(A4);\draw[->](A3)--(A2);\draw[->](A3)--(A6);\draw[->](A5)--(A6);\draw[->](A5')--(A6);
\draw[->] (3.1,0.8)--(3.1,-0.8);\draw[->] (2.9,0.8)--(2.9,-0.8);

\begin{scope}[xshift=7cm]
\node at (0,0) {and $Q(\ttau)=$};
\node (B1) at (1,0) {$1$}; \node (B2+) at (2,1) {$2^+$}; \node (B2-) at (2,-1) {$2^-$}; \node (B3+) at (3,0.5) {$3^+$}; \node (B3-) at (3,-0.5) {$3^-$}; \node (B4+) at (3,1.5) {$4^+$}; \node (B4-) at (3,-1.5) {$4^-$}; \node (B5) at (5,0) {$5$};\node (B6+) at (4,1) {$6^+$};\node (B6-) at (4,-1) {$6^-$};  

\draw[->](B1)--(B2+);\draw[->](B1)--(B2-);\draw[->](B2+)--(B4+);\draw[->](B2-)--(B4-);\draw[->](B6+)--(B4+);\draw[->](B6-)--(B4-);\draw[->](B3+)--(B2+);\draw[->](B3-)--(B2-);\draw[->](B5)--(B6+);
\draw[->](B5)--(B6-);\draw[->](B3+)--(B6+);\draw[->](B3-)--(B6-);
\draw[->] (2.9,1.3)--(2.9,0.7);\draw[->] (2.9,-1.3)--(2.9,-0.7);\draw[->] (3.1,1.3)--(3.1,0.7);\draw[->] (3.1,-1.3)--(3.1,-0.7);

\end{scope}

\end{tikzpicture}}\]

\begin{remark}
Even though the surface $\tSigma$ is not unique, an easy calculation shows that its rank is always $4g+2b+p-3= 2{\rm rk}(\Sigma)-(p+1)$ so does not depend on the choice of $\tau$. This is confirmed by the parametrization of indecomposable objects of $\cC_{\tau}$ in terms of curves on $\tSigma$ made in the next section.
\end{remark}

\subsection{Action of $\sigma$ on $\widetilde{\Sigma}$ and orbifold structure on $\Sigma$}

In this subsection we show that the action of $\sigma$ as an automorphism of order $2$ on $(Q(\ttau),S(\ttau))$ is induced by an automorphism of order two of the surface $\widetilde{\Sigma}$.

\begin{proposition}\label{prop sigma homeo}
The homeomorphism $S:\Sigma^+\sqcup \Sigma^-\to \Sigma^+\sqcup\Sigma^-$ exchanging $\Sigma^+$ and $\Sigma^-$ induces a homeomorphism $\sigma:\widetilde{\Sigma}\to \widetilde{\Sigma}$ of order 2 whose fixed points are the $P^+$, $P\in \cP$. 
\end{proposition}

\begin{proof}
We prove first that the map $\sigma$ is well-defined. Since $S$ has order two, for any $x\in [P_1^+,P_2^+]$ we have $$S\circ\Psi_P(x)=\Psi_P^{-1}\circ S(x)=\varphi_P^{-1}\circ I \circ\varphi_P(x).$$
Therefore if $\pi(x)=\pi(y)$ then $\pi(S(x))=\pi(S(y))$ and $\sigma$ is well defined. It has clearly order two and is orientation preserving since $S$ is.

Let $x\in \Sigma^+$ such that $\pi(S(x))=\pi(x)$. Then there exists $\eg{1}P,\ldots,\eg{2\ell-1}P\in \cP$ such that 
$$ \Psi_{\eg{2\ell-1}P}\circ\cdots\circ \Psi_{\eg{2}P}^{-1}\circ \Psi_{\eg{1}P}(x)=S(x).$$
If $\ell\geq 2$ then the maps $\Psi_{\eg{i}P}$ compose (cf Picture \eqref{severalpunctures}), thus $x=\eg{1}P^+_1$ and  $\Psi_{\eg{2\ell-1}P}\circ\cdots\circ \Psi_{P_2}^{-1}\circ \Psi_{P_1}(x)=\eg{2\ell-1}P^-_1$ while $S(x)=\eg{1}P^-_1$. So this situation cannot occur.
If $\ell=1$, then $\Psi_P(x)=S(x)$ implies that $x$ is a fixed point of the map $\varphi_P^{-1}\circ I\circ\varphi_P$ so is $P^+=\varphi_P^{-1}(\frac{1}{2})$.

It remains to show that $\sigma$ is a homeomorphism. Since it has order two, it is sufficient to prove that it is an open map.
 Let $U$ be an open set in $\widetilde{\Sigma}$. By definition of the topology on $\widetilde{\Sigma}$ the open set $\pi^{-1}(U)$ satisfies the gluing condition, that is, we have 
 $$\forall P\in \cP,\quad \Psi_P(\pi^{-1}(U)\cap [P_1^+,P_2^+])=\pi^{-1}(U)\cap [P_1^-,P_2^-].$$
 The set $\pi^{-1}(\sigma(U))=S(\pi^{-1}(U))$ is open in $\Sigma^+\sqcup\Sigma^-$ since $S$ is a homeomorphism. Moreover for $P\in \cP$ we have
 \[\begin{array}{rcll} \Psi_P(\pi^{-1}(\sigma(U))\cap [P_1^+,P_2^+]) & = & \Psi_P(S(\pi^{-1}(U))\cap S( [P_1^-,P_2^-]) & \\
  & = & \Psi_P(S(\pi^{-1}(U)\cap [P_1^-,P_2^-])) & \textrm{since }S\textrm{ is bijective,}\\
  & = & S\circ\Psi_P^{-1} (\pi^{-1}(U)\cap [P_1^-,P_2^-]) & \\
   & = & S (\pi^{-1}(U)\cap [P_1^+,P_2^+]) & \textrm{by the gluing condition,}\\
    & = & S(\pi^{-1}(U))\cap S([P_1^+,P_2^+]) &\\
    & =& \pi^{-1}(\sigma (U))\cap [P_1^-,P_2^-]. &
 \end{array}\]
 Hence $\pi^{-1}(\sigma(U))$ satisfies the gluing condition, and $\sigma(U)$ is an open set in $\widetilde{\Sigma}$. 
  
\end{proof}
We referer to \cite[Chapter 13]{Thurston} for basic definitions on orbifolds.

\begin{corollary} There is a homeomorphism $\widetilde{\Sigma}/\sigma\to \Sigma$ that induces an orbifold structure on $\Sigma$ on which every $P\in \cP$ is an orbifold point of order $2$.
\end{corollary}

\begin{proof} We have homeomorphisms $$\widetilde{\Sigma}/\sigma\simeq \Sigma^+\sqcup\Sigma^-/(S,\Psi_P,P\in \cP)\simeq \Sigma^+/(\varphi_P\circ I\circ\varphi_P^{-1},P\in \cP)\simeq \Sigma.$$ 
The group $\bZ/2\bZ$ acts then properly and discontinuously on the surface $\widetilde{\Sigma}$, hence by \cite[13.2.2]{Thurston} $\Sigma$ has an orbifold structure and $\widetilde{\Sigma}$ is an orbifold cover of $\Sigma$.

\end{proof}

\section{Indecomposable objects of $\mathcal{C}_{\tau}$ in terms of curves on $\widetilde{\Sigma}$}\label{section3}

From now on and in the rest of the paper, we assume that $k$ is an algebraically closed field of characteristic different from $2$.

Let $(\Sigma,\cM,\cP), \tau$ and  $(\widetilde{\Sigma},\widetilde{\cM}),\widetilde{\tau}$ be as in Section \ref{section2}. Denote by $\cC_{\widetilde{\tau}}$ and $\cC_\tau$ the corresponding cluster categories. Corollary \ref{corollary functor cluster categories} implies that provided we have a description of the indecomposable objects of $\cC_{\widetilde{\tau}}$ and of the action of the automorphism $\sigma$ on these objects, we obtain a complete description of the indecomposable objects of $\cC_\tau$. Since the surface $(\widetilde{\Sigma},\widetilde{\cM})$ does not have any punctures, a complete parametrization of the indecomposable objects of $\cC_{\widetilde{\tau}}$ has been given by Br\"ustle and Zhang in \cite{BZ}. The aim of this section is to understand the action of the automorphism $\sigma$ to obtain a complete description of the objects of $\cC_\tau$. 

We start with some definitions.
\begin{definition}
Given the functor $F':\cC_{\tau}\to \cC_{\ttau}$ given by Corollary \ref{corollary functor cluster categories}, we call an object $X$ of $\cC_\tau$ a \emph{$\tSigma$-string} (resp. \emph{$\tSigma$-band}) object if $F'X$ is a string object (or the sum of two string objects) (resp. band).   
\end{definition}

\subsection{$\tSigma$-string objects}

Denote by $\pi_1(\widetilde{\Sigma},\widetilde{\cM})$ the groupoid of paths on $\widetilde{\Sigma}$ with endpoints in $\widetilde{\cM}$ (see Section \ref{section4} for precise definition). In \cite{BZ}, the authors associate to each non trivial element $\gamma$ in $\pi_1(\widetilde{\Sigma},\widetilde{\cM})$ an indecomposable object ${\tt M}^{\widetilde{\tau}}(\gamma)\in \cC_{\widetilde{\tau}}$. First recall the following classical fact.

\begin{lemma}\label{lemm::order-arc-triangulation}
 Let $\Sigma$ be a surface and $\tau$ be an ideal triangulation.  If $\gamma$ is a curve on $\Sigma$, then up to homotopy, $\gamma$ can be chosen so that it crosses arcs of $\tau$ transversally, crosses them finitely many times, and does not cross the same arc twice in succession.
 In this case, $\gamma$ is completely determined (up to homotopy) by the order in which it crosses the arcs of $\tau$.
\end{lemma}

 Let us give here the description of the map ${\tt M}^{\widetilde{\tau}}(?)$ by induction on the minimal number $\ell^{\widetilde{\tau}}(\gamma)$ of arcs of $\widetilde{\tau}$ intersected by a transversal representative of $\gamma$. Note that in \cite{BR,ABCP,BZ}, the description is a direct one rather than an inductive one. We choose this inductive presentation since it permits to describe more easily the action of $\sigma$ on the objects of the category $\cC_{\ttau}$ (Lemma \ref{lemma sigma string}).


For any arc $i\in \widetilde{\tau}$, the object ${\tt M}^{\widetilde{\tau}}(i)$ is defined to be the image $X_i$ of the indecomposable object $e_i\widehat{\Gamma}(Q(\widetilde{\tau}),S(\widetilde{\tau}))$ under the natural functor $$\cD \widehat{\Gamma}(Q(\widetilde{\tau}),S(\widetilde{\tau}))\longrightarrow \cC_{\ttau}.$$

If $\ell^{\ttau}(\gamma)=1$ then $\gamma$ is homotopic to the flip $\flip i$ of some arc $i$ in $\ttau$. Then the object $\tM(\flip i)$ sits in the following triangle
\begin{equation}\label{m_flip}\xymatrix{\bigoplus_{a\in Q_1,t(a)=i}\tM(s(a))[-1]\ar[r]^-{(a)} & \tM(i)[-1]\ar[r] & \tM(\flip i)\ar[r] & \bigoplus_a\tM(s(a))}.\end{equation}

The object $X=\bigoplus_i X_i$ is a cluster-tilting object in $\cC_{\ttau}$ by \cite{Ami09}. By results of \cite{BMR07,KR07}, the functor ${\rm H}=\Hom{\cC} (X[-1],-):\cC_{\ttau}\to \mod \textrm{Jac} (Q(\ttau),S(\ttau))$ is essentially surjective and sends any indecomposable objects not isomorphic to $X$ to an indecomposable module.  
For any arc $i$ of $\ttau$, the module $\rm H (X_i[-1])$ is isomorphic to the projective indecomposable associated to the vertex $i$.  Therefore, applying ${\rm H}$ to the triangle \eqref{m_flip}, we get that the module $\textrm{H} (\tM(\flip i))$ is isomorphic to the simple $S(i)$ supported by the vertex $i$.

For any arrow $\alpha:i\to j$ in $Q(\ttau)$, we define $\tM(\alpha)$ to be the indecomposable object of $\cC_{\ttau}$ such that ${\rm H}\tM(\alpha)$ is the indecomposable module supported by the arrow $\alpha$.  Then there exists a short exact sequence in $\mod \textrm{Jac}(Q(\ttau),S(\ttau))$: 
\begin{equation}\label{m_arrow}\xymatrix{0\ar[r] & {\rm H}\tM(\flip i)\ar[r]^{{\bf i}_\alpha} & {\rm H}\tM(\alpha)\ar[r]^{{\bf p}_\alpha}\ar[r] & {\rm H}\tM(\flip j) \ar[r] & 0}.\end{equation}

\begin{definition}
Let $\gamma$ be a non trivial element in $\pi_1(\tSigma,\widetilde{\cM})$ which is not an arc of $\ttau$. We say that $\gamma$ starts \emph{directly} (resp. \emph{undirectly}) if the first angle of $\ttau$ intersected by a transversal representative of $\gamma$ agrees (resp. disagrees) with the orientation of $\tSigma$ (see the picture below) or if $\gamma=\flip i$ for some arc $i$. 

For $\gamma=\flip i$, we set ${\bf i}_\gamma={\bf p}_\gamma={\rm Id}:{\rm H}\tM(\flip i)\to {\rm H}\tM(\flip i)$.

Let $\gamma$ be an element in $\pi_1(\tSigma,\widetilde{\cM})$ with $\ell^{\ttau}(\gamma)\geq 2$. If $i_1,\ldots, i_\ell$ is the sequence of arcs intersected by $\gamma$, denote by $\gamma'$ the element of $\pi_1(\tSigma,\widetilde{\cM})$ corresponding to $i_2,\ldots, i_\ell$. 
\end{definition}

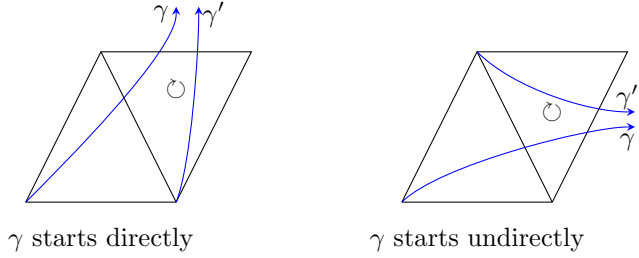
\[\scalebox{1}{
\begin{tikzpicture}[>=stealth,scale=1]
\draw (0,0)--(2,0)--(3,2)--(1,2)--(0,0);
\draw (1,2)--(2,0);
\node at (2,1.5) {$\circlearrowright$};

\draw[blue] (0,0)..controls (0.5,0.5) and (2,2)..(2,2.5);
\draw[blue,->](2,2.5)--(2,2.6);
\node at (1.8,2.5) {$\gamma$};

\draw[blue] (2,0).. controls (2.2,0.5)and (2.3,2).. (2.3,2.5);
\draw[blue,->](2.3,2.5)--(2.3,2.6);
\node at (2.5,2.5) {$\gamma'$};

\node at (1,-0.5) {$\gamma$ starts directly};

\begin{scope}[xshift=5cm]
\draw (0,0)--(2,0)--(3,2)--(1,2)--(0,0);
\draw (1,2)--(2,0);
\node at (2,1.2) {$\circlearrowright$};
\draw[blue] (0,0)..controls (0.5,0.5) and (2.5,1)..(3,1);
\draw[blue,->] (3,1)--(3.1,1);
\draw[blue] (1,2)..controls (1.5,1.5) and (2.5,1.2)..(3,1.2);
\draw[blue,->](3,1.2)--(3.1,1.2);
\node at (3,1.4) {$\gamma'$};
\node at (3,0.8) {$\gamma$};

\node at (1,-0.5) {$\gamma$ starts undirectly};
\end{scope}

\end{tikzpicture}}
\]
The next proposition characterizes objects $\tM(\gamma)$ together with maps $\bf i_\gamma:{\rm H}\tM(\flip i_1)\to {\rm H}\tM(\gamma)$ if $\gamma$ starts directly and maps $\bf p_\gamma:{\rm H}\tM(\gamma)\to {\rm H}\tM(\flip i_1)$ if $\gamma$ starts undirectly by induction on $\ell^{\ttau}(\gamma)$.
It is a direct consequence of results of \cite{BR,ABCP, BZ}.
\begin{proposition}\label{prop tM} Let $\gamma$ be an element  in $\pi_1(\tSigma,\widetilde{\cM})$ with $\ell^{\ttau}(\gamma)\geq 2$. Then the indecomposable object $\tM(\gamma)$ and the maps ${\bf i}_\gamma, \bf p_\gamma$ are uniquely characterized by the following induction properties:

\begin{enumerate}
\item If $\gamma$ and $\gamma'$ start directly, there exists a commutative diagram of the following form, where horizontal sequences are short exact sequences:
\[\xymatrix{0 \ar[r]&{\rm H}\tM(\flip i_1)\ar[r]^{{\bf i}_\alpha}\ar@{=}[d] & {\rm H}\tM(\alpha)\ar[r]^{\bf p_{\alpha}} \ar[d]& {\rm H}\tM(\flip i_2)\ar[r]^{\alpha}\ar[d]^{{\bf i}_{\gamma'}} & 0\\ 0\ar[r] &{\rm H}\tM(\flip i_1)\ar[r]^{\bf i_\gamma}& {\rm H}\tM(\gamma)\ar[r]&{\rm H}\tM(\gamma')\ar[r] &0  }\]

\item if $\gamma$ starts directly and $\gamma'$ starts undirectly, there exists a commutative diagram of the following form, where horizontal sequences are short exact sequences:
\[\xymatrix{ 0\ar[r] & {\rm H}\tM(\flip i_1)\ar[r]^{\bf i_\gamma}& {\rm H}\tM(\gamma)\ar[r]\ar[d]&{\rm H}\tM(\gamma')\ar[r] \ar[d]^{{\bf p}_{\gamma'}}&0  \\ 0\ar[r] & {\rm H}\tM(\flip i_1)\ar[r]^{{\bf i}_\alpha}\ar@{=}[u] & {\rm H}\tM(\alpha)\ar[r]^{\bf p_{\alpha}}& {\rm H}\tM(\flip i_2)\ar[r]^{\alpha} & 0\\}\]

\item If $\gamma$ starts undirectly and  $\gamma'$ starts directly, there exists a commutative diagram of the following form, where horizontal sequences are triangles:
\[\xymatrix{0\ar[r] &{\rm H}\tM(\flip i_2)\ar[r]^{\bf i_\alpha}\ar[d]^{\bf i_{\gamma'}} & {\rm H}\tM(\alpha)\ar[r]^{\bf p_{\alpha}} \ar[d]& {\rm H}\tM(\flip i_1)\ar[r]\ar@{=}[d] & 0\\0\ar[r] & {\rm H}\tM(\gamma')\ar[r]& {\rm H}\tM(\gamma)\ar[r]^{\bf p_\gamma}&{\rm H}\tM(\flip i_1)\ar[r] ^\alpha& 0  }\]

\item If $\gamma$ and $\gamma'$ start undirectly, there exists a commutative diagram of the following form, where horizontal sequences are triangles:
\[\xymatrix{0\ar[r]& {\rm H}\tM(\gamma')\ar[r]\ar[d]^{\bf p_{\gamma'}} &{\rm H} \tM(\gamma)\ar[r]^{\bf p_{\gamma}} \ar[d]& {\rm H}\tM(\flip i_1)\ar[r]\ar@{=}[d] & 0\\ 0\ar[r] & {\rm H}\tM(\flip i_2)\ar[r]^{\bf i_\alpha}& {\rm H}\tM(\alpha)\ar[r]^{\bf p_\alpha}&{\rm H}\tM(\flip i_1)\ar[r] & 0  }\]

\end{enumerate}

Moreover $\tM(\gamma)\simeq \tM(\beta)$ if and only if $\gamma=\beta$ or $\gamma=\beta^{-1}$.
\end{proposition}

By Proposition \ref{prop sigma homeo}, the action of $\sigma$ on the quiver with potential $(Q(\ttau),S(\ttau))$ is induced by an automorphism of the surface $\tSigma$ also denoted by $\sigma$. This automorphism induces an action of $\bZ/2\bZ$ on the set $\pi_1(\tSigma, \widetilde{\cM})$. The next result links directly this action to the action of $\sigma$ on the objects of $\cC_{\ttau}$.

\begin{lemma}\label{lemma sigma string} 
For any $\gamma$ non trivial in $\pi_1(\tSigma,\widetilde{\cM})$, we have an isomorphism in~$\cC_{\ttau}$
$$\tM(\gamma)^{\sigma}\simeq \tM(\sigma(\gamma)).$$
\end{lemma}

\begin{proof}
Proposition \ref{prop sigma homeo} implies that if $i$ is a vertex of $Q(\ttau)$ that is an arc of $\ttau$, then the curve $\sigma i$ on the surface $\tSigma$ is homotopic to the arc $\sigma i$ associated to the vertex $\sigma i$.
The rest follows directly from the above description of the objects $\tM(\gamma)$. Indeed, as an automorphism of triangulated categories, $\sigma$ sends a distinguished triangle of $\cC_{\ttau}$ to a distinguished triangle.
 \end{proof}

Combining the description of string objects of $\cC_{\ttau}$ with Corollary \ref{corollary functor cluster categories} and Lemma~\ref{lemma sigma string} we obtain a complete description of the $\tSigma$-string objects of $\cC_\tau$.

\begin{corollary}
The $\tSigma$-string objects in the cluster category $\cC_{\tau}$ are in bijection with the following union of sets:
\[\{\gamma\in \pi_1(\tSigma,\widetilde{\cM})\ |\ \sigma\gamma\neq \gamma^{-1}\}\cup(\{\gamma\in \pi_1(\tSigma,\widetilde{\cM})^* \ |\ \sigma\gamma=\gamma ^{-1}\}\times \bZ/2\bZ).\]
\end{corollary}

\begin{proof}
By Proposition \ref{prop tM} and Lemma \ref{lemma sigma string}, an isomorphism $\tM(\gamma)^{\sigma}\simeq \tM(\gamma)$ implies either $\sigma \gamma=\gamma$ or $\sigma\gamma=\gamma^{-1}$. Since the endpoints of an element $\gamma$ in $\pi_1(\tSigma,\widetilde{\cM})$ are in $\widetilde{\cM}$ so are not fixed points under $\sigma$, the curve $\sigma\gamma$ is never homotopic to $\gamma$ unless $\gamma$ is trivial.
The rest follows directly from Corollary \ref{corollary functor cluster categories}.
\end{proof}

\subsection{$\tSigma$-band objects}

Denote by $\pi_1^{\rm free}(\tSigma)$ the set of loops on $\tSigma$ modulo free homotopy, that is homotopy that does not fix the base point. Such a loop is called \emph{primitive} if it is not a proper power of a loop (see also Definition~\ref{def::primitive}). Note that primitive loops correspond to \emph{irreducible} loops introduced in \cite{BZ}. We call them primitive, since it is the usual terminology in the theory of free groups and lattices.

For $[\gamma]\in \pi_1^{\rm free}(\tSigma)$ primitive, there exists an arc $i$ of $\ttau$ such that a transversal representative of $\gamma$ intersects locally first undirectly then directly the two triangles adjancent to $i$, as in the following picture.

 \[\scalebox{1}
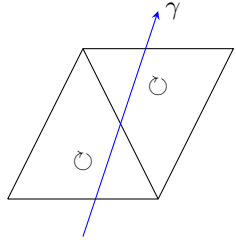
{
\begin{tikzpicture}[>=stealth,scale=1]
\draw (0,0)--(2,0)--(3,2)--(1,2)--(0,0);
\draw (1,2)--(2,0);
\node at (2,1.5) {$\circlearrowright$};
\node at (1,0.5) {$\circlearrowright$};
\draw[blue,->] (1,-0.5)--(2,2.5);
\node at (2.2,2.5) {$\gamma$};

\end{tikzpicture}}\]

 Indeed if such an $i$ does not exist, then either the angles ot $\ttau$ intersected by the curve $\gamma$ all agree with the orientation of $\tSigma$ or all disagree. This implies that the curve $\gamma$ turns around a base point of the triangulation. This situation cannot occur since there are no punctures in $(\tSigma,\widetilde{\cM})$. 
 
 Let $\gamma$ be a transversal representative of $[\gamma]$ such that $i$ is the first arc intersected by $\gamma$, and denote by $i_2,\ldots, i_\ell$ the arcs successively intersected by $\gamma$. Define an element $\gamma'$ of $\pi_1(\tSigma,\widetilde{\cM})$ as the transversal curve intersecting successively the arcs $i,i_2,\ldots,i_\ell,i$: 
 
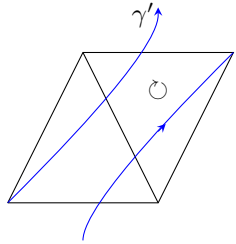
\[\scalebox{1}{
\begin{tikzpicture}[>=stealth,scale=1]
\draw (0,0)--(2,0)--(3,2)--(1,2)--(0,0);
\draw (1,2)--(2,0);
\node at (2,1.5) {$\circlearrowright$};

\draw[blue] (0,0)..controls (0.5,0.5) and (2,2)..(2,2.5);
\draw[blue,->](2,2.5)--(2,2.6);
\node at (1.8,2.5) {$\gamma'$};
\draw[blue] (3,2)..controls (2.5,1.5) and (1,0).. (1,-0.5);
\draw[blue,->](2,.95)--(2.1,1.05);

\end{tikzpicture}}\]

It is immediate to see that both $\gamma'$ and $\gamma'^{-1}$ start directly with the arc $i$.

\begin{proposition}\label{prop tB}
Let $[\gamma]\in \pi_1^{\rm free}(\tSigma)$ be primitive , and $\lambda\in k^*$. The band object $\tB([\gamma],\lambda)$ associated to $[\gamma]$ and $\lambda$ in \cite{BZ} is defined to be the indecomposable object such that the module ${\rm H}\tB([\gamma],\lambda)$ is the cokernel of the following morphism:
\[\xymatrix{{\rm H}\tM(\flip i)\ar[rr]^{\bf i_{\gamma'}-\lambda\bf i_{{\gamma'}^{-1}}} & & {\rm H}\tM(\gamma')}.\]
Moreover we have $\tB([\gamma],\lambda)\simeq \tB([\beta],\mu)$ if and only if $([\gamma], \lambda)=([\beta], \mu)$ or $([\gamma], \lambda) = ([\beta^{-1}], \mu^{-1})$. 

The application $\tB$ can be extended to any element $([\gamma^n],\lambda)$ of $\pi_1^{\rm free}(\tSigma)\times k^*$ with $\gamma$ primitive, the object $\tB([\gamma^n],\lambda)$  being the (unique) non trivial $n$-extension of $\tB([\gamma],\lambda)$ with itself.
\end{proposition}

\begin{proof}
The result follows immediately from the description of band modules over gentle algebras by Butler and Ringel \cite{BR}. 

The last statement comes from the fact that if $\gamma'$ is defined as above, then $\gamma'^{-1}$ can be chosen as the curve $(\gamma^{-1})'$. Then we have the following commutative diagram:
 \[\xymatrix{{\rm H}\tM(\flip i)\ar[d]^{-\lambda}\ar[rr]^{\bf i_{\gamma'^{-1}}-\lambda\bf i_{{\gamma'}}} & & {\rm H}\tM(\gamma') \ar@{=}[d]  \\ {\rm H}\tM(\flip i)\ar[rr]^{\bf i_{\gamma'}-\lambda^{-1}\bf i_{{\gamma'}^{-1}}} & &{\rm H} \tM(\gamma') }.\] 
\end{proof}

As an automorphism of $\tSigma$, $\sigma$ acts also on the set $\pi_1^{\rm free}(\tSigma)$, and clearly restricts to the subset of primitive  loops. The next result is then an analogue of Lemma \ref{lemma sigma string}. 

\begin{lemma}\label{lemma sigma band}
Let $[\gamma]\in \pi_1^{\rm free}(\tSigma)$ non trivial and $\lambda\in k^*$. There is an isomorphism in $\cC_{\ttau}$:
$$\tB([\gamma],\lambda)^\sigma\simeq \tB(\sigma[\gamma],\lambda).$$
\end{lemma}

\begin{proof} The statement is an immediate consequence of Lemma \ref{lemma sigma string} and Proposition~\ref{prop tB} in the case where $\gamma$ is primitive . It follows then for any $\gamma^n$ by the unicity of extension of  $\tB([\gamma],\lambda)$ by itself.
\end{proof}

Combining the description of band objects of $\cC_{\ttau}$ with Corollary \ref{corollary functor cluster categories}, Proposition \ref{prop tB} and Lemma \ref{lemma sigma band} we obtain a complete description of the $\tSigma$-band objects of $\cC_\tau$.

Denote by $\sim$ the equivalence relation on the set $\pi_1^{\rm free}(\tSigma)\times k^*$ generated by $([\gamma],\lambda)\sim([\gamma^{-1}],\lambda^{-1})$.

\begin{corollary}\label{coro::band-objects-sigma-tilde}
The set of $\tSigma$-band objects of the category $\cC_{\tau}$ is in bijection with the union of the following sets:

\begin{itemize}
\item $\{[\gamma]\in \pi_1^{\rm free}(\tSigma)\ |\ \sigma[\gamma]\neq [\gamma],[\gamma^{-1}]\}\times k^*/\sim $
\item $ \{[\gamma]\in \pi_1^{\rm free}(\tSigma)^*\ |\ \sigma[\gamma]=[\gamma^{-1}]\}\times (k^*\setminus\{\pm 1\})/\sim$

\item $(\{[\gamma]\in \pi_1^{\rm free}(\tSigma)^*\ |\ \sigma[\gamma]= [\gamma]\}\times k^*/\sim)\times \{\pm 1\}$ 
\item $(\{[\gamma]\in \pi_1^{\rm free}(\tSigma)^*\ |\ \sigma[\gamma]=[\gamma^{-1}]\}\times \{\pm 1\}/\sim)\times \{\pm 1\}.$
\end{itemize}
\end{corollary}

\begin{proof}
The assignment $\tB$ induces a bijection between the set $\pi_1^{\rm free}(\tSigma)^*\times k^*/\sim$ and the set of band objects in $\cC_{\ttau}$.  By Lemma \ref{lemma sigma band}, we have $\tB([\gamma],\lambda)^\sigma\simeq \tB([\gamma],\lambda)$ if and only if $(\sigma[\gamma],\lambda)\sim ([\gamma],\lambda)$, that is if $\sigma[\gamma]=[\gamma]$ or if $\sigma[\gamma]=\gamma^{-1}$ and $\lambda=\lambda^{-1}$. The result is then a direct application of Corollary \ref{corollary functor cluster categories}.
\end{proof}

\begin{remark}
\begin{enumerate}
\item
  As we will see in the examples in Section \ref{section5}, these four sets are non empty in general. Moreover, since $\pi_1(\tSigma)$ is a free group, we cannot have $\sigma [\gamma]=[\gamma]=[\gamma^{-1}]$ (see Section \ref{section4}).
\item As shown in \cite{RR85} (see also \cite{GP}), the functor $F$ commutes with the Auslander-Reiten translation, therefore $F$ sends a connected component of the Auslander-Reiten quiver of $\cC_{\ttau}$ to union of connected components of the Auslander-Reiten quiver of $\cC_{\tau}$. For $([\tgamma],\lambda)\in \pi_1^{\rm free}(\tSigma)\times k^*$ with $\tgamma$ primitive, the indecomposable object $\tB([\tgamma],\lambda)$ is at the base of a homogenous tube in the Auslander-Reiten quiver of $\cC_{\ttau}$. Therefore if $\sigma[\tgamma]\neq [\tgamma]$ and if $(\sigma[\tgamma],\lambda)\neq ([\tgamma^{-1}],\lambda^{-1})$, then the image of the corresponding tube is also an homogenous tube.

But if  $\sigma[\tgamma]= [\tgamma]$ or if $(\sigma[\tgamma],\lambda)= ([\tgamma^{-1}],\lambda^{-1})$, then $\tB([\tgamma],\lambda)\simeq X_1\oplus X_2$ is the sum of two indecomposable objects satisfying $\tau(X_1\oplus X_2)=X_1\oplus X_2$. Either both $X_1$ and $X_2$ are stabilized by $\tau$ in which case the image of the tube is two homogenous tubes, or $\tau X_1=X_2$ in which case the image of the homogenous tube is a homogenous tube of rank $2$. Therefore non rigid objects can be sent to rigid objects by the functor $F$. Both cases may occur as shown in the examples in Section \ref{section5}, and it may be interesting to determine exactly when.

\end{enumerate}
\end{remark}

\section{Fundamental groupoids of $\widetilde{\Sigma}$ and $\Sigma$}\label{section4}

\subsection{Generalities}\label{sect::generalities-on-fundamental-groupoids}
In general, if $\Sigma$ is any surface and $E$ is a subset of $\Sigma$, we define its \emph{fundamental groupoid} $\pi_1(\Sigma, E)$
to be the groupoid whose objects are the points in $E$ and whose morphism set from a point $x$ to a point $y$ is the set of 
homotopy classes of oriented paths from $x$ to $y$.  If $E = \Sigma$, then we simply write $\pi_1(\Sigma)$.  
We write $\pi_1(\Sigma, E)^*$ for the set of all morphisms of $\pi_1(\Sigma, E)$ which are not identities.

If the surface has punctures and if $E$ does not contain any punctures, then we define the \emph{orbifold fundamental groupoid} $\pi_1^{\rm orb}(\Sigma, E)$ to be the quotient
of $\pi_1(\Sigma \setminus \cP, E)$ by the equivalence relation given by

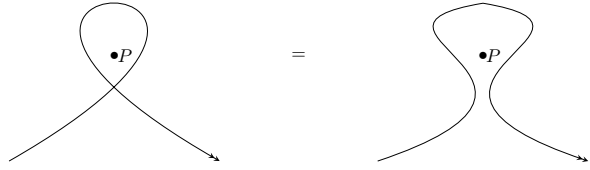
\[ 
 \scalebox{0.7}{
 
\begin{tikzpicture}[>=stealth,scale=1]

 \node (P) at (0 , 0) {$\bullet$};
 \node at (0.2 , 0) {$P$};
 
 \draw[->>] (-2,-2) .. controls (5,2) and (-5,2) .. (2,-2);
 
 \node at (3.5,0) {$=$};
 
\begin{scope}[xshift=7cm,yshift=0cm]
 
 \node at (0,0) {$\bullet$};
 \node at (0.2 , 0) {$P$};
 
 \draw[->>] (-2,-2) .. controls (2.5,-0.5) and (-3,0.5) .. (0,1) .. controls (3,0.5) and (-2.5,-0.5) .. (2,-2); 
\end{scope} 

\end{tikzpicture}
}
\]
where $P$ is a puncture.

We return to the situation where we have a covering of surfaces $p:\widetilde{\Sigma}\to \Sigma$.
A way to understand the fundamental groupoid is by way of a deformation retract of the surface.

For $\widetilde{\Sigma}$, let $\widetilde{\Gamma}_{\widetilde{\tau}}$ be the dual graph of $\tau$, that is, $\widetilde{\Gamma}_{\widetilde{\tau}}$ has one vertex $x_T$ for every triangle $T$ of $\widetilde{\tau}$
and two vertices are joined by an edge if the corresponding triangles share an edge.
For $\Sigma$, let $\Gamma_\tau$ be defined similarly.  Note that for each self-folded triangle in $\tau$, there is a loop in $\Gamma_\tau$.

The following is classical.
\begin{proposition}
 The graph $\widetilde{\Gamma}_{\widetilde{\tau}}$ is a deformation retract of $\widetilde{\Sigma}$, and
 the graph $\Gamma_\tau$ is a deformation retract of $\Sigma\setminus \cP$.
\end{proposition}

\begin{corollary}
 We have isomorphisms of groupoids $\pi_1(\widetilde{\Sigma}, \{ x_T \}_{T\in \widetilde{\tau}}) \cong \pi_1(\widetilde{\Gamma}_{\widetilde{\tau}}, \{ x_T \}_{T\in \widetilde{\tau}})$
 and $\pi_1(\Sigma\setminus \cP, \{ x_T \}_{T\in\tau}) \cong \pi_1(\Gamma_\tau, \{ x_T \}_{T\in\tau})$.
\end{corollary}
As a consequence, we can say that $\pi_1(\widetilde{\Sigma}, x_T)$ is a free group generated by elements 
\[
 a_1, b_1, \ldots, a_{\widetilde{g}}, b_{\widetilde{g}}, c_1, \ldots, c_{\widetilde{b}-1},
\]
where ${\widetilde{g}}$ is the genus of $\widetilde{\Sigma}$ and $\widetilde{b}$ is its number of boundary components.

Note that $G=\{1,\sigma\}$ acts on $\widetilde{\Gamma}_{\widetilde{\tau}}$ and that the quotient is (retracted to) $\Gamma_\tau$. 
This allows us to understand the orbifold groupoid by generators and relations as follows.
For every puncture $P$ of $\Sigma$, let $e_P$ be the corresponding loop in $\Gamma_\tau$.  Then
\[
 \pi_1^{\rm orb}(\Sigma, \{ x_T \}_{T\in\tau}) \cong \pi_1(\Gamma_\tau, \{ x_T \}_{T\in\tau}) / \langle e_P^2 \ | \ P\in \cP \rangle.
\]
Another consequence is that
\[
 \pi_1^{\rm orb}(\Sigma, x_T) \cong \bZ^{*2g+b-1} * (\bZ/2\bZ)^{*p},
\]
where $g$ is the genus of $\Sigma$, $b$ its number of boundary components and $p$ its number of punctures.
Denote the generators of the $\bZ^{*2g+b-1}$ part by $a_1, b_1, \ldots, a_{g}, b_{g}, c_1, \ldots, c_{b-1}$
and those of the $(\bZ/2\bZ)^{*p}$ part by $s_P, P\in \cP$.  Note that each $s_P$ is a conjugate (in the groupoid sense)
of $e_P$; say $s_P = w_Pe_P(w_P)^{-1}$.

In view of this, the following properties will be useful to us later.

\begin{lemma}\label{lemm::elements-of-order-2}
 \begin{enumerate}
  \item \cite[Corollary I.1.1]{Serre} Any element of order $2$ in $\pi_1^{\rm orb}(\Sigma, x_T)$ is conjugate to one of the $s_P$, with $P$ a puncture.    
  \item \cite[Theorem 2]{Serre} If $P$ is a puncture, then the only elements of $\pi_1^{\rm orb}(\Sigma, x_T)$ which commute with $s_P$ are $1$ and $s_P$.
 \end{enumerate} 
\end{lemma}

\begin{theorem}[Schreier]\cite[Theorem 5]{Serre}\label{theo::subgroup-of-free-group}
 Any subgroup of a free group is free.  In particular, if two elements of a free group commute, then they are both powers of a third element of the group.
\end{theorem}

We end this subsection with a general result comparing the fundamental groupoid of $\widetilde{\Sigma}$ with the orbifold fundamental
groupoid of $\Sigma$.

\begin{proposition}\label{prop::functor-between-groupoids}
 Let $p:\widetilde{\Sigma} \to \Sigma$ be as above.  Then $p$ induces a functor of groupoids
 \[
  \Phi: \pi_1(\widetilde{\Sigma}, \widetilde{\Sigma}\setminus p^{-1}(\cP)) \longrightarrow \pi_1^{\rm orb}(\Sigma, \Sigma \setminus \cP).
 \]
Moreover,
  \begin{enumerate}
   \item $\Phi$ is faithful;
   \item $\Phi$ is surjective on objects, and each object of $\pi_1^{\rm orb}(\Sigma, \Sigma \setminus \cP)$ has exactly two preimages.
  \end{enumerate}
\end{proposition}
\demo{ To prove that $\Phi$ is well-defined, consider the following diagram of functors:
\[
 \xymatrix{ \pi_1(\widetilde{\Sigma}\setminus p^{-1}(\cP), \widetilde{\Sigma}\setminus p^{-1}(\cP)) \ar[d]^{S}\ar[r]^{\phantom{xxxx} T} & \pi_1(\Sigma\setminus \cP, \Sigma \setminus \cP)\ar[d]^{U} \\
            \pi_1(\widetilde{\Sigma}, \widetilde{\Sigma}\setminus p^{-1}(\cP)) \ar@{.>}[r]^{\Phi} & \pi_1^{\rm orb}(\Sigma, \Sigma \setminus \cP).
 }
\]
The functors $S$ and $U$ are clearly surjective on morphisms and on objects.  
Let $\gamma$ be a morphism in $\pi_1(\widetilde{\Sigma}, \widetilde{\Sigma}\setminus p^{-1}(\cP))$.
We need to show that $p\circ \gamma$, viewed as an element of $\pi_1^{\rm orb}(\Sigma, \Sigma \setminus \cP)$, does not depend on a choice of
representative of the homotopy class of $\gamma$.  
First note that, up to homotopy, we can choose $\gamma$ so that it avoids $p^{-1}(\cP)$.
Thus $\gamma$ can be chosen so that it is a morphism in $\pi_1(\widetilde{\Sigma}\setminus p^{-1}(\cP), \widetilde{\Sigma}\setminus p^{-1}(\cP))$.
This choice, however, is \emph{not} independent of the homotopy class of $\gamma$.  
Indeed, in $\pi_1(\widetilde{\Sigma}, \widetilde{\Sigma}\setminus p^{-1}(\cP))$, a homotopy is allowed to go through $p^{-1}(\cP)$, while
it is not allowed to in $\pi_1(\widetilde{\Sigma}\setminus p^{-1}(\cP))$.
Thus $\pi_1(\widetilde{\Sigma}, \widetilde{\Sigma}\setminus p^{-1}(\cP))$ is the quotient 
of $\pi_1(\widetilde{\Sigma}\setminus p^{-1}(\cP), \widetilde{\Sigma}\setminus p^{-1}(\cP))$ by the relations given by these missing homotopies:

\[ 
 \scalebox{0.7}{
 
\begin{tikzpicture}[>=stealth,scale=1]

 \node at (0 , 0) {$\bullet$};
 \node at (0.2 , 0) {$\widetilde{P}$};
 
 \draw[dotted] (-1.8, 0) -- (1.8,0);
 \node at (1.5, 0.3) {$\tSigma^\pm$};
 \node at (1.5, -0.3) {$\tSigma^\mp$};
 
 \draw[->>] (0,2) .. controls (-1,1) and (-1,-1) .. (0,-2);
 
 \node at (3.5,0) {$=$};
 
\begin{scope}[xshift=7cm,yshift=0cm]
 
 \node at (0,0) {$\bullet$};
 \node at (0.2 , 0) {$\widetilde{P}$};
 
 \draw[dotted] (-1.8, 0) -- (1.8,0);
 \node at (1.5, 0.3) {$\tSigma^\pm$};
 \node at (1.5, -0.3) {$\tSigma^\mp$};
 
 \draw[->>] (0,2) .. controls (1,1) and (1,-1) .. (0,-2); 
\end{scope} 

\end{tikzpicture}
}
\]
where $\widetilde{P} \in p^{-1}(\cP)$.

Once we apply $T$, these new relations become, in $\pi_1(\Sigma\setminus \cP, \Sigma \setminus \cP)$, precisely the relations
given at the begining of the section to define $\pi_1^{\rm orb}(\Sigma, \Sigma \setminus \cP)$.  Thus $p \circ \gamma$, viewed in
this last groupoid, does not depend on the choice of representative of $\gamma$.  This shows that $\Phi$ is well-defined.

Statement (1) follows from the fact, shown above, that for any morphisms $\delta$ and $\delta'$ in $\pi_1(\widetilde{\Sigma}\setminus p^{-1}(\cP), \widetilde{\Sigma}\setminus p^{-1}(\cP))$
having the same starting and ending points, their images by $U\circ T$ are equal if and only if their images by $S$ are equal.

Statement (2) follows from the fact that $\sigma$ is an involution acting on $\widetilde{\Sigma}$ and that the set
of fixed points is exactly $p^{-1}(\cP)$. 
}

\subsection{Strings: from $\widetilde{\Sigma}$ to $\Sigma$}
Recall our setting from Section \ref{subs::construction}: we have a covering of surfaces $p:\widetilde{\Sigma}\to \Sigma$ given by an action of an 
involution $\sigma$ on $\widetilde{\Sigma}$.  This action defines an autoequivalence $\sigma: \pi_1(\widetilde{\Sigma}, \widetilde{\cM}) \to \pi_1(\widetilde{\Sigma}, \widetilde{\cM})$, and 
post-composition with $p$ defines a functor $\Phi|_{\widetilde{\cM}}: \pi_1(\widetilde{\Sigma}, \widetilde{\cM}) \to \pi_1^{\rm orb}(\Sigma, \cM)$
(see Proposition \ref{prop::functor-between-groupoids}).  

\begin{proposition}
 \begin{enumerate}
  \item The functor $\Phi|_{\widetilde{\cM}}$ is surjective on objects, and the preimage of each object of $\pi_1^{\rm orb}(\Sigma, \cM)$ contains exactly two objects.
  \item The functor $\Phi|_{\widetilde{\cM}}$ is surjective on morphisms.
        For any morphism $\widetilde{\gamma}$ in $\pi_1(\widetilde{\Sigma}, \widetilde{\cM})$, we have that $\Phi|_{\widetilde{\cM}}(\widetilde{\gamma}) = \Phi|_{\widetilde{\cM}}(\sigma\widetilde{\gamma})$.
        Moreover, $\sigma\widetilde{\gamma}$ and $\widetilde{\gamma}$ are the only preimages of $\Phi|_{\widetilde{\cM}}(\widetilde{\gamma})$.
 \end{enumerate}
\end{proposition}
\demo{ Point (1) follows from Proposition \ref{prop::functor-between-groupoids}(1).

In point (2), the fact that $\Phi|_{\widetilde{\cM}}(\sigma\widetilde{\gamma}) = \Phi|_{\widetilde{\cM}}(\widetilde{\gamma})$ follows from the equality $p = p\circ\sigma$.  
To prove that $\Phi$ is surjective on morphisms, let $\tau$ be the triangulation of $\Sigma$ used to build $\widetilde{\Sigma}$.  
Let $\gamma$ be a morphism in $\pi_1^{\rm orb}(\Sigma, \cM)$.  If $\gamma$ is in $\tau$, then it clearly lifts to a morphism in $\pi_1(\widetilde{\Sigma}, \widetilde{\cM})$.  

Assume that $\gamma$ is not in $\tau$.
We can assume that $\gamma$ crosses the arcs of $\tau$ transversally. 
Let $\tau_1, \ldots, \tau_m$ be the arcs of $\tau$ crossed by $\gamma$, in that order. 

If $\gamma$ never crosses the arcs of a self-folded triangle, then it admits two lifts, one of which is entirely contained in $\widetilde{\Sigma}^{+}$ and the other in $\widetilde{\Sigma}^{-}$.

If $\gamma$ does cross the arcs of self-folded triangles, then split $\gamma$ into sub-paths $\gamma_1, \ldots, \gamma_r$ such that 
each $\gamma_i$ does not cross the internal arcs of self-folded triangles except on their starting and ending points.  Then $\gamma_1$ lifts to a path $\gamma_1^+$ that is contained in $\widetilde{\Sigma}^{+}$,
$\gamma_2$ lifts to a path $\gamma_2^-$ whose starting point is the ending point of $\gamma_1^+$ and which is contained in $\widetilde{\Sigma}^{-}$, and so on.  Thus we construct a lift of $\gamma$ by $\Phi$.
Hence $\Phi$ is surjective on morphisms.

Finally, if $\widetilde{\gamma}$ is a lift of $\gamma$, then $\widetilde{\gamma}$ is determined by its starting point and the order in which it crosses the arcs of $\widetilde{\tau}$, by Lemma \ref{lemm::order-arc-triangulation}.
The only two possible starting points for $\widetilde{\gamma}$ are the two lifts of the starting point of $\gamma$, and the order in which $\widetilde{\gamma}$ crosses the arcs of $\widetilde{\tau}$
is determined by the order in which $\gamma$ crosses the arcs of $\tau$.  Hence $\widetilde{\gamma}$ and $\sigma\widetilde{\gamma}$ are the only lifts of $\gamma$.

}

\begin{corollary}\label{coro::bijection-string-objects}
 The $\widetilde{\Sigma}$-string objects of $\cC_\tau$ are in bijection with the set
 \[
   \big\{ \{\gamma, \gamma^{-1} \} \ | \ \gamma \in \pi_1^{\rm orb}(\Sigma, \cM), \gamma \neq \gamma^{-1} \big\} \cup (\{ \gamma \in \pi_1^{\rm orb}(\Sigma, \cM)^* \ | \ \gamma = \gamma^{-1}  \} \times \bZ/2\bZ).
 \]

\end{corollary}

In order to link this with classes of curves studied in \cite{FST} (where curves are allowed to have a puncture as an endpoint) and \cite{QZ}, we introduce the following definition.

\begin{definition}\label{defi::P'}
 For a puncture $P$ on $\Sigma$, let $P' = x_{T_P}$ be a point in the self-folded triangle $T_P$ around $P$.  Let $\cP'$ be the set of all such $P'$.
 We let $\pi_1^{\rm orb}(\Sigma, \cM; \cP')$ be the set of morphisms with source in $\cM$ and target in $\cP'$.
\end{definition}
Note that $\pi_1^{\rm orb}(\Sigma, \cM; \cP')$ is not a groupoid, since composition is not well-defined.  
The points in $\cP'$ must be thought of as points ``very close'' to the punctures, and the maps from $\cM$ to $\cP'$ as curves joining marked points on the boundary to punctures. But in contrast with \cite{FST}, a curve in $\pi_1^{\rm orb}(\Sigma,\cM;\cP')$ with endpoint $P'\in \cP'$ can reach $P'$ in two different ways, so the tagging of arcs in \cite{FST} is already encoded in $\pi_1^{\rm orb}(\Sigma,\cM;\cP')$.

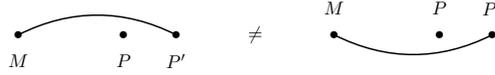
\[ 
 \scalebox{0.7}{
 
\begin{tikzpicture}[>=stealth,scale=1] 

\node at (0,0) {$\bullet$};
\node at (2,0) {$\bullet$};
\node at (3,0) {$\bullet$};
\node at (0,-0.5) {$M$};
\node at (2,-.5) {$P$};
\node at (3,-.5) {$P'$};

\draw[thick] (0,0)..controls (1,0.5) and (2,0.5)..(3,0);

\node at (4.5,0) {$\neq$};

\begin{scope}[xshift=6cm]
\node at (0,0) {$\bullet$};
\node at (2,0) {$\bullet$};
\node at (3,0) {$\bullet$};
\node at (0,0.5) {$M$};
\node at (2,.5) {$P$};
\node at (3,.5) {$P'$};

\draw[thick] (0,0)..controls (1,-0.5) and (2,-0.5)..(3,0);
\end{scope}
\end{tikzpicture}}\]

\begin{proposition}\label{prop::from-M-to-P}
 There is a natural bijection between the sets $\pi_1^{\rm orb}(\Sigma, \cM; \cP')$ and $(\{ \gamma \in \pi_1^{\rm orb}(\Sigma, \cM)^* \ | \ \gamma = \gamma^{-1}  \} \times \bZ/2\bZ)$.
\end{proposition}
\demo{Let $(\gamma, \varepsilon)$ be in the second set.  Then $\gamma^2=1$ implies that $\gamma = v^{-1}e_P v$ for some $v:M\to P'$, by Lemma \ref{lemm::elements-of-order-2}; moreover, we can assume that
the leftmost letter in the reduced expression of $v$ is not $e_P$.  
Then define the map from the second set in the statement to the first
\[
 (\gamma, \varepsilon) \longmapsto e_P^{\varepsilon} v.
\]
This map is well defined, since if $\gamma = w^{-1}e_Q w$, with $w:M\to Q'$ whose leftmost letter is its reduced expression is not $e_Q$, then $e_Q$ and $e_P$ are both conjugate to $s_P$, so $P=Q$, and $v^{-1}e_P v = w^{-1}e_P w$ implies that 
$vw^{-1}$ commutes with $e_P$, which is conjugate to $s_P$, so $vw^{-1}$ is either equal to $1$ or to $e_P$, by Lemma \ref{lemm::elements-of-order-2}.  Thus $v=e_Pw$ or $v=w$; the first case is impossible by our assumptions
on $v$ and $w$.  So $v=w$, and the map is well-defined.

Conversely, define a map from $\pi_1^{\rm orb}(\Sigma, \cM; \cP')$ to $(\{ \gamma \in \pi_1^{\rm orb}(\Sigma, \cM)^* \ | \ \gamma = \gamma^{-1}  \} \times \{\pm 1\})$ by
\[
 \big( v:M\to P' \big) \longmapsto (\gamma= v^{-1}e_P v, \varepsilon),
\]
where $\varepsilon$ is $0$ or $1$ depending on whether $e_P$ is not or is the leftmost letter in the reduced expression of $v$.
Then clearly the two maps defined above are mutually inverse.
}

\begin{corollary}\label{cor::map-string-to-tagged-arcs}
 The $\widetilde{\Sigma}$-string objects of $\cC_\tau$ are in bijection with the set
 \[
   \big\{ \{\gamma, \gamma^{-1} \} \ | \ \gamma \in \pi_1^{\rm orb}(\Sigma, \cM), \gamma \neq \gamma^{-1} \big\} \cup \pi_1^{\rm orb}(\Sigma, \cM; \cP').
 \]
\end{corollary}
\demo{Apply Corollary \ref{coro::bijection-string-objects} and Proposition \ref{prop::from-M-to-P}.
}

\subsection{Generalities on the free fundamental group}
The following general lemma on groupoids will be useful in dealing with the notion of free fundamental group.

\begin{lemma}\label{lemm::conjugacy-groupoids}
 Let $G$ be a connected groupoid.  Denote by $G(x,y)$ the set of morphisms from the object $x$ to the object $y$ in $G$.
 \begin{enumerate}
  \item\label{gamma} Any morphism $\gamma\in G(x,y)$ induces an isomorphism of groups 
    \[ 
      \bar{\gamma} : G(x,x) \to G(y,y): \alpha \mapsto \gamma\alpha\gamma^{-1}.
    \]
    
  \item\label{bij} For any object $z$ of $G$, denote by $G(z,z)_{cl}$ the set of conjugacy classes of the group $G(z,z)$.
    Then the bijection $G(x,x)_{cl}\to G(y,y)_{cl}$ induced by the morphism $\bar{\gamma}$ of (\ref{gamma}) does not depend on $\gamma$.

 \end{enumerate}
\end{lemma}
\demo{ Point (\ref{gamma}) is trivial.  Point (\ref{bij}) follows from the fact that if $\gamma, \delta \in G(x,y)$, then
       for any $\alpha\in G(x,x)$, we have that $\gamma\alpha\gamma^{-1}$ and $\delta\alpha\delta^{-1}$ are conjugate to each other
       (using the element $\gamma\delta^{-1}\in G(y,y)$).
}

\begin{definition}
 For any connected groupoid $G$, define $G^{\rm free}$ to be a set together with bijections $c_x: G(x,x)_{cl} \to G^{\rm free}$ for every object
 $x$ of $G$, such that the following hold: for all $\gamma\in G(x,y)$, we have that
 \[
  c_x = c_y\circ \bar{\gamma}.
 \]
 
 In the special case where $G$ is a fundamental groupoid $\pi_1(\Sigma)$, we call $\pi_1^{\rm free}(\Sigma)$ the \emph{free fundamental group}
 of $\Sigma$. We define the \emph{free orbifold fundamental group} similarly.
\end{definition}

\begin{remark}\label{rema::groupoids}
 The set $G^{\rm free}$ can be defined as a categorical colimit.  Let $C: G \to Set$ be the functor sending each $x$ to $G(x,x)_{cl}$ and each $\gamma$ to $\bar{\gamma}$. 
 Then $G^{\rm free}$ is the colimit of this functor in the category of sets.
\end{remark}

\begin{proposition}\label{prop::psi}
 \begin{itemize}
  \item The projection $p:\widetilde{\Sigma}\to \Sigma$ induces a map $\Psi: \pi_1^{\rm free}(\widetilde{\Sigma}) \to \pi_1^{\rm orb,free}(\Sigma)$.
  \item The action of $\sigma$ on $\widetilde{\Sigma}$ induces a map $\sigma : \pi_1^{\rm free}(\widetilde{\Sigma}) \to \pi_1^{\rm free}(\widetilde{\Sigma})$.  This map is an involution.
 \end{itemize} 
\end{proposition}
\demo{This is a consequence of Remark \ref{rema::groupoids} and of the universal property of colimits.
}

\subsection{Bands: from $\widetilde{\Sigma}$ to $\Sigma$}

We now turn to some properties of the map $\Psi: \pi_1^{\rm free}(\widetilde{\Sigma}) \to \pi_1^{\rm orb,free}(\Sigma)$ defined in Proposition \ref{prop::psi}.
We will need the following result on fundamental groups.

\begin{lemma}\label{lemm::orbifold-image}
 Let $x_0^+$ be a point in the interior of $\widetilde{\Sigma}^+$ and not on any arcs of $\widetilde{\tau}$, and let $x_0=p(x_0^+)$.
 Let $\Phi|_{x_0^+}: \pi_1(\widetilde{\Sigma}, x_0^+)\to \pi^{orb}_1(\Sigma, x_0)$ be the morphism induced by $p$ (see Proposition \ref{prop::functor-between-groupoids}).
 
 Then the image of $\Phi|_{x_0^+}$ contains exactly the paths that cross internal arcs of self-folded triangles an even number of times.
 In particular, for every puncture $P$, the element $s_P$ and its conjugates are not in the image of $\Phi|_{x_0^+}$.
\end{lemma}
\demo{ This follows from the fact that any path $\widetilde{\gamma}\in \pi_1(\widetilde{\Sigma}, x_0^+)$ starts and ends in $\widetilde{\Sigma}^+$,
so it passes between $\widetilde{\Sigma}^+$ and $\widetilde{\Sigma}^-$ and even number of times.  Applying $p$, these passages become exactly
the points where $p(\widetilde{\gamma})$ crosses and internal arc of a self-folded triangle.
}

Let us define the notion of primitivity for general free products of cyclic groups. This notion coincides with the usual notion of primitive elements in the literature for free groups. 

\begin{definition}\label{def::primitive}
Let $\gamma$ be an element in a group which is a free product of cyclic groups. We call $\gamma$ \emph{primitive} if $\gamma$ is torsionfree and is a generator of the maximal cyclic subgroup containing it. 
\end{definition}

Note that with this definition, if $\gamma\in \pi_1^{\rm orb}(\Sigma,x_0)$ satisfies $\gamma^2\neq 1$ then $\gamma$ is torsionfree, and so can be written in a unique way as a positive power of a primitive element.

\begin{proposition}\label{prop::preparation-bijection}
Let $\Psi: \pi_1^{\rm free}(\widetilde{\Sigma}) \to \pi_1^{\rm orb,free}(\Sigma)$ be as in Proposition \ref{prop::psi}.
 Let $[\widetilde{\gamma}] \in \pi_1^{\rm free}(\widetilde{\Sigma})$ be represented by a closed loop $\widetilde{\gamma} \in \pi_1(\widetilde{\Sigma}, x_0^+)$, with $x_0^+$ in the interior of $\widetilde{\Sigma}^+$.
 \begin{enumerate}
 \item If $\tgamma$ is primitive and if $\sigma[\tgamma]=[\tgamma]$, then $\Psi([\widetilde{\gamma}]) = [\alpha^2]$, with $\alpha \in \pi_1^{\rm orb}(\Sigma, x_0)$ primitive (where $x_0 = p(x_0^+)$).
 
 \item $\sigma[\tgamma]=[\tgamma^{-1}]$ if and only if $[\Phi(\tgamma)]=[\Phi(\tgamma)^{-1}]$, and in such situation $\tgamma$ is primitive if and only if $\Phi(\tgamma)$ is.
 
 \end{enumerate}

 Let $[\gamma]\in \pi_1^{\rm orb,free}(\Sigma)$ be represented by a primitive closed loop $\gamma\in \pi_1^{\rm orb}(\Sigma, x_0)$. 
 
 \begin{enumerate}
 \item[(3)] $[\gamma]$ is not in the image of $\Psi$ if, and only if, $[\gamma^2] = \Psi(\widetilde{\delta})$, with $\widetilde{\delta}$ primitive
  and $\sigma[\widetilde{\delta}] = [\widetilde{\delta}]$.
  
  \item[(4)] If $[\gamma]=[\gamma^{-1}]$, then $[\gamma]=\Psi[\tgamma]$ for some $\tgamma\in \pi_1(\tSigma,x_0^+)$ which satisfies $\sigma[\tgamma]=[\tgamma^{-1}].$
  \end{enumerate} 
  
  \end{proposition}

\demo{Let us first prove (1).  Assume that $\sigma [\widetilde{\gamma}] = [\widetilde{\gamma}]$.  
This is equivalent to the existence of a path $h:x_0^+ \to x_0^-$ such that $\sigma\widetilde{\gamma} = h\widetilde{\gamma} h^{-1}$.  Applying $\sigma$, we get that 
\[  
 \widetilde{\gamma} = \sigma h \cdot \sigma\widetilde{\gamma} \cdot (\sigma h )^{-1} = (\sigma h \cdot h) \widetilde{\gamma} (\sigma h \cdot h)^{-1}.
\]
Therefore $\sigma h \cdot h$ and $\widetilde{\gamma}$ commute in $\pi_1(\widetilde{\Sigma}, x_0^+)$.
Since this group is free, and since $\widetilde{\gamma}$ is primitive, this implies (by Theorem \ref{theo::subgroup-of-free-group})
that $\sigma h \cdot h$ is a power of $\widetilde{\gamma}$.  Write $\sigma h \cdot h = \widetilde{\gamma}^n$.

Now, recall that $\Phi|_{x_0^+}$ is a restriction of the functor $\Phi$ of Proposition \ref{prop::functor-between-groupoids}.
Then 
\[
  \Phi|_{x_0^+}(\tgamma) = \Phi(\widetilde{\gamma}) = \Phi(\sigma \widetilde{\gamma}) = \Phi(\sigma h \cdot \widetilde{\gamma} \cdot \sigma h^{-1})) = \Phi(h)\Phi(\widetilde{\gamma})\Phi(h)^{-1}.
\]
Thus $\Phi(h)$ and $\Phi(\widetilde{\gamma})$ commute.

Assume that $n=2m+1$ is odd.  Then $\widetilde{\gamma} = \sigma h \cdot h \widetilde{\gamma}^{-2m}$.  Applying $\Phi$, we get
\[
 \Phi(\widetilde{\gamma}) = \Phi(h)^2 \Phi(\widetilde{\gamma})^{-2m} = (\Phi(h) \Phi(\widetilde{\gamma})^{-m})^2, \\
\]
where the second equality holds since $\Phi(h)$ and $\Phi(\widetilde{\gamma})$ commute.  Setting $\alpha := \Phi(h) \Phi(\widetilde{\gamma})^{-m}$, we get that
\[
 \Psi([\widetilde{\gamma}]) = [\Phi(\widetilde{\gamma})] = [\alpha^2].
\]
Note that $\alpha$ is primitive.  Indeed, if $\alpha= \beta^r$, then $\beta^2$ lifts to a closed loop $\widetilde{\beta}$ such that $\widetilde{\gamma} = \widetilde{\beta}^r$, so that $r=1$ since $\widetilde{\gamma}$ is primitive.

Assume now that $n=2m$ is even.  Then, setting again $\alpha := \Phi(h) \Phi(\widetilde{\gamma})^{-m}$, we get
\[
 \alpha^2 = \Phi(h)^2\Phi(\widetilde{\gamma})^{-n} = \Phi(h)^2 \Phi(\sigma h \cdot h)^{-1} = 1.
\]

Thus, by Lemma \ref{lemm::elements-of-order-2}, $\alpha$ is conjugate to an element of the form $s_P$, with $P$ a puncture.

But $\alpha$ and $\Phi(\widetilde{\gamma})$ commute, since $\Phi(h)$ and $\Phi(\widetilde{\gamma})$ do.  
Hence, $\Phi(\widetilde{\gamma})$ is conjugate to an element which commutes with $s_P$.  The only such elements are $1$ and $s_P$, by Lemma \ref{lemm::elements-of-order-2}.
Moreover, by Lemma \ref{lemm::orbifold-image}, $\Phi(\widetilde{\gamma})$ is not conjugate to $s_P$.
Therefore $\Phi(\widetilde{\gamma}) = 1$, and since $\Phi$ is injective by Proposition \ref{prop::functor-between-groupoids}, we get that $\widetilde{\gamma} =1$, a contradition.
Thus, $n$ cannot be even.

This proves  (1).

\bigskip

We now prove (2). The first implication is clear. Let us prove the second.  Assume that $[\Phi(\tgamma)] = [\Phi(\tgamma)^{-1}]$.  
Write $\Phi(\widetilde{\gamma}) = z \Phi(\widetilde{\gamma})^{-1} z^{-1}$, with $z\in \pi_1^{\rm orb}(\Sigma, x_0)$.  We can lift $z$ to a path $\widetilde{z}$ on $\widetilde{\Sigma}$ so that
$z \Phi(\widetilde{\gamma})^{-1} z^{-1} = \Phi( \widetilde{z} \widetilde{\gamma}^{-1} \widetilde{z}^{-1})$.  

We can choose $\widetilde{z}$ so that one of two cases occurs: either $\widetilde{z}$ goes from $x_0^+$ to $x_0^-$, or it goes from $x_0^+$ to itself.
In the first case, we get $\Phi( \widetilde{z} \widetilde{\gamma}^{-1} \widetilde{z}^{-1}) = \Phi(\widetilde{\gamma}) = \Phi(\sigma(\widetilde{\gamma}))$, and
since $\Phi$ is faithful by Proposition \ref{prop::functor-between-groupoids}, then $\sigma(\widetilde{\gamma}) = \widetilde{z} \widetilde{\gamma}^{-1} \widetilde{z}^{-1}$, so 
$\sigma[\widetilde{\gamma}] = [\widetilde{\gamma}]$.  

The second case is impossible.  Indeed, in this case, $\Phi( \widetilde{z} \widetilde{\gamma}^{-1} \widetilde{z}^{-1}) = \Phi(\widetilde{\gamma})$
implies that $\widetilde{z} \widetilde{\gamma}^{-1} \widetilde{z}^{-1} = \widetilde{\gamma}$ (again by Proposition \ref{prop::functor-between-groupoids}), so $\widetilde{\gamma}$ is conjugate to its inverse.
But since $\pi_1(\widetilde{\Sigma}, x_0^+)$ is a free group, this is impossible (since $\widetilde{\gamma}$ is non-trivial).  

\medskip

If $\Phi(\tgamma)$ is primitve, then so is $\tgamma$. Conversely, assume that $\tgamma$ is primitive and that $\Phi(\tgamma)=\alpha^n$ for some $n\in \mathbb N^*$ and $\alpha$ in $\pi_1^{\rm orb}(\Sigma,x_0)$. Let $\widetilde{\alpha}\in \pi_1(\tSigma,\{x_0^+,x_0^-\})$ be a lift of $\alpha$. If $\widetilde{\alpha}$ is a closed loop, then $\Phi(\widetilde{\alpha}^n)=\Phi(\tgamma)$ so $n=1$ by faithfulness of $\Phi$ and primitivity of $\tgamma$. If $\widetilde{\alpha}$ is a path from $x_0^+$ to $x_0^-$, then two situations may occur depending on the parity of $n$. 
If $n$ is odd, then $\alpha^n$ crosses internal arcs of self-folded triangles an odd number of times, which is impossible by Lemma \ref{lemm::orbifold-image}, since $\alpha^n = \Phi(\tgamma)$.
If $n$ is even, we obtain $$\tgamma=(\sigma \widetilde{\alpha}\cdot\widetilde{\alpha})^{\frac{n}{2}},$$ so $\sigma[\tgamma]=[\tgamma]$. But by hypothesis $\sigma[\tgamma]= [\tgamma^{-1}]$ and a non trivial element is never conjugate to its inverse in a free group, a contradiction. Hence $\widetilde{\alpha}$ is a closed loop, and $\Phi(\tgamma)$ is primitive.

\bigskip

We now prove (3).

Let $\gamma \in \pi_1^{\rm orb}(\Sigma, x_0)$ be primitive and such that $[\gamma]$ is not in the image of $\Psi$.
By Lemma \ref{lemm::orbifold-image}, $\gamma$ crosses internal arcs of self-folded triangles of $\tau$ an odd number of times.
Therefore $\gamma^2$ crosses such arcs an even number of times, so $\gamma^2$ is in the image of $\Phi$, again by Lemma \ref{lemm::orbifold-image}.
Thus there exists $\widetilde{\delta}$ such that $\Phi(\widetilde{\delta}) = \gamma^2$, so $\Psi([\widetilde{\delta}]) = [\gamma^2]$.
Since $\gamma$ is primitive, then so is $\widetilde{\delta}$, and by part (1), $\sigma[\widetilde{\delta}] = [\widetilde{\delta}]$.

Let $\widetilde{\beta}\in \pi_1(\tSigma, \{x_0^+, x_0^-\})$ be a path such that $\Phi(\widetilde{\beta})=\gamma$, and such that $\widetilde{\beta}=x_0^{+}$. Its endpoint is necessarily $x_0^{-}$ since $[\gamma]$ is not in the image of $\Psi$.  Thus
$\Phi(\sigma\widetilde{\beta} \cdot \widetilde{\beta}) = \gamma^2 = \Phi(\widetilde{\delta})$, 
so 
\[ 
 \widetilde{\delta} = \sigma\widetilde{\beta} \cdot \widetilde{\beta} =\widetilde{\beta}^{-1} \cdot \sigma\widetilde{\delta} \cdot \widetilde{\beta},
\]
and so $\sigma[\widetilde{\delta}] = [\widetilde{\delta}]$.

\bigskip

Finally, we prove (4).  Assume $[\gamma]=[\gamma^{-1}]$ for some primitive element $\gamma$ in $\pi_1^{\rm orb}(\Sigma,x_0)$. First we show that $[\gamma]$ is in the image of $\Psi$. If not, by (3) there exists a loop $\widetilde{\delta}$ such that $\Phi(\widetilde{\delta})=\gamma^2$ and $\sigma[\widetilde{\delta}]=[\widetilde{\delta}]$. Since $[\gamma^2]=[\gamma^{-2}]$ so $\sigma[\widetilde{\delta}]=[\widetilde{\delta}^{-1}]$ by (2). But then $[\widetilde{\delta}]=[\widetilde{\delta}^{-1}]$ which is impossible. The rest follows from (2).

}

\begin{corollary}\label{coro::bijection-bands}
 \begin{enumerate}
  \item We have a bijection between the following sets:
    \begin{enumerate}
     \item \( \Big\{ \{[\widetilde{\gamma}], [\sigma\widetilde{\gamma}] \} \ | \ [\widetilde{\gamma}] \in \pi_1^{\rm free}(\widetilde{\Sigma}) 
              \textrm{ primitive and such that } [\sigma\widetilde{\gamma}] \neq [\widetilde{\gamma}], [\widetilde{\gamma}^{-1}] \Big\} ;
           \)
     \item \(
             \Big\{ [\gamma] \in \pi_1^{\rm orb,free}(\Sigma) \ | \ [\gamma] \ primitive\  \in \Ima \Psi, [\gamma] \neq [\gamma^{-1}] \Big\} .
           \)
    \end{enumerate}
    
  \item We have a bijection between the sets
    \begin{enumerate}
     \item \( \Big\{ \{[\widetilde{\gamma}], [\sigma\widetilde{\gamma}] \} \ | \ [\widetilde{\gamma}] \in \pi_1^{\rm free}(\widetilde{\Sigma}) 
              \textrm{ primitive and such that } [\sigma\widetilde{\gamma}]  = [\widetilde{\gamma}^{-1}] \Big\} ;
           \)
     \item \(
             \Big\{ [\gamma] \in \pi_1^{\rm orb,free}(\Sigma) \ | \ [\gamma] \textrm{  primitive and } [\gamma] = [\gamma^{-1}] \Big\} .
           \)
    \end{enumerate}
    
  \item We have a bijection between the sets
    \begin{enumerate}
     \item \( \Big\{ [\widetilde{\gamma}] \ | \ [\widetilde{\gamma}] \in \pi_1^{\rm free}(\widetilde{\Sigma}) 
              \textrm{ primitive and such that } [\sigma\widetilde{\gamma}]  = [\widetilde{\gamma}] \Big\} \times k^* \times \bZ/2\bZ;
           \)
     \item \(
             \Big\{ [\alpha] \in \pi_1^{\rm orb,free}(\Sigma) \ | \ [\alpha] \notin \Ima \Psi \textrm{ primitive} \Big\} \times k^*.
           \)
    \end{enumerate}

 \end{enumerate}
\end{corollary}
\demo{ The first two bijections are obtained by applying $\Psi$ and Proposition \ref{prop::preparation-bijection}.  The third one is defined by
\[
 ([\widetilde{\gamma}], \lambda, \pm 1) \longmapsto ([\alpha], \pm \lambda'),
\]
where $[\alpha^2] = \Psi(\widetilde{\gamma})$ as in Proposition \ref{prop::preparation-bijection} and $(\lambda')^2 = \lambda$.

}

\begin{corollary}\label{cor::bands-final}
The $\tSigma$-band objects in $\cC_\tau$ are in bijections with the union of the following sets:

\begin{itemize}
\item $\Big\{ [\gamma]\in \pi_1^{\rm orb,free}(\Sigma) |\ [\gamma]\neq [\gamma^{-1}]\Big\}\times k^*/\sim $
\item $\Big\{ [\gamma]\in \pi_1^{\rm orb,free}(\Sigma) |\ \gamma^2\neq 1 \textrm{ and }[\gamma]= [\gamma^{-1}]\Big\}\times k^*\backslash \{\pm 1\}/ \sim$
\item $\Big\{  [\gamma]\in \pi_1^{\rm orb,free}(\Sigma) |\ \gamma^2\neq 1 \textrm{ and }[\gamma]= [\gamma^{-1}]\Big\}\times (\bZ/2\bZ)^2$
\end{itemize}

where $([\gamma],\lambda)\sim ([\gamma^{-1}],\lambda)$. 

\end{corollary}

\demo{ The only thing to notice is that if $\gamma^2\neq 1$, then $\gamma$ is torsionfree in $\pi_1^{\rm orb}(\Sigma,x_0)$, so $\gamma$ can be written in a unique way as a positive power of a primitive element. Then the result is a direct consequence of  Corollary \ref{coro::bijection-bands} together with Corollary \ref{coro::band-objects-sigma-tilde}. }

\begin{remark}
\begin{enumerate}
\item
Note that even if the sets described in Corollary \ref{cor::bands-final} do no depend on the choice of the triangulation $\tau$, the bijection heavily depends on it and on the surface $\tSigma$ in particular, it is not clear at all what happen if we perform a sequence of mutations (see Appendix of \cite{CS16} for the unpunctured case). 
\item Note that if apply Corollary \ref{cor::map-string-to-tagged-arcs} with Corollary \ref{cor::bands-final} in the case where $\cP$ is empty, that is if $\Sigma$ does not have any puncture, then we recover the descritption of the objects of $\cC_{\tau}$ given in \cite{BZ}. Indeed we have $\pi_1^{\rm orb}(\Sigma)=\pi_1(\Sigma)$ and any $[\gamma]\neq 1$ in $\pi_1^{\rm free}(\Sigma)$ satisfies $[\gamma]\neq[\gamma^{-1}]$. 
\end{enumerate}
\end{remark}

\subsection{Link with (generalized) tagged arcs}

To form a link with \cite{FST} and \cite{QZ}, we would like to relate the last set in Corollary \ref{cor::bands-final} in terms of curves which may go from one puncture to another.  
Define $\cP'$ as in Definition \ref{defi::P'}.  We think of morphisms in $\pi_1^{\rm orb}(\Sigma, \cP')$ as tagged curves joining two punctures, since there are two ways to reach a point $P'\in \cP'$ in $\pi_1^{\rm orb}(\Sigma,\cP')$.

\begin{proposition}\label{prop::link-tagged-arcs}
Let $\gamma$ be a primitive element in $\pi_1^{\rm orb}(\Sigma,x_0)$ such that $[\gamma]=[\gamma^{-1}]$. Then there exist  $P,Q\in \cP$, and $v\in \pi_1^{\rm orb}(\Sigma,x_0)$ such that $[\gamma]=[s_Pvs_Qv^{-1}]$. Moreover $(P,Q,v)$ are uniquely determined by $[\gamma]$ up to the equivalence generated by $(P,Q,v)\sim (Q,P,s_Q^{\epsilon'}v^{-1}s_P^\epsilon )$, where $\epsilon,\epsilon'\in \{0,1\}$.
\end{proposition}

\demo{
By Proposition \ref{prop::preparation-bijection} (4), $\gamma=\Phi(\tgamma)$ for some $\tgamma$  in $\pi_1(\tSigma,x_0^+)$ satisfying $\sigma[\tgamma]=[\tgamma^{-1}]$.  Then
there exists a path $h:x_0^+ \to x_0^-$ such that $\sigma \widetilde{\gamma} = h\widetilde{\gamma}^{-1} h^{-1}$.
Applying $\sigma$, we get $\widetilde{\gamma} = \sigma h \cdot \sigma\widetilde{\gamma}^{-1} \cdot \sigma h^{-1}$, so that 
$\sigma \widetilde{\gamma} = (h\sigma h) \cdot \sigma \widetilde{\gamma} \cdot (h \sigma h )^{-1}$.

Since $\pi_1(\widetilde{\Sigma}, x_0^-)$ is a free abelian group, this is only possible if $h \sigma h$ is a power of $\sigma \widetilde{\gamma}$,
or equivalently, if $\sigma h \cdot h = \widetilde{\gamma}^n$ for some integer $n$.

But then $\widetilde{\gamma}^n = \sigma h \cdot (h \sigma h) \cdot \sigma h^{-1} = \sigma h \cdot \sigma \widetilde{\gamma}^n \cdot \sigma h^{-1}
  = (\sigma h \cdot h) \cdot \widetilde{\gamma}^{-n} \cdot (\sigma h \cdot h)^{-1}$.
This implies that $n=0$, since the only element conjugate to its inverse in a free abelian group is the identity.

Therefore, $\sigma h = h^{-1}$, and $\Phi(h)$ is an involution.  By Lemma \ref{lemm::elements-of-order-2}, we can write
$\Phi(h) = g_P s_Pg_P^{-1}$, where $P$ is a puncture.

Similarly, we note that $\widetilde{\gamma}^{-1} h^{-1}\cdot \sigma(\widetilde{\gamma}^{-1} h^{-1}) = 
\widetilde{\gamma}^{-1} h^{-1} h \widetilde{\gamma} h^{-1} \sigma h^{-1} = 1_{x_0^+}$.
Thus $\Phi(\widetilde{\gamma}^{-1} h^{-1})$ is an involution, and by Lemma \ref{lemm::elements-of-order-2}, we can write
$\Phi(\widetilde{\gamma}^{-1} h^{-1}) = g_Q s_Qg_Q^{-1}$, where $Q$ is a puncture.

Putting this together, we get that $$[\gamma]=[\Phi(\widetilde{\gamma})] = [\Phi(h)\Phi(\widetilde{\gamma}^{-1} h^{-1})] = [g_P s_Pg_P^{-1}g_Q s_Qg_Q^{-1}]=[s_Pvs_Qv^{-1}]$$ with $v=g_P^{-1}g_Q$.

When $h$ is fixed, $P$ (resp. $Q$) is uniquely determined, and $g_P$ (resp. $g_Q$) is determined up to multiplication by $s_P$ (resp. $s_Q$) on the right, so $v=g_P^{-1}g_Q$ is determined up to multiplication by $s_P$ on the left and $s_Q$ on the right. 

Let $h'$ be such that $\sigma \tgamma=h'\tgamma^{-1}h'^{-1}$, then $h'^{-1}h$ commute with $\tgamma$. Hence since $\tgamma$ is primitive, there exists $\ell\in \bZ$ such that $h'=h\tgamma^{\ell}$.  Thus we have 
$$\Phi(h')=\Phi(h) \gamma^\ell=(g_Ps_Pg_P^{-1})(g_Ps_Pg_P^{-1}g_Qs_Qg_Q^{-1})^\ell=(g_Qs_Qg_Q^{-1}g_Ps_Pg_P^{-1})^{\ell-1}(g_Qs_Qg_Q^{-1}).$$

We obtain then 
\[ \Phi(h')= \left\{\begin{array}{lcc} g'_Qs_Q(g'_Q)^{-1} & \textrm{with }g'_Q=\gamma^{-\frac{\ell-1}{2}}g_Q & \textrm{if }\ell\textrm{ is odd;}\\
g''_Ps_P(g''_P)^{-1} & \textrm{with }g''_P=\gamma^{-\frac{\ell-2}{2}}g_Qs_Qg_Q^{-1}g_P & \textrm{if }\ell\textrm{ is even and }\end{array}\right.\]

\[ \Phi(\tgamma^{-1}h')= \left\{\begin{array}{lcc} g'_Ps_P(g'_P)^{-1} & \textrm{with }g'_P=\gamma^{-\frac{\ell-1}{2}}g_Qs_Qg_Q^{_1}g_P & \textrm{if }\ell\textrm{ is odd;}\\
g''_Qs_Q(g''_Q)^{-1} & \textrm{with }g''_Q=\gamma^{-\frac{\ell}{2}}g_P & \textrm{if }\ell\textrm{ is even. }\end{array}\right.\]

Consequently if $\ell$ is odd, we obtain $v'=(g'_Q)^{-1}g'_P=(vs_Q)^{-1}$ and if $\ell$ is even we obtain $v''=(g''_P)^{-1}g''_Q=s_Pv$. That finishes the proof.
}

\begin{corollary}\label{cor::map-band-to-tagged-arcs}
There is an injective map from the set $$\Big\{[\gamma]\in \pi_1^{\rm orb,free}(\Sigma), \gamma \textrm{ primitive and }[\gamma]=[\gamma^{-1}]\Big\}\times (\bZ/2\bZ)^2$$ to the set $$\Big\{ \{w, w^{-1} \} \subset \pi_1^{\rm orb}(\Sigma,\cP'), w^2\neq 1\Big\}.$$

\end{corollary}

\demo{
Let $[\gamma]$ be in the first set, let $P$,  $Q$ and $v$ such that $[\gamma]=[s_Pvs_Qv^{-1}]$ as in Proposition \ref{prop::preparation-bijection}. We may chose $v$ such that the reduced expression of $v$ does not start with $s_P$ and does not end with $s_Q$. Then $v$ is uniquely determined up to the relation $v\sim v^{-1}$. Remember from subsection \ref{sect::generalities-on-fundamental-groupoids}, that we have $s_P=w_Pe_pw_P^{-1}$ (resp. $s_Q=w_Qe_Qw_Q^{-1}$) in $\pi_1^{\rm orb}(\Sigma,\{x_T\}_T)$, where $e_P$ (resp. $Q$) is the loop from $P'=x_{T_P}$ (resp.  $Q'=x_{T_Q}$) to itself in $\Gamma_\tau$ around the puncture $P$ (resp. $Q$). The above map is then defined to be the assignment $$([\gamma],\epsilon_1,\epsilon_2)\mapsto w=e_P^{\epsilon_1}w_P^{-1}vw_Qe_Q^{\epsilon_2}.$$ 

If $w=1_{x_{R'}}$ then we have $P=Q=R$ and $w_P^{-1}vw_P=1$, so $v=1_{x_0}$. Hence we have $[\gamma]=[s_Pvs_Pv^{-1}]=[1]$, which contradicts the fact that $\gamma$ is primitive.

If $w^2=1$ then we can compose $w$ with itself. This implies that $P=Q$ and thus $v^2=1$. Therefore we get
$$[\gamma]=[s_Pvs_Pv^{-1}]=[(s_P v)^2],$$ which contradicts the primitivity of $\gamma$.

An inverse of this map is clearly given by the assignment:
$$w\mapsto ([e_Pwe_Q{w}^{-1}], \epsilon_1,\epsilon_2)$$ where $Q'$ is the starting point of $w$, $P'$ its endpoint and where $\epsilon_1$ (resp. $\epsilon_2$) is $0$ or $1$ depending on wether the leftmost (resp. rightmost) letter in the reduced expression of $w$ is $e_P$ (resp. $e_Q$) or not.
This shows that the map $([\gamma], \epsilon_1,\epsilon_2)\mapsto w$ is injective.

}

\begin{remark}
\begin{enumerate}
\item
The inverse map can be understood as a map sending a curve from $Q$ to $P$ to a loop surrounding $Q$ and $P$ as in the following picture. 

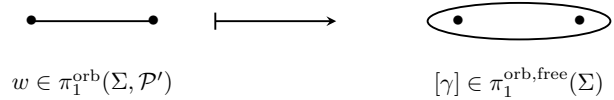
\[ 
 \scalebox{0.9}{
\begin{tikzpicture}[>=stealth,scale=0.9] 

\node at (0,0) {$\bullet$};
\node at (2,0) {$\bullet$};
\node at (1,-1) {$w\in \pi_1^{\rm orb}(\Sigma,\cP')$};
\draw[thick] (0,0)--(2,0);

\draw[thick,|->] (3,0)--(5,0);
\node at (7,0) {$\bullet$};
\node at (9,0) {$\bullet$};
 \draw[thick] (8,0) ellipse (1.5 and 0.3);
\node at (8,-1) {$[\gamma]\in \pi_1^{\rm orb,free}(\Sigma)$};

\end{tikzpicture}}\]
Therefore if $P\neq Q$ and if $w$ is a curve without selfintersection; that is, a tagged arc in the sense of \cite{FST}, then $\gamma$ can be chosen without selfintersection. So it cannot be a power greater than 2 of a primitive curve on $\Sigma$. Therefore, we obtain all tagged arcs from differents punctures. It is however not clear that all tagged arcs from $P$ to $P$ are in the image. The next item of the remark shows that the manipulations in the orbifold fundamental groups can be quite subtle.

\item
We can extend the inverse map to a map from $\pi_1^{\rm orb}(\Sigma,\cP')$ to $\pi_1^{\rm orb,free}(\Sigma)\times (\bZ/2\bZ)^2$. However, this map is not injective as shown in the following example. Let $P$ and $Q$ be two punctures and let $w$ an element in $\pi_1^{\rm orb}(\Sigma,\{P'\})$ of the form $he_Qh^{-1}$. Then its image will not be primitive and will be the same as the image of $h^{-1}e_Qh$ which is an element in $\pi_1^{\rm orb}(\Sigma,\{Q'\})$ so different from $w$. Indeed we have $$\begin{array}{rcl}[\gamma] &= &[e_P(he_Qh^{-1})e_P(he_Qh^{-1})^{-1}]\\ &= &[(e_Phe_Qh^{-1})^2]\\ &= &[e_Q(h^{-1}e_Ph)e_Q(h^{-1}e_Ph)^{-1}].\end{array}$$ This can also be interpreted topologically by the following picture:

\[ 
 \scalebox{0.9}{
\begin{tikzpicture}[>=stealth,scale=0.9] 

\node at (0,0) {$\bullet$};
\node at (2,0) {$\bullet$};
\node at (1,1) {$w=he_Ph^{-1}\in \pi_1^{\rm orb}(\Sigma,\cP')$};
\draw[thick] (0,0)..controls (1,1) and (2.5,0.5).. (2.5,0)..controls (2.5,-0.5) and (1,-1)..(0,0);

\begin{scope}[rotate=180, yshift=3cm, xshift=-2cm]
\node at (0,0) {$\bullet$};
\node at (2,0) {$\bullet$};
\node at (1,1) {$w'=h^{-1}e_Ph\in \pi_1^{\rm orb}(\Sigma,\cP')$};
\draw[thick] (0,0)..controls (1,1) and (2.5,0.5).. (2.5,0)..controls (2.5,-0.5) and (1,-1)..(0,0);

\end{scope}

\draw[thick,|->] (3,0)--(5,0);
\node at (7,0) {$\bullet$};
\node at (9,0) {$\bullet$};
\draw[thick] (6.5,0)..controls (6.5,1) and (9.5,1).. (9.5,0).. controls (9.5,-1) and (6.7,-0.8) ..(6.7,0)..controls (6.7,0.8) and (9.3,-0.8)..(9.3,0)..controls (9.3,0.8) and (6.5,-1).. (6.5,0);
\node at (8,1.2) {$[e_Qwe_Qw^{-1}]$};

\node at (8,-1.5) {$=$};

\draw[thick,|->] (3,-3)--(5,-3);

\begin{scope}[rotate=180,xshift=-16cm,yshift=3cm]

\node at (7,0) {$\bullet$};
\node at (9,0) {$\bullet$};
\draw[thick] (6.5,0)..controls (6.5,1) and (9.5,1).. (9.5,0).. controls (9.5,-1) and (6.7,-0.8) ..(6.7,0)..controls (6.7,0.8) and (9.3,-0.8)..(9.3,0)..controls (9.3,0.8) and (6.5,-1).. (6.5,0);
\node at (8,1.2) {$[e_Pw'e_Pw'^{-1}]$};
\end{scope}

\begin{scope}[yshift=-1.5cm]
\node at (11,0) {$\bullet$};
\node at (13,0) {$\bullet$};
\node at (10,1) {$=$};
\node at (10,-1){$=$};

\draw[thick] (10.5,0)..controls (10.5,1) and (13.5,1).. (13.5,0)..controls (13.5,-1) and (10.8,-1)..(10.8,0)..controls (10.8,0.8) and (13.2,0.7).. (13.2,0)..controls (13.2,-0.7) and (10.5,-0.7)..(10.5,0);

\node at (13,-1.2) {$[(e_Qh^{-1}e_Ph)^2] $ in $\pi_1^{\rm orb,free}(\Sigma)$};
\end{scope}

\end{tikzpicture}}\]

\item In \cite{QZ}, the authors assign to each ``generalized'' tagged arc an indecomposable object in the cluster category $\cC_\tau$.  We expect that their construction coincide with ours.  It does in the examples of the next section, and it should morally do in general since their construction uses the description of indecomposable modules over clannish algebras of Crawley-Boevey \cite{CB} that uses itself the construction of Reiten and Riedtmann \cite{RR85}. 

\end{enumerate}

\end{remark}

\section{Examples}\label{section5}
\subsection{Example of the disc with two punctures}

Let $(\Sigma, \cP, \cM)$ be a disc with two punctures $P$ and $Q$ and two marked points $A$ and $B$ on the boundary.
Consider the following triangulation $\tau$, together with its quiver with potential $(Q(\tau), S(\tau))$:

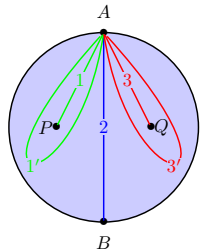
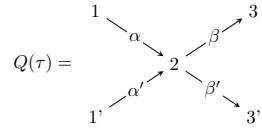
\[ 
 \scalebox{0.7}{
 
\begin{tikzpicture}[>=stealth,scale=1] 

 \draw[thick, fill=blue!20] (0,0) circle (1.8); 
 
 \node (P) at (-0.9 , 0) {$\bullet$};
 \node at (-1.1 , 0) {$P$};
 
 \node (Q) at (0.9 , 0) {$\bullet$};
 \node at (1.1 , 0) {$Q$};
 
 \node (A) at (0 , 1.8) {$\bullet$};
 \node at (0, 2.2) {$A$};
 
 \node (B) at (0 , -1.8) {$\bullet$};
 \node at (0, -2.2) {$B$};
 
 \draw[thick, blue, shorten >= -5pt, shorten <= -5pt] (A) -- (B) node[midway, fill=blue!20, inner sep=1pt]{$2$};
 
 \draw[thick, green, shorten >= -5pt, shorten <= -5pt] (A) -- (P) node[midway, fill=blue!20, inner sep=1pt]{$1$};
 \draw[thick, green, shorten >= -5pt, shorten <= -5pt] (A) .. controls (-3, -1.5) and (-0.5, -1.5) .. (A) node[midway, fill=blue!20, inner sep=0pt]{$1'$};
 
 \draw[thick, red, shorten >= -5pt, shorten <= -5pt] (A) -- (Q) node[midway, fill=blue!20, inner sep=1pt]{$3$};
 \draw[thick, red, shorten >= -5pt, shorten <= -5pt] (A) .. controls (3, -1.5) and (0.5, -1.5) .. (A) node[midway, fill=blue!20, inner sep=0pt]{$3'$};
 
\begin{scope}[xshift=5cm,yshift=1cm]
 
 \node at (0,0) {$Q(\tau) = $};
 
 \node (2) at (2.5 ,0) {2};
 \node (1) at (1, 1) {1};
 \node (1') at (1, -1) {1'};
 \node (3) at (4, 1) {3};
 \node (3') at (4, -1) {3'};
 
 \draw[->] (1) -- (2) node[midway, fill=white, inner sep=1pt]{$\alpha$};
 \draw[->] (1') -- (2) node[midway, fill=white, inner sep=1pt]{$\alpha'$};
 \draw[->] (2) -- (3) node[midway, fill=white, inner sep=1pt]{$\beta$};
 \draw[->] (2) -- (3') node[midway, fill=white, inner sep=1pt]{$\beta'$};
 
 \node at (0,-2.5) {$S(\tau) = 0$};
 
\end{scope}

\end{tikzpicture}

}
\]

The surface $\widetilde{\Sigma}$ is a cylinder, and its triangulation $\widetilde{\tau}$ and quiver with potiential are given by:

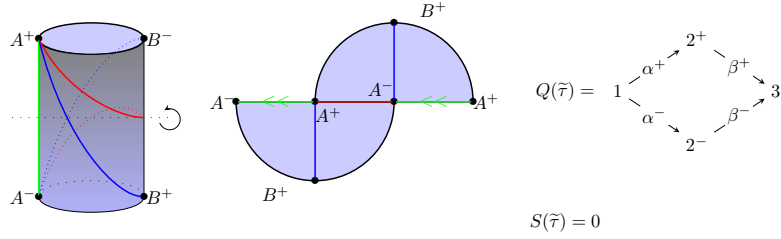
\[
 \scalebox{0.7}{

 \begin{tikzpicture}[>=stealth,scale=1] 
   \draw[thick, fill= blue!20] (0,0) ellipse (1 and 0.3);
   
   \shadedraw[bottom color=blue!30] (-1,0)..controls (-1,-0.4) and (1,-0.4)..(1,0) -- (1, -3) .. controls (1, -3.4) and (-1, -3.4) .. (-1, -3) -- (-1, 0);
   \draw[loosely dotted] (-1,-3)..controls (-1,-2.6) and (1,-2.6) ..(1,-3);
   
   \node (A+) at (-1,0) {$\bullet$};
   \node at (-1.3, 0) {$A^+$};
   
   \node (B-) at (1,0) {$\bullet$};
   \node at (1.3, 0) {$B^-$};
   
   \node (A-) at (-1,-3) {$\bullet$};
   \node at (-1.3, -3) {$A^-$};
   
   \node (B+) at (1,-3) {$\bullet$};
   \node at (1.3, -3) {$B^+$};
   
   \draw[loosely dotted] (-1.5, -1.5) -- (1.5, -1.5);
   \draw[->] (1.3,-1.6) arc (-160:120:0.2);
   
   \draw[thick, green, shorten >= -5pt, shorten <= -5pt] (A+) -- (A-);
   \draw[thick, blue, shorten >= -5pt, shorten <= -5pt] (A+) .. controls (-0.5, -2) and (0.5, -3) .. (B+);
   \draw[dotted, blue, shorten >= -5pt, shorten <= -5pt] (A-) .. controls (-0.5, -1) and (0.5, 0) .. (B-);
   
   \draw[thick, red,   shorten <= -5pt] (A+) .. controls (-0.5, -1) and (0.5, -1.5) .. (1, -1.5); 
   \draw[dotted, red, shorten >= -5pt] (1, -1.5) .. controls (0.5, -1) and (-0.5, -1.5) .. (A-);

  \begin{scope}[xshift=5cm,yshift=-1.2cm]
   
   \draw[thick, fill=blue!20]  (-2.25,0) -- (0.75 ,0)  arc(0:-180:1.5)  --cycle ;
   \draw[thick, fill=blue!20] (-0.75,0) -- (2.25 ,0) arc(0:180:1.5) --cycle;
   
   \node (A-1) at (-2.25, 0) {$\bullet$};
   \node at (-2.50, 0) {$A^-$};
   
   \node (A-2) at (0.75, 0) {$\bullet$};
   \node at (0.50, 0.25) {$A^-$};
   
   \node (A+1) at (-0.75, 0) {$\bullet$};
   \node at (-0.50, -0.25) {$A^+$};
   
   \node (A+2) at (2.25, 0) {$\bullet$};
   \node at (2.50, 0) {$A^+$};
   
   \node (B-) at (0.75, 1.5) {$\bullet$};
   \node at (1.5, 1.75) {$B^+$};
   
   \node (B+) at (-0.75, -1.5) {$\bullet$};
   \node at (-1.5, -1.75) {$B^+$};
   
   \draw[thick, green, shorten >= -5pt, shorten <= -5pt] (A+1) -- (A-1) node[midway]{$<<$};
   \draw[thick, green, shorten >= -5pt, shorten <= -5pt] (A+2) -- (A-2) node[midway]{$<<$};
   \draw[thick, blue, shorten >= -5pt, shorten <= -5pt] (A+1) -- (B+);
   \draw[thick, blue, shorten >= -5pt, shorten <= -5pt] (A-2) -- (B-);
   \draw[thick, red, shorten >= -5pt, shorten <= -5pt] (A+1) -- (A-2);
   
  \end{scope}

   \begin{scope}[xshift=9cm,yshift=-1cm]
 
 \node at (0,0) {$Q(\widetilde{\tau}) = $};
 
  \node (1) at (1,0) {$1$};
  \node (2+) at (2.5,1) {$2^+$};
  \node (2-) at (2.5,-1) {$2^-$};
  \node (3) at (4,0) {$3$};
  
  \draw[->] (1) -- (2+) node[midway, fill=white, inner sep=1pt]{$\alpha^+$};
  \draw[->] (1) -- (2-) node[midway, fill=white, inner sep=1pt]{$\alpha^-$};
  \draw[->] (2+) -- (3) node[midway, fill=white, inner sep=1pt]{$\beta^+$};
  \draw[->] (2-) -- (3) node[midway, fill=white, inner sep=1pt]{$\beta^-$};
 
 \node at (0,-2.5) {$S(\widetilde{\tau}) = 0$};
 
\end{scope}
   
 \end{tikzpicture}
 }
\]
Note that $Q(\tau)$ is of type $\widetilde{D}_4$, while $Q(\widetilde{\tau})$ is of type $\widetilde{A}_3$.
In this case, there is only one primitive curve (up to inverse) $\widetilde{\gamma} \in \pi_1^{\rm free}(\widetilde{\Sigma})$:
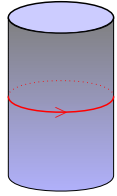
\[
 \scalebox{0.7}{
 \begin{tikzpicture}[>=stealth,scale=1]
 
    \draw[thick, fill= blue!20] (0,0) ellipse (1 and 0.3);
   
   \shadedraw[bottom color=blue!30] (-1,0)..controls (-1,-0.4) and (1,-0.4)..(1,0) -- (1, -3) .. controls (1, -3.4) and (-1, -3.4) .. (-1, -3) -- (-1, 0);
   
   \draw[thick, red] (-1,-1.5)..controls (-1,-1.9) and (1,-1.9)..(1,-1.5) node[midway]{$>$};
   \draw[dotted, red] (-1,-1.5)..controls (-1,-1.1) and (1,-1.1)..(1,-1.5);
   
 \end{tikzpicture}
 }
\]
Note that $\sigma(\widetilde{\gamma}) = \widetilde{\gamma}^{-1}$.  Thus, by Lemma \ref{lemma sigma band} and Proposition \ref{prop tB}, we have that

\[
 \tB([\gamma], \lambda)^\sigma \cong \tB([\gamma^{-1}], \lambda) \cong \tB([\gamma], \lambda^{-1}).
\]
Thus this band module is not fixed by $\sigma$, unless $\lambda = \pm 1$.  Note that the image of $\gamma$ on $\Sigma$ is a loop surrounding the two punctures :
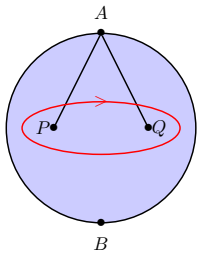
\[
\scalebox{0.7}{
 
\begin{tikzpicture}[>=stealth,scale=1] 

 \draw[thick, fill=blue!20] (0,0) circle (1.8);

 \node (P) at (-0.9 , 0) {$\bullet$};
 \node at (-1.1 , 0) {$P$};
 
 \node (Q) at (0.9 , 0) {$\bullet$};
 \node at (1.1 , 0) {$Q$};
 
 \node (A) at (0 , 1.8) {$\bullet$};
 \node at (0, 2.2) {$A$};
 
 \node (B) at (0 , -1.8) {$\bullet$};
 \node at (0, -2.2) {$B$};
 
 \draw[thick,  shorten >= -5pt, shorten <= -5pt] (A) -- (P);
 
 \draw[thick,  shorten >= -5pt, shorten <= -5pt] (A) -- (Q);

 \draw[thick, red] (0,0) ellipse (1.5 and 0.5); 
 
 \node[red] at (0, 0.5){$>$};
 
\end{tikzpicture}
 }
\]
As representations, we can express the modules ${\rm H}\tB([\gamma], \lambda)$ and ${\rm H}F\tB([\gamma], \lambda)$ as follows:
\[
 \scalebox{0.7}{
 
 \begin{tikzpicture}[>=stealth,scale=1] 
  
  \node at (-1,0) {${\rm H}\tB([\gamma], \lambda) = $};
 
  \node (1) at (1,0) {$k$};
  \node (2+) at (2.5,1) {$k$};
  \node (2-) at (2.5,-1) {$k$};
  \node (3) at (4,0) {$k$};
  
  \draw[->] (2+) -- (1) node[midway, above=5pt, fill=white, inner sep=1pt]{$\lambda$};
  \draw[->] (2-) -- (1) node[midway, below=5pt, fill=white, inner sep=1pt]{$1$};
  \draw[->] (3) -- (2+) node[midway, above=5pt, fill=white, inner sep=1pt]{$1$};
  \draw[->] (3) -- (2-) node[midway, below=5pt, fill=white, inner sep=1pt]{$1$};

  \begin{scope}[xshift=7cm,yshift=0cm]
  
  \node at (0,0) {${\rm H}F\tB([\gamma], \lambda) \simeq $};
 
 \node (2) at (2.5 ,0) {$k^2$};
 \node (1) at (1, 1) {$k$};
 \node (1') at (1, -1) {$k$};
 \node (3) at (4, 1) {$k$};
 \node (3') at (4, -1) {$k$};
 
 \draw[->] (2) -- (1) node[midway, above=5pt]{\footnotesize $\begin{pmatrix} 1 \\ 1 \end{pmatrix}$};
 \draw[->] (2) -- (1') node[midway, below=5pt]{\footnotesize $\begin{pmatrix} 1 \\ -1 \end{pmatrix}$};
 \draw[->] (3) -- (2) node[midway, above=5pt]{\footnotesize $\begin{pmatrix} \lambda & 1 \end{pmatrix}$};
 \draw[->] (3') -- (2) node[midway, below=5pt]{\footnotesize $\begin{pmatrix} \lambda & -1 \end{pmatrix}$};

 \end{scope}
 
 \end{tikzpicture}

 }
\]
The computation of the module on the right can be done using the triangles of Proposition \ref{prop tM}. We can see that, indeed, the module on the right is indecomposable, unless $\lambda = \pm 1$.  In this case, ${\rm H}F\tB([\gamma], \lambda) $ decomposes as 
\[
 \scalebox{0.7}{
 
 \begin{tikzpicture}[>=stealth,scale=1]   
  
  \node at (-1,0) {${\rm H}F\tB([\gamma], 1)\simeq $};
 
 \node (2) at (2.5 ,0) {$k$};
 \node (1) at (1, 1) {$k$};
 \node (1') at (1, -1) {$0$};
 \node (3) at (4, 1) {$k$};
 \node (3') at (4, -1) {$0$};
 
 \draw[->] (2) -- (1) node[midway, above=5pt]{$1$};
 \draw[->] (2) -- (1') node[midway, below=5pt]{$0$};
 \draw[->] (3) -- (2) node[midway, above=5pt]{$1$};
 \draw[->] (3') -- (2) node[midway, below=5pt]{$0$};
 
 \node at (5, 0) {$\oplus$};
 
 \node (2b) at (7.5 ,0) {$k$};
 \node (1b) at (6, 1) {$0$};
 \node (1'b) at (6, -1) {$k$};
 \node (3b) at (9, 1) {$0$};
 \node (3'b) at (9, -1) {$k$};
 
 \draw[->] (2b) -- (1b) node[midway, above=5pt]{$0$};
 \draw[->] (2b) -- (1'b) node[midway, below=5pt]{$1$};
 \draw[->] (3b) -- (2b) node[midway, above=5pt]{$0$};
 \draw[->] (3'b) -- (2b) node[midway, below=5pt]{$1$};

 \begin{scope}[yshift=-3.5cm]
   \node at (-1,0) {${\rm H}F\tB([\gamma], 1)\simeq $};
 
 \node (2) at (2.5 ,0) {$k$};
 \node (1) at (1, 1) {$0$};
 \node (1') at (1, -1) {$k$};
 \node (3) at (4, 1) {$k$};
 \node (3') at (4, -1) {$0$};
 
 \draw[->] (2) -- (1) node[midway, above=5pt]{$0$};
 \draw[->] (2) -- (1') node[midway, below=5pt]{$1$};
 \draw[->] (3) -- (2) node[midway, above=5pt]{$1$};
 \draw[->] (3') -- (2) node[midway, below=5pt]{$0$};
 
 \node at (5, 0) {$\oplus$};
 
 \node (2b) at (7.5 ,0) {$k$};
 \node (1b) at (6, 1) {$k$};
 \node (1'b) at (6, -1) {$0$};
 \node (3b) at (9, 1) {$0$};
 \node (3'b) at (9, -1) {$k$};
 
 \draw[->] (2b) -- (1b) node[midway, above=5pt]{$1$};
 \draw[->] (2b) -- (1'b) node[midway, below=5pt]{$0$};
 \draw[->] (3b) -- (2b) node[midway, above=5pt]{$0$};
 \draw[->] (3'b) -- (2b) node[midway, below=5pt]{$1$};
  
 \end{scope}

 \end{tikzpicture}

 }
\]
The four indecomposable modules appearing in the above pictures correspond to the four tagged arcs $w$, $e_P w$, $w e_Q$ and $e_Pwe_Q$ in $\pi_1^{\rm orb}(\Sigma,\cP')$ corresponding to $[\gamma]$ (see Corollary \ref{cor::map-band-to-tagged-arcs}):
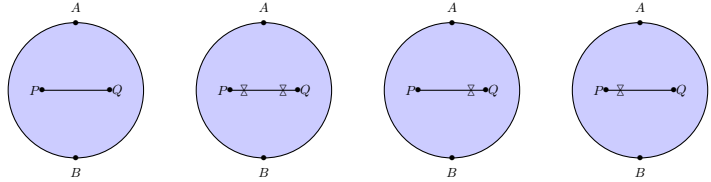
\[
 \scalebox{0.5}{
 
\begin{tikzpicture}[>=stealth,scale=1] 

 \draw[thick, fill=blue!20] (0,0) circle (1.8); 
 
 \node (P) at (-0.9 , 0) {$\bullet$};
 \node at (-1.1 , 0) {$P$};
 
 \node (Q) at (0.9 , 0) {$\bullet$};
 \node at (1.1 , 0) {$Q$};
 
 \node (A) at (0 , 1.8) {$\bullet$};
 \node at (0, 2.2) {$A$};
 
 \node (B) at (0 , -1.8) {$\bullet$};
 \node at (0, -2.2) {$B$};
 
 \draw[thick, shorten >= -5pt, shorten <= -5pt] (P) -- (Q) ;

 \begin{scope}[xshift=5cm]
  \draw[thick, fill=blue!20] (0,0) circle (1.8); 
 
 \node (P) at (-0.9 , 0) {$\bullet$};
 \node at (-1.1 , 0) {$P$};
 
 \node (Q) at (0.9 , 0) {$\bullet$};
 \node at (1.1 , 0) {$Q$};
 
 \node (A) at (0 , 1.8) {$\bullet$};
 \node at (0, 2.2) {$A$};
 
 \node (B) at (0 , -1.8) {$\bullet$};
 \node at (0, -2.2) {$B$};
 
 \draw[thick, shorten >= -5pt, shorten <= -5pt] (P) -- (Q) node[very near end, rotate=90]{$\bowtie$} node[very near start, rotate=90]{$\bowtie$};
 \end{scope}
 
 \begin{scope}[xshift=10cm]
  \draw[thick, fill=blue!20] (0,0) circle (1.8); 
 
 \node (P) at (-0.9 , 0) {$\bullet$};
 \node at (-1.1 , 0) {$P$};
 
 \node (Q) at (0.9 , 0) {$\bullet$};
 \node at (1.1 , 0) {$Q$};
 
 \node (A) at (0 , 1.8) {$\bullet$};
 \node at (0, 2.2) {$A$};
 
 \node (B) at (0 , -1.8) {$\bullet$};
 \node at (0, -2.2) {$B$};
 
 \draw[thick, shorten >= -5pt, shorten <= -5pt] (P) -- (Q) node[very near end, rotate=90]{$\bowtie$};
 \end{scope}
 
 \begin{scope}[xshift=15cm]
  \draw[thick, fill=blue!20] (0,0) circle (1.8); 
 
 \node (P) at (-0.9 , 0) {$\bullet$};
 \node at (-1.1 , 0) {$P$};
 
 \node (Q) at (0.9 , 0) {$\bullet$};
 \node at (1.1 , 0) {$Q$};
 
 \node (A) at (0 , 1.8) {$\bullet$};
 \node at (0, 2.2) {$A$};
 
 \node (B) at (0 , -1.8) {$\bullet$};
 \node at (0, -2.2) {$B$};
 
 \draw[thick, shorten >= -5pt, shorten <= -5pt] (P) -- (Q) node[very near start, rotate=90]{$\bowtie$};
 \end{scope}

 \end{tikzpicture}
 }
\]
In this example, we can also say something about the Auslander-Reiten components in $\cP(Q(\tau), S(\tau))$ and $\cP(Q(\widetilde{\tau}), S(\widetilde{\tau}))$, since they are well understood (see, for instance, \cite{Ringel}).
For instance, the module ${\rm H}\tB([\gamma], \lambda)$ lives in a homogenous tube $\cT_\lambda$.  
The above computations tell us that if $\lambda \neq \pm 1$, then $F\cT_\lambda = F\cT_{\lambda^{-1}}$ is a homogenous tube, while the additive closure of $F\cT_{\pm 1}$ is a tube of rank two.

\subsection{Example of the cylinder with one puncture}

Let $(\Sigma,\cP,\cM)$ be a cylinder with one puncture $\cP=\{P\}$ and three marked points $\{A,B,C\}$ on the boundary. Let $\tau$ be the following triangulation and let $(Q(\tau),S(\tau))$ be the associated quiver with potential:

  \[\scalebox{0.7}
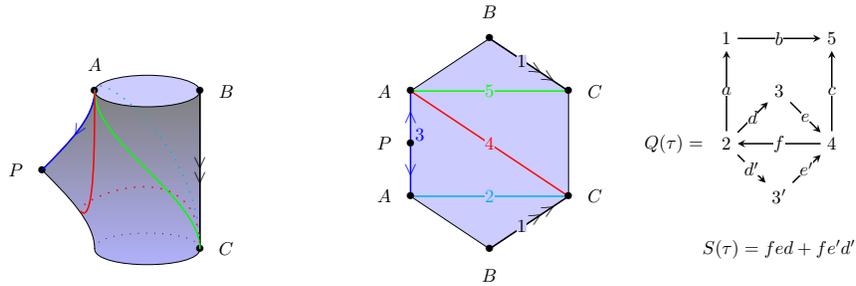
{
\begin{tikzpicture}[>=stealth,scale=1]

\shadedraw[bottom color=blue!30] (-1,0)..controls (-1,-0.4) and (1,-0.4)..(1,0)--(1,3)--(-1,3)..controls (-1,2.5) and (-1.5,2)..(-2,1.5)..controls (-1.5,1) and (-1,0.5)..(-1,0);
\draw[loosely dotted] (-1,0)..controls (-1,0.4) and (1,0.4) ..(1,0);
\draw[fill= blue!20] (0,3) ellipse (1 and 0.3);
\node at (1,0) {$\bullet$};
\node at (1,3) {$\bullet$};
\node at (-1,3) {$\bullet$};
\node at (-2,1.5) {$\bullet$};
\draw[thick,blue] (-1,3)..controls (-1,2.5) and (-1.5,2)..node [rotate=-120]{$>$}(-2,1.5);

\draw[thick] (1,3)--node [rotate=-90]{$>>$}(1,0);

\draw[thick,green] (-1,3)..controls (-1,2) and (1,1)..(1,0);
\draw[thick, loosely dotted, cyan] (-1,3)..controls (-1,3.5) and (1,2)..(1,0);
\draw[thick,red] (-1,3)..controls (-1,2) and (-1,0.5).. (-1.25,0.7);
\draw[thick, red, loosely dotted] (-1.25,0.7)..controls (-1.5,0.9) and (1,2)..(1,0);

\node at (1.5,0) {$C$};
\node at (1.5,3) {$B$};
\node at (-1,3.5) {$A$};
\node at (-2.5,1.5) {$P$};

\begin{scope}[xshift=5cm,yshift=1cm]

\draw[fill=blue!20](0,0)--(1.5,-1)--(3,0)--(3,2)--(1.5,3)--(0,2)--(0,0);
\draw[thick, blue](0,2)--node[blue, fill=blue!20, inner sep=0pt, xshift=5pt, yshift=5pt]{$3$}(0,0);
\draw[thick](1.5,-1)--node[fill=blue!20, inner sep=0pt,xshift=-4pt,yshift=-2pt]{$1$}(3,0);
\draw[thick] (3,2)--node[fill=blue!20, inner sep=0pt,xshift=-4pt,yshift=2pt]{$1$}(1.5,3);

\node[rotate=30] at (2.5,-0.33){$>>$};
\node[rotate=-30] at (2.5,2.33){$>>$};

\node[rotate=-90,blue] at (0,0.5) {$>$};
\node[rotate=90,blue] at (0,1.5){$>$};
\node at (-0.5,0) {$A$};
\node at (1.5,-1.5) {$B$};
\node at (-0.5,2) {$A$};

\node at (1.5,3.5) {$B$};
\node at (3.5,2) {$C$};
\node at (3.5,0){$C$};
\node at (-0.5,1){$P$};

\draw[thick,cyan] (3,0)--node[fill=blue!20, inner sep=0pt]{$2$}(0,0);
\draw[thick, red] (0,2)--node[fill=blue!20, inner sep=0pt]{$4$}(3,0);
\draw[thick,green] (0,2)--node[fill=blue!20, inner sep=0pt]{$5$}(3,2);

\node at (0,0) {$\bullet$};\node at (1.5,-1) {$\bullet$};\node at (3,0) {$\bullet$};\node at (3,2) {$\bullet$};\node at (1.5,3) {$\bullet$};\node at (0,2) {$\bullet$};
\node at (0,1) {$\bullet$};

\end{scope}

\begin{scope}[xshift=11cm, yshift=2cm]
\node at (-1,0) {$Q(\tau)=$};
\node (A2) at (0,0) {$2$};\node (A1) at (0,2) {$1$};\node (A3) at (1,1) {$3$};\node (A3') at (1,-1) {$3'$};\node (A4) at (2,0) {$4$};\node (A5) at (2,2) {$5$};
\draw[thick,->](A2)--node[fill=white, inner sep=0pt]{$a$}(A1);
\draw[thick,->](A1)--node[fill=white, inner sep=0pt]{$b$}(A5);
\draw[thick,->](A4)--node[fill=white, inner sep=0pt]{$c$}(A5);
\draw[thick,->](A2)--node[fill=white, inner sep=0pt]{$d$}(A3);
\draw[thick,->](A2)--node[fill=white, inner sep=0pt]{$d'$}(A3');
\draw[thick,->](A3)--node[fill=white, inner sep=0pt]{$e$}(A4);
\draw[thick,->](A3')--node[fill=white, inner sep=0pt]{$e'$}(A4);
\draw[thick,->](A4)--node[fill=white, inner sep=0pt]{$f$}(A2);
\node at (1,-2) {$S(\tau)=fed+fe'd'$};

\end{scope}

\end{tikzpicture}}
\]

The surface $\tSigma$ and the triangulation $\ttau$ are then given as follows:

  \[\scalebox{0.7}
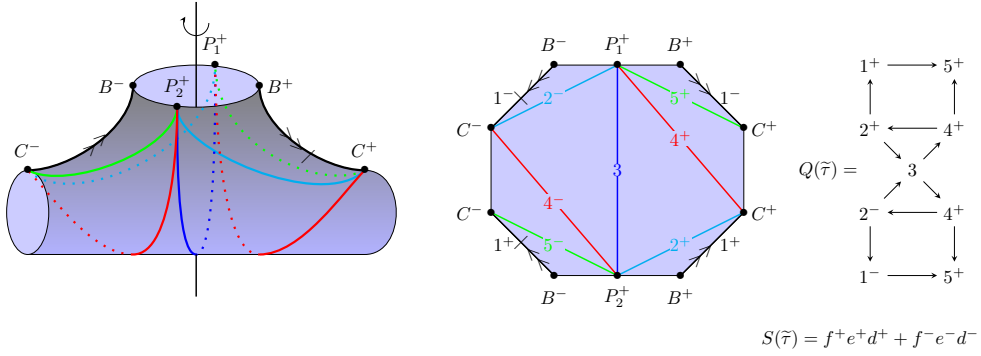
{
\begin{tikzpicture}[>=stealth,scale=0.8]

\shadedraw[bottom color=blue!30] (-4,0)--(4,0)..controls (5,0) and (5,2)..(4,2)..controls (3,2) and (1.5,3)..(1.5,4)--(-1.5,4)..controls (-1.5,3) and (-3,2)..(-4,2)--(-4,0);

\draw[fill=blue!20] (0,4) ellipse (1.5 and 0.5);\draw[fill=blue!20] (-4,1) ellipse (0.5 and 1);

\draw[very thick](4,2)..controls (3,2) and (1.5,3)..node[rotate=-45]{$\bf{>>|}$}(1.5,4);
\draw[very thick](-4,2)..controls (-3,2) and (-1.5,3)..node[rotate=45]{$\bf{>>}$}(-1.5,4);

\draw[very thick, green] (-4,2)..controls (-3,1.5) and (-0.5,2.5).. (-0.45,3.5);
\draw[very thick, ,loosely dotted, green] (4,2)..controls (3,1.5) and (0.5,2.5).. (0.45,4.5);
\draw[very thick, cyan] (4,2)..controls (3,1) and (-0.5,2.5).. (-0.45,3.5);
\draw[very thick, ,loosely dotted, cyan] (-4,2)..controls (-3,1) and (0.5,2.5).. (0.45,4.5);
\draw[very thick, blue] (-0.45,3.5)..controls (-.45,3) and (-0.5,0)..(0,0);
\draw[very thick, loosely dotted, blue] (0.45,4.5).. controls (0.45,4) and (0.5,0)..(0,0);

\draw[very thick, red] (-0.45,3.5)..controls (-0.45,3) and (-0.5,0)..(-1.5,0);
\draw[very thick, loosely dotted, red] (-1.5,0).. controls (-2.5,0) and (-3.5,1.5)..(-4,2);
\draw[very thick, red] (1.5,0).. controls (2.5,0) and (3.5,1.5)..(4,2);
\draw[very thick, red, loosely dotted] (0.45,4.5)..controls (0.45,4) and (0.5,0)..(1.5,0);

\draw[thick] (0,-1)--(0,0);
\draw[thick] (0,3.5)--(0,6);
\draw[->]  (0.3,5.5) arc (0:-210:0.3);

\node at (-4,2) {$\bullet$};
\node at (4,2) {$\bullet$};
\node at (-1.5,4) {$\bullet$};
\node at (1.5,4) {$\bullet$};
\node at (-0.45,3.5) {$\bullet$};
\node at (0.45,4.5) {$\bullet$};

\node at (-4,2.5) {$C^-$};
\node at (4,2.5) {$C^+$};
\node at (-0.45,4) {$P_2^+$};
\node at (0.45,5) {$P_1^+$};
\node at (-2,4) {$B^-$};
\node at (2,4) {$B^+$};

\begin{scope}[xshift=10cm, yshift=2cm]
\draw[fill=blue!20] (-3,-1)--(-3,1)--(-1.5,2.5)--(1.5,2.5)--(3,1)--(3,-1)--(1.5,-2.5)--(-1.5,-2.5)--(-3,-1);

\draw[thick] (-3,1)--node [xshift=-10pt]{$1^-$}(-1.5,2.5);
\draw[thick] (3,1)--node [xshift=10pt]{$1^-$}(1.5,2.5);
\draw[thick] (3,-1)--node [xshift=10pt]{$1^+$}(1.5,-2.5);
\draw[thick] (-3,-1)--node [xshift=-10pt]{$1^+$}(-1.5,-2.5);

\node[rotate=-135] at (-2,2) {$>>|$};
\node[rotate=135] at (-2,-2) {$>>|$};
\node[rotate=-45] at (2,2) {$>>$};
\node[rotate=45] at (2,-2) {$>>$};

\draw[thick, red] (-3,1)--node [fill=blue!20,inner sep=0pt]{$4^-$}(0,-2.5);
\draw[thick, red] (3,-1)--node [fill=blue!20,inner sep=0pt]{$4^+$}(0,2.5);
\draw[thick, green] (-3,-1)--node [fill=blue!20,inner sep=0pt]{$5^-$}(0,-2.5);
\draw[thick, green] (3,1)--node [fill=blue!20,inner sep=0pt]{$5^+$}(0,2.5);
\draw[thick, cyan] (-3,1)--node [fill=blue!20,inner sep=0pt]{$2^-$}(0,2.5);
\draw[thick, cyan] (3,-1)--node [fill=blue!20,inner sep=0pt]{$2^+$}(0,-2.5);
\draw[thick, blue] (0,-2.5)--node [fill=blue!20,inner sep=0pt]{$3$}(0,2.5);

\node at (-3,-1) {$\bullet$};\node at (-3,1) {$\bullet$};\node at (-1.5,2.5) {$\bullet$};\node at (1.5,2.5) {$\bullet$};\node at (3,1) {$\bullet$};\node at (3,-1) {$\bullet$};\node at (1.5,-2.5) {$\bullet$};\node at (-1.5,-2.5) {$\bullet$};\node at (0,-2.5) {$\bullet$};\node at (0,2.5) {$\bullet$};

\node at (-3.5,-1) {$C^-$};\node at (-3.5,1) {$C^-$};\node at (-1.5,3) {$B^-$};\node at (1.5,3) {$B^+$};\node at (3.5,1) {$C^+$};\node at (3.5,-1) {$C^+$};\node at (1.5,-3) {$B^+$};\node at (-1.5,-3) {$B^-$};\node at (0,-3) {$P^+_2$};\node at (0,3) {$P^+_1$};

\end{scope}

\begin{scope}[xshift=15cm, yshift=2cm]
\node at (0,0) {$Q(\ttau)=$};

\node (A1+) at (1,2.5) {$1^+$};\node (A1-) at (1,-2.5) {$1^-$};\node (A2+) at (1,1) {$2^+$};\node (A2-) at (1,-1) {$2^-$};\node (A3) at (2,0) {$3$};\node (A4+) at (3,1) {$4^+$};\node (A4-) at (3,-1) {$4^+$};\node (A5+) at (3,2.5) {$5^+$};\node (A5-) at (3,-2.5) {$5^+$};

\draw[->] (A2+)--(A1+);\draw[->] (A2-)--(A1-);
\draw[->] (A1+)--(A5+);\draw[->] (A1-)--(A5-);
\draw[->] (A4+)--(A5+);\draw[->] (A4-)--(A5-);
\draw[->] (A4+)--(A2+);\draw[->] (A4-)--(A2-);
\draw[->] (A2+)--(A3);\draw[->] (A2-)--(A3);
\draw[->] (A3)--(A4+);\draw[->] (A3)--(A4-);

\node at (1,-4) {$S(\ttau)=f^+e^+d^++f^-e^-d^-$};

\end{scope}

\end{tikzpicture}}\]

The automorphism $\sigma$ of the surface $\tSigma$ is represented here as a rotation of angle $\pi$ around a line intersecting the arc $3$.

\subsubsection{Case 1: $\widetilde{\gamma}\in\pi_1(\tSigma,\widetilde{M})$ and $\sigma \widetilde{\gamma}\neq \widetilde{\gamma}^{-1}$}

Let $\widetilde{\gamma}$ be the element in $\pi_1(\tSigma,\widetilde{M})$ intersecting the arc $3$ and then $2^-$. Then the curve $\sigma(\widetilde{\gamma})$ intersects the arcs $3$ and $2^+$, so is not homotopic to $\widetilde{\gamma}^{-1}$.

  \[\scalebox{0.7}{
\begin{tikzpicture}[>=stealth,scale=0.8]

\begin{scope}[xshift=0cm, yshift=0cm, scale=0.7]
\draw[fill=blue!20] (-3,-1)--(-3,1)--(-1.5,2.5)--(1.5,2.5)--(3,1)--(3,-1)--(1.5,-2.5)--(-1.5,-2.5)--(-3,-1);
\node[rotate=-135] at (-2,2) {$>>|$};
\node[rotate=135] at (-2,-2) {$>>|$};
\node[rotate=-45] at (2,2) {$>>$};
\node[rotate=45] at (2,-2) {$>>$};

\node at (-3,-1) {$\bullet$};\node at (-3,1) {$\bullet$};\node at (-1.5,2.5) {$\bullet$};\node at (1.5,2.5) {$\bullet$};\node at (3,1) {$\bullet$};\node at (3,-1) {$\bullet$};\node at (1.5,-2.5) {$\bullet$};\node at (-1.5,-2.5) {$\bullet$};\node at (0,-2.5) {$\bullet$};\node at (0,2.5) {$\bullet$};
\draw[thick, blue] (3,-1)--node [rotate=135]{$>$}(-1.5,2.5);
\draw[thick, red] (1.5,-2.5)--node [rotate=-45]{$>$}(-3,1);
\node[blue] at (1.5,1) {$\widetilde{\gamma}$};
\node[red] at (-1.5,-1) {$\sigma\widetilde{\gamma}$};

\end{scope}

\end{tikzpicture}}\]

Then we have
\[{\rm H}\tM(\widetilde{\gamma})=\begin{smallmatrix}3\\2^-\end{smallmatrix}\quad \textrm{and}\quad {\rm H}\tM(\sigma \widetilde{\gamma})=\begin{smallmatrix}3\\2^+\end{smallmatrix}\]
as module over the Jacobian algebra ${\rm Jac}((Q(\ttau),S(\ttau))$. An easy calculation gives 
\[{\rm H}F\tM(\widetilde{\gamma})={\rm H}F\tM(\sigma\widetilde{\gamma})=\begin{smallmatrix}3 & &3'\\ &2 &\end{smallmatrix}\]
which corresponds to the following curve $\gamma$ in $\pi_1^{\rm orb}(\Sigma,\cM)$:

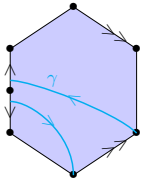
  \[\scalebox{0.7}{
\begin{tikzpicture}[>=stealth,scale=0.8]

\draw[fill=blue!20](0,0)--(1.5,-1)--(3,0)--(3,2)--(1.5,3)--(0,2)--(0,0);
\node at (0,0) {$\bullet$};\node at (1.5,-1) {$\bullet$};\node at (3,0) {$\bullet$};\node at (3,2) {$\bullet$};\node at (1.5,3) {$\bullet$};\node at (0,2) {$\bullet$};

\node[rotate=30] at (2.5,-0.33){$>>$};
\node[rotate=-30] at (2.5,2.33){$>>$};

\node at (0,1){$\bullet$};
\node[rotate=-90] at (0,0.5) {$>$};
\node[rotate=90] at (0,1.5){$>$};

\draw[thick, cyan] (1.5,-1).. controls (1.5,0) and (0.5,0.75)..node [cyan, rotate=-35]{$>$}(0,0.75);
\draw[thick, cyan] (3,0).. controls (2.5,0.5) and (0.5,1.25)..node [cyan, rotate=150]{$>$}(0,1.25);
\node[cyan] at (1,1.25) {$\gamma$};

\end{tikzpicture}}\]

\subsubsection{Case 2: $\widetilde{\gamma}\in \pi_1(\tSigma,\widetilde{\cM})$ and $\sigma \widetilde{\gamma}=\widetilde{\gamma}^{-1}$.}

Consider now the curve $\widetilde{\gamma}$ in $\pi_1(\tSigma,\widetilde{\cM})$ intersecting the arcs $2^+$, $3$ and $2^-$.   

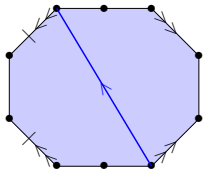
  \[\scalebox{0.7}{
\begin{tikzpicture}[>=stealth,scale=0.6]

\draw[fill=blue!20] (-3,-1)--(-3,1)--(-1.5,2.5)--(1.5,2.5)--(3,1)--(3,-1)--(1.5,-2.5)--(-1.5,-2.5)--(-3,-1);
\node[rotate=-135] at (-2,2) {$>>|$};
\node[rotate=135] at (-2,-2) {$>>|$};
\node[rotate=-45] at (2,2) {$>>$};
\node[rotate=45] at (2,-2) {$>>$};

\node at (-3,-1) {$\bullet$};\node at (-3,1) {$\bullet$};\node at (-1.5,2.5) {$\bullet$};\node at (1.5,2.5) {$\bullet$};\node at (3,1) {$\bullet$};\node at (3,-1) {$\bullet$};\node at (1.5,-2.5) {$\bullet$};\node at (-1.5,-2.5) {$\bullet$};\node at (0,-2.5) {$\bullet$};\node at (0,2.5) {$\bullet$};
\draw[thick, blue] (1.5,-2.5)--node [rotate=125]{$>$}(-1.5,2.5);

\end{tikzpicture}}\]
We obtain the following modules over the Jacobian algebras:

\[{\rm H}\tM(\widetilde(\gamma))=\begin{smallmatrix}&3&\\ 2^+ && 2^-\end{smallmatrix} \quad \textrm{and} \quad {\rm H}F\tM(\widetilde{\gamma})=\begin{smallmatrix}3\\2\end{smallmatrix}\oplus \begin{smallmatrix}3'\\ 2\end{smallmatrix}.\]

The image $\gamma$ of $\widetilde{\gamma}$ in $\pi_1^{\rm orb}(\Sigma, \cM)$ is the curve corresponding to the sequence of arcs $232$. It corresponds to the arcs $w$ and $e_Pw$ in $\pi_1^{\rm orb}(\Sigma,\cM;\cP')$ by Corollary~\ref{cor::map-string-to-tagged-arcs}. 

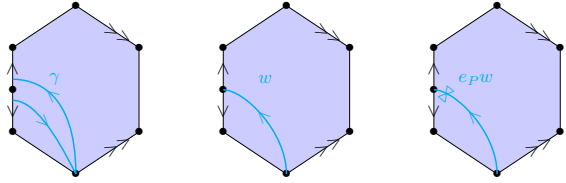
  \[\scalebox{0.7}{
\begin{tikzpicture}[>=stealth,scale=0.8]

\draw[fill=blue!20](0,0)--(1.5,-1)--(3,0)--(3,2)--(1.5,3)--(0,2)--(0,0);
\node at (0,0) {$\bullet$};\node at (1.5,-1) {$\bullet$};\node at (3,0) {$\bullet$};\node at (3,2) {$\bullet$};\node at (1.5,3) {$\bullet$};\node at (0,2) {$\bullet$};

\node[rotate=30] at (2.5,-0.33){$>>$};
\node[rotate=-30] at (2.5,2.33){$>>$};

\node at (0,1){$\bullet$};
\node[rotate=-90] at (0,0.5) {$>$};
\node[rotate=90] at (0,1.5){$>$};

\draw[thick, cyan] (1.5,-1).. controls (1,0) and (0.5,0.75)..node [cyan, rotate=-45]{$>$}(0,0.75);
\draw[thick, cyan] (1.5,-1).. controls (1.5,1) and (0.5,1.25)..node [cyan, rotate=140]{$>$}(0,1.25);
\node[cyan] at (1,1.25) {$\gamma$};

\begin{scope}[xshift=5cm]
\draw[fill=blue!20](0,0)--(1.5,-1)--(3,0)--(3,2)--(1.5,3)--(0,2)--(0,0);
\node at (0,0) {$\bullet$};\node at (1.5,-1) {$\bullet$};\node at (3,0) {$\bullet$};\node at (3,2) {$\bullet$};\node at (1.5,3) {$\bullet$};\node at (0,2) {$\bullet$};

\node[rotate=30] at (2.5,-0.33){$>>$};
\node[rotate=-30] at (2.5,2.33){$>>$};

\node at (0,1){$\bullet$};
\node[rotate=-90] at (0,0.5) {$>$};
\node[rotate=90] at (0,1.5){$>$};

\draw[thick, cyan] (1.5,-1).. controls (1.5,0) and (0.5,1)..node [cyan, rotate=-225]{$>$}(0,1);

\node[cyan] at (1,1.25) {$w$};

\end{scope}

\begin{scope}[xshift=10cm]

\draw[fill=blue!20](0,0)--(1.5,-1)--(3,0)--(3,2)--(1.5,3)--(0,2)--(0,0);
\node at (0,0) {$\bullet$};\node at (1.5,-1) {$\bullet$};\node at (3,0) {$\bullet$};\node at (3,2) {$\bullet$};\node at (1.5,3) {$\bullet$};\node at (0,2) {$\bullet$};

\node[rotate=30] at (2.5,-0.33){$>>$};
\node[rotate=-30] at (2.5,2.33){$>>$};

\node at (0,1){$\bullet$};
\node[rotate=-90] at (0,0.5) {$>$};
\node[rotate=90] at (0,1.5){$>$};

\draw[thick, cyan] (1.5,-1).. controls (1.5,0) and (0.5,1)..node [cyan, rotate=-225]{$>$}(0,1);
\node[cyan,rotate=50] at (0.3,0.9) {$\bowtie$};

\node[cyan] at (1,1.25) {$e_Pw$};

\end{scope}

\end{tikzpicture}}\]

\subsubsection{Case 3: $[\tgamma]\in\pi_1^{\rm free}(\tSigma)$ and $\sigma[\tgamma]\neq [\tgamma],[\tgamma^{-1}]$.}

Let $\tgamma$ be a closed curve corresponding the the sequence of arcs $1^-5^-4^-2^-$, and $\lambda\in k^*$. 

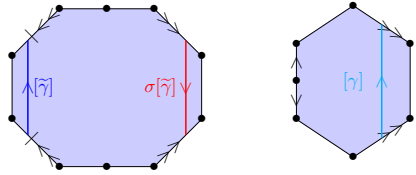
   \[\scalebox{0.7}{
\begin{tikzpicture}[>=stealth,scale=0.6]

\draw[fill=blue!20] (-3,-1)--(-3,1)--(-1.5,2.5)--(1.5,2.5)--(3,1)--(3,-1)--(1.5,-2.5)--(-1.5,-2.5)--(-3,-1);
\node[rotate=-135] at (-2,2) {$>>|$};
\node[rotate=135] at (-2,-2) {$>>|$};
\node[rotate=-45] at (2,2) {$>>$};
\node[rotate=45] at (2,-2) {$>>$};

\node at (-3,-1) {$\bullet$};\node at (-3,1) {$\bullet$};\node at (-1.5,2.5) {$\bullet$};\node at (1.5,2.5) {$\bullet$};\node at (3,1) {$\bullet$};\node at (3,-1) {$\bullet$};\node at (1.5,-2.5) {$\bullet$};\node at (-1.5,-2.5) {$\bullet$};\node at (0,-2.5) {$\bullet$};\node at (0,2.5) {$\bullet$};

\draw[thick,blue](-2.5,-1.5)--node[rotate=90]{$>$}(-2.5,1.5);
\node[blue] at (-2,0) {$[\tgamma]$};
\draw[thick,red] (2.5,-1.5)--node[rotate=-90]{$>$}(2.5,1.5);
\node[red] at (1.7,0) {$\sigma[\tgamma]$};
\begin{scope}[xshift=6cm, yshift=-1cm,scale=1.2]

\draw[fill=blue!20](0,0)--(1.5,-1)--(3,0)--(3,2)--(1.5,3)--(0,2)--(0,0);
\node at (0,0) {$\bullet$};\node at (1.5,-1) {$\bullet$};\node at (3,0) {$\bullet$};\node at (3,2) {$\bullet$};\node at (1.5,3) {$\bullet$};\node at (0,2) {$\bullet$};

\node[rotate=30] at (2.5,-0.33){$>>$};
\node[rotate=-30] at (2.5,2.33){$>>$};

\node at (0,1){$\bullet$};
\node[rotate=-90] at (0,0.5) {$>$};
\node[rotate=90] at (0,1.5){$>$};

\draw[thick, cyan] (2.25,-0.5)--node [cyan, rotate=90]{$>$}(2.25,2.5);

\node[cyan] at (1.5,1) {$[\gamma]$};

\end{scope}
\end{tikzpicture}}\]

The corresponding modules are given by the following representations:

   \[\scalebox{0.7}{
\begin{tikzpicture}[>=stealth,scale=0.6]

\begin{scope}[xshift=0cm, yshift=0cm]
\node at (-1,0) {${\rm H}\tB([\tgamma],\lambda)=$};

\node (A1+) at (1,2.5) {$0$};\node (A1-) at (1,-2.5) {$k$};\node (A2+) at (1,1) {$0$};\node (A2-) at (1,-1) {$k$};\node (A3) at (2,0) {$0$};\node (A4+) at (3,1) {$0$};\node (A4-) at (3,-1) {$k$};\node (A5+) at (3,2.5) {$0$};\node (A5-) at (3,-2.5) {$0$};

\draw[<-] (A2-)--node[xshift=-5pt]{$1$}(A1-);
\draw[<-] (A1-)--node[yshift=-5pt]{$1$}(A5-);
\draw[<-] (A4-)--node[xshift=5pt]{$1$}(A5-);
\draw[<-] (A4-)--node[yshift=5pt]{$\lambda$}(A2-);

\end{scope}

\begin{scope}[xshift=8cm, yshift=0cm]
\node at (-1,0) {${\rm H}\tB(\sigma[\tgamma],\lambda)=$};

\node (A1+) at (1,2.5) {$k$};\node (A1-) at (1,-2.5) {$0$};\node (A2+) at (1,1) {$k$};\node (A2-) at (1,-1) {$0$};\node (A3) at (2,0) {$0$};\node (A4+) at (3,1) {$k$};\node (A4-) at (3,-1) {$0$};\node (A5+) at (3,2.5) {$k$};\node (A5-) at (3,-2.5) {$0$};

\draw[<-] (A2+)--node[xshift=-5pt]{$1$}(A1+);
\draw[<-] (A1+)--node[yshift=5pt]{$1$}(A5+);
\draw[<-] (A4+)--node[xshift=5pt]{$1$}(A5+);
\draw[<-] (A4+)--node[yshift=5pt]{$\lambda$}(A2+);

\end{scope}

\begin{scope}[xshift=22cm, yshift=0cm]
\node at (-5,0) {${\rm H}F\tB([\tgamma],\lambda)={\rm H}F\tB(\sigma[\tgamma],\lambda)=$};
\node (A2) at (0,0) {$k$};\node (A1) at (0,2) {$k$};\node (A3) at (1,1) {$0$};\node (A3') at (1,-1) {$0$};\node (A4) at (2,0) {$k$};\node (A5) at (2,2) {$k$};
\draw[<-](A2)--node[xshift=-5pt]{$1$}(A1);
\draw[<-](A1)--node[yshift=5pt]{$1$}(A5);
\draw[<-](A4)--node[xshift=5pt]{$1$}(A5);
\draw[<-](A4)--node[yshift=5pt]{$\lambda$}(A2);

\end{scope}

\end{tikzpicture}}\]

\subsubsection{Case 4: $[\tgamma]\in\pi_1^{\rm free}(\tSigma)$ and $\sigma[\tgamma]=[\tgamma]$.}

Let $\tgamma$ be a closed curve on $\tSigma$ corresponding to the sequence of arcs $1^+5^+4^+32^-1^-5^-4^-32^+$. Then the image of $[\tgamma]$ in $\pi_1^{\rm orb,free}(\Sigma)$ is not primitive but is a square of a primitive element.

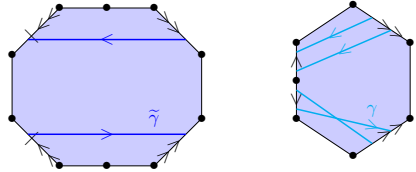
\[\scalebox{0.7}{
\begin{tikzpicture}[>=stealth,scale=0.6]

\draw[fill=blue!20] (-3,-1)--(-3,1)--(-1.5,2.5)--(1.5,2.5)--(3,1)--(3,-1)--(1.5,-2.5)--(-1.5,-2.5)--(-3,-1);
\node[rotate=-135] at (-2,2) {$>>|$};
\node[rotate=135] at (-2,-2) {$>>|$};
\node[rotate=-45] at (2,2) {$>>$};
\node[rotate=45] at (2,-2) {$>>$};

\node at (-3,-1) {$\bullet$};\node at (-3,1) {$\bullet$};\node at (-1.5,2.5) {$\bullet$};\node at (1.5,2.5) {$\bullet$};\node at (3,1) {$\bullet$};\node at (3,-1) {$\bullet$};\node at (1.5,-2.5) {$\bullet$};\node at (-1.5,-2.5) {$\bullet$};\node at (0,-2.5) {$\bullet$};\node at (0,2.5) {$\bullet$};

\draw[thick, blue] (-2.5,1.5)--node{$<$}(2.5,1.5);\draw[thick, blue] (-2.5,-1.5)--node{$>$}(2.5,-1.5);
\node[blue] at (1.5,-1) {$\tgamma$};

\begin{scope}[xshift=6cm, yshift=-1cm,scale=1.2]

\draw[fill=blue!20](0,0)--(1.5,-1)--(3,0)--(3,2)--(1.5,3)--(0,2)--(0,0);
\node at (0,0) {$\bullet$};\node at (1.5,-1) {$\bullet$};\node at (3,0) {$\bullet$};\node at (3,2) {$\bullet$};\node at (1.5,3) {$\bullet$};\node at (0,2) {$\bullet$};

\node[rotate=30] at (2.5,-0.33){$>>$};
\node[rotate=-30] at (2.5,2.33){$>>$};

\node at (0,1){$\bullet$};
\node[rotate=-90] at (0,0.5) {$>$};
\node[rotate=90] at (0,1.5){$>$};

\draw[thick, cyan] (2,2.66)--node[rotate=-145]{$>$}(0,1.75);
\draw[thick, cyan] (2.5,2.33)--node[rotate=-145]{$>$}(0,1.25);
\draw[thick, cyan] (2,-.66)--(0,0.75);
\draw[thick, cyan] (2.5,-0.33)--(0,0.25);
\node[cyan,rotate=-20] at (2,-0.2) {$>$};

\node[cyan] at (2,0.2) {$\gamma$};

\end{scope}

\end{tikzpicture}}\]

The corresponding representations are given as follows:

\[\scalebox{0.7}{
\begin{tikzpicture}[>=stealth,scale=0.8]

\begin{scope}[xshift=0cm, yshift=0cm]
\node at (-1,0) {${\rm H}\tB([\tgamma],\lambda)=$};

\node (A1+) at (1,2.5) {$k$};\node (A1-) at (1,-2.5) {$k$};\node (A2+) at (1,1) {$k$};\node (A2-) at (1,-1) {$k$};\node (A3) at (2,0) {$k^2$};\node (A4+) at (3,1) {$k$};\node (A4-) at (3,-1) {$k$};\node (A5+) at (3,2.5) {$k$};\node (A5-) at (3,-2.5) {$k$};

\draw[<-] (A2-)--node[xshift=-5pt]{$1$}(A1-);\draw[<-] (A2+)--node[xshift=5pt]{$1$}(A1+);
\draw[<-] (A1-)--node[yshift=-5pt]{$1$}(A5-);\draw[<-] (A1+)--node[yshift=5pt]{$\lambda$}(A5+);
\draw[<-] (A4-)--node[xshift=5pt]{$1$}(A5-);\draw[<-] (A4+)--node[xshift=-5pt]{$1$}(A5+);
\draw[->] (A4-)--node[yshift=10pt,xshift=5pt]{$\left(\begin{smallmatrix}0\\1\end{smallmatrix}\right)$}(A3);
\draw[->] (A4+)--node[yshift=10pt,xshift=-5pt]{$\left(\begin{smallmatrix}1\\0\end{smallmatrix}\right)$}(A3);
\draw[->] (A3)--node[xshift=-10pt,yshift=-5pt]{$\left(\begin{smallmatrix}0 & 1\end{smallmatrix}\right)$}(A2+);
\draw[->] (A3)--node[xshift=-10pt,yshift=5pt]{$\left(\begin{smallmatrix}1 & 0\end{smallmatrix}\right)$}(A2-);

\end{scope}

\begin{scope}[xshift=6cm, yshift=0cm]
\node at (-1,0) {$\simeq$};

\node (A1+) at (1,2.5) {$k$};\node (A1-) at (1,-2.5) {$k$};\node (A2+) at (1,1) {$k$};\node (A2-) at (1,-1) {$k$};\node (A3) at (2,0) {$k^2$};\node (A4+) at (3,1) {$k$};\node (A4-) at (3,-1) {$k$};\node (A5+) at (3,2.5) {$k$};\node (A5-) at (3,-2.5) {$k$};

\draw[<-] (A2-)--node[xshift=-5pt]{$1$}(A1-);\draw[<-] (A2+)--node[xshift=5pt]{$1$}(A1+);
\draw[<-] (A1-)--node[yshift=-5pt]{$\lambda$}(A5-);\draw[<-] (A1+)--node[yshift=5pt]{$1$}(A5+);
\draw[<-] (A4-)--node[xshift=5pt]{$1$}(A5-);\draw[<-] (A4+)--node[xshift=-5pt]{$1$}(A5+);
\draw[->] (A4-)--node[yshift=10pt,xshift=5pt]{$\left(\begin{smallmatrix}0\\1\end{smallmatrix}\right)$}(A3);
\draw[->] (A4+)--node[yshift=10pt,xshift=-5pt]{$\left(\begin{smallmatrix}1\\0\end{smallmatrix}\right)$}(A3);
\draw[->] (A3)--node[xshift=-10pt,yshift=-5pt]{$\left(\begin{smallmatrix}0 & 1\end{smallmatrix}\right)$}(A2+);
\draw[->] (A3)--node[xshift=-10pt,yshift=5pt]{$\left(\begin{smallmatrix}1 & 0\end{smallmatrix}\right)$}(A2-);

\node at (5,0) {$={\rm H}\tB([\tgamma],\lambda)^\sigma$};
\end{scope}

\end{tikzpicture}}\]
where each non appearing arrow is the zero map.

A direct computation (using for example the triangles given in Proposition \ref{prop tM}) gives: 
 
\[\scalebox{0.7}{
\begin{tikzpicture}[>=stealth,scale=0.8]
\begin{scope}[xshift=0cm, yshift=0cm]
\node at (-2,0) {${\rm H}F\tB([\tgamma],\lambda)=$};
\node (A2) at (0,0) {$k^2$};\node (A1) at (0,2) {$k^2$};\node (A3) at (1,1) {$k^2$};\node (A3') at (1,-1) {$k^2$};\node (A4) at (2,0) {$k^2$};\node (A5) at (2,2) {$k^2$};

\draw[<-](A2)--(A1);
\draw[<-](A1)--(A5);
\draw[<-](A4)--node[xshift=12pt,yshift=0pt]{$\left(\begin{smallmatrix}0 & \lambda\\ 1 & 0\end{smallmatrix}\right)$}(A5);
\draw[->](A4)--(A3);\draw[->](A4)--node[xshift=8pt]{$-1$}(A3');
\draw[<-](A2)--(A3);\draw[<-](A2)--(A3');

\end{scope}

\begin{scope}[xshift=5cm, yshift=0cm]
\node at (-1,0) {$\simeq$};
\node (A2) at (0,0) {$k$};\node (A1) at (0,2) {$k$};\node (A3) at (1,1) {$k$};\node (A3') at (1,-1) {$k$};\node (A4) at (2,0) {$k$};\node (A5) at (2,2) {$k$};

\draw[<-](A2)--(A1);
\draw[<-](A1)--(A5);
\draw[<-](A4)--node[xshift=5pt,yshift=0pt]{$\lambda'$}(A5);
\draw[->](A4)--(A3);\draw[->](A4)--node[xshift=8pt]{$-1$}(A3');
\draw[<-](A2)--(A3);\draw[<-](A2)--(A3');

\end{scope}

\begin{scope}[xshift=9cm, yshift=0cm]
\node at (-1,0) {$\oplus$};
\node (A2) at (0,0) {$k$};\node (A1) at (0,2) {$k$};\node (A3) at (1,1) {$k$};\node (A3') at (1,-1) {$k$};\node (A4) at (2,0) {$k$};\node (A5) at (2,2) {$k$};

\draw[<-](A2)--(A1);
\draw[<-](A1)--(A5);
\draw[<-](A4)--node[xshift=10pt,yshift=0pt]{$-\lambda'$}(A5);
\draw[->](A4)--(A3);\draw[->](A4)--node[xshift=8pt]{$-1$}(A3');
\draw[<-](A2)--(A3);\draw[<-](A2)--(A3');

\end{scope}

\end{tikzpicture}}\]
where each non appearing arrow is the zero map and each unlabeled arrow is the identity map, and where $\lambda'^2=\lambda$. Note that these two indecomposable summands are a $\sigma$-orbit.  

\subsubsection{Case 5: $[\tgamma]\in \pi_1^{\rm free}(\tSigma)$ and $\sigma[\tgamma]=[\tgamma^{-1}]$}

Let $\tgamma$ be a closed curve on $\tSigma$ corresponding to the sequence of arcs $4^-34^+5^+1^+2^+32^-1^-5^-$.

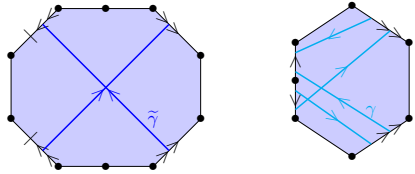
\[\scalebox{0.7}{
\begin{tikzpicture}[>=stealth,scale=0.6]

\draw[fill=blue!20] (-3,-1)--(-3,1)--(-1.5,2.5)--(1.5,2.5)--(3,1)--(3,-1)--(1.5,-2.5)--(-1.5,-2.5)--(-3,-1);
\node[rotate=-135] at (-2,2) {$>>|$};
\node[rotate=135] at (-2,-2) {$>>|$};
\node[rotate=-45] at (2,2) {$>>$};
\node[rotate=45] at (2,-2) {$>>$};

\node at (-3,-1) {$\bullet$};\node at (-3,1) {$\bullet$};\node at (-1.5,2.5) {$\bullet$};\node at (1.5,2.5) {$\bullet$};\node at (3,1) {$\bullet$};\node at (3,-1) {$\bullet$};\node at (1.5,-2.5) {$\bullet$};\node at (-1.5,-2.5) {$\bullet$};\node at (0,-2.5) {$\bullet$};\node at (0,2.5) {$\bullet$};

\draw[thick, blue] (-2,2)--node[rotate=-45, xshift=5pt,yshift=0pt]{$<$}(2,-2);\draw[thick, blue] (-2,-2)--node[rotate=45, xshift=-5pt,yshift=0pt]{$>$}(2,2);
\node[blue] at (1.5,-1) {$\tgamma$};

\begin{scope}[xshift=6cm, yshift=-1cm,scale=1.2]

\draw[fill=blue!20](0,0)--(1.5,-1)--(3,0)--(3,2)--(1.5,3)--(0,2)--(0,0);
\node at (0,0) {$\bullet$};\node at (1.5,-1) {$\bullet$};\node at (3,0) {$\bullet$};\node at (3,2) {$\bullet$};\node at (1.5,3) {$\bullet$};\node at (0,2) {$\bullet$};

\node[rotate=30] at (2.5,-0.33){$>>$};
\node[rotate=-30] at (2.5,2.33){$>>$};

\node at (0,1){$\bullet$};
\node[rotate=-90] at (0,0.5) {$>$};
\node[rotate=90] at (0,1.5){$>$};

\draw[thick, cyan] (2,2.66)--node[rotate=-145]{$>$}(0,1.75);
\draw[thick, cyan] (2.5,2.33)--node[rotate=-145]{$<$}(0,0.25);
\draw[thick, cyan] (2,-.66)--node[rotate=-35]{$>$} (0,0.75);
\draw[thick, cyan] (2.5,-0.33)--node[rotate=-35]{$<$}(0,1.25);

\node[cyan] at (2,0.2) {$\gamma$};

\end{scope}

\end{tikzpicture}}\]

A direct computation gives the following representations:

\[\scalebox{0.7}{
\begin{tikzpicture}[>=stealth,scale=1]

\begin{scope}[xshift=0cm, yshift=0cm]
\node at (-1,0) {${\rm H}\tB([\tgamma],\lambda)=$};

\node (A1+) at (1,2.5) {$k$};\node (A1-) at (1,-2.5) {$k$};\node (A2+) at (1,1) {$k$};\node (A2-) at (1,-1) {$k$};\node (A3) at (2,0) {$k^2$};\node (A4+) at (3,1) {$k$};\node (A4-) at (3,-1) {$k$};\node (A5+) at (3,2.5) {$k$};\node (A5-) at (3,-2.5) {$k$};

\draw[<-] (A2-)--(A1-);\draw[<-] (A2+)--(A1+);
\draw[<-] (A1-)--(A5-);\draw[<-] (A1+)--(A5+);
\draw[<-] (A4-)--(A5-);\draw[<-] (A4+)--(A5+);
\draw[->] (A4-)--node[yshift=10pt,xshift=5pt]{$\left(\begin{smallmatrix}1\\0\end{smallmatrix}\right)$}(A3);
\draw[->] (A4+)--node[yshift=10pt,xshift=-5pt]{$\left(\begin{smallmatrix}\lambda\\ 0\end{smallmatrix}\right)$}(A3);
\draw[->] (A3)--node[xshift=-10pt,yshift=-5pt]{$\left(\begin{smallmatrix}0 & 1\end{smallmatrix}\right)$}(A2+);
\draw[->] (A3)--node[xshift=-10pt,yshift=5pt]{$\left(\begin{smallmatrix}0 & 1\end{smallmatrix}\right)$}(A2-);

\end{scope}

\node at (5,0) {and};

\begin{scope}[xshift=10cm, yshift=0cm]
\node at (-2,0) {${\rm H}F\tB([\tgamma],\lambda)=$};
\node (A2) at (0,0) {$k^2$};\node (A1) at (0,2) {$k^2$};\node (A3) at (1,1) {$k^2$};\node (A3') at (1,-1) {$k^2$};\node (A4) at (2,0) {$k^2$};\node (A5) at (2,2) {$k^2$};

\draw[<-](A2)--(A1);
\draw[<-](A1)--(A5);
\draw[<-](A4)--(A5);
\draw[->](A4)--node[yshift=10pt,xshift=5pt,fill=white, inner sep=0pt] {$\left(\begin{smallmatrix}\lambda & 1\\ 0 & 0\end{smallmatrix}\right)$}(A3);
\draw[->](A4)--node[yshift=-10pt,xshift=10pt]{$\left(\begin{smallmatrix}-\lambda & 1\\ 0 & 0\end{smallmatrix}\right)$}(A3');
\draw[<-](A2)--node[yshift=10pt,xshift=-5pt,fill=white, inner sep=0pt]{$\left(\begin{smallmatrix}0 & 1 \\ 0 & 1\end{smallmatrix}\right)$}(A3);
\draw[<-](A2)--node[yshift=-10pt,xshift=-10pt] {$\left(\begin{smallmatrix}0 & 1\\ 0 & -1\end{smallmatrix}\right)$}(A3');

\end{scope}

\end{tikzpicture}}\]

One then checks that the module ${\rm H}F\tB([\tgamma],\lambda)$ is indecomposable for $\lambda\neq \pm1$. For $\lambda=1$ it decomposes into the sum of two indecomposable modules as follows:

\[\scalebox{0.7}{
\begin{tikzpicture}[>=stealth,scale=1]

\begin{scope}[xshift=0cm, yshift=0cm]
\node at (-2,0) {${\rm H}F\tB([\tgamma],1)\simeq$};
\node (A2) at (0,0) {$k$};\node (A1) at (0,2) {$k$};\node (A3) at (1,1) {$k^2$};\node (A3') at (1,-1) {$0$};\node (A4) at (2,0) {$k$};\node (A5) at (2,2) {$k$};

\draw[<-](A2)--(A1);
\draw[<-](A1)--(A5);
\draw[<-](A4)--(A5);
\draw[->](A4)--node[yshift=7pt,xshift=5pt,fill=white, inner sep=0pt] {$\left(\begin{smallmatrix}1\\ 0 \end{smallmatrix}\right)$}(A3);

\draw[<-](A2)--node[yshift=7pt,xshift=-5pt,fill=white, inner sep=0pt]{$\left(\begin{smallmatrix}0 & 1 \end{smallmatrix}\right)$}(A3);

\end{scope}

\begin{scope}[xshift=4cm, yshift=0cm]
\node at (-1,0) {$\oplus$};
\node (A2) at (0,0) {$k$};\node (A1) at (0,2) {$k$};\node (A3) at (1,1) {$0$};\node (A3') at (1,-1) {$k^2$};\node (A4) at (2,0) {$k$};\node (A5) at (2,2) {$k$};

\draw[<-](A2)--(A1);
\draw[<-](A1)--(A5);
\draw[<-](A4)--(A5);

\draw[->](A4)--node[yshift=-7pt,xshift=7pt]{$\left(\begin{smallmatrix} 1\\ 0 \end{smallmatrix}\right)$}(A3');

\draw[<-](A2)--node[yshift=-7pt,xshift=-7pt] {$\left(\begin{smallmatrix}0 & 1\end{smallmatrix}\right)$}(A3');

\end{scope}

\end{tikzpicture}}\]

These two objects correspond to the tagged arcs $w$ and $e_Pw e_P$ in $\pi_1^{\rm orb}(\Sigma,\cP')$ given by Corollary \ref{cor::map-band-to-tagged-arcs}.

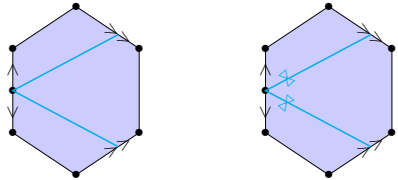
\[\scalebox{0.7}{
\begin{tikzpicture}[>=stealth,scale=0.8]
\begin{scope}[xshift=0cm, yshift=0cm,scale=1]

\draw[fill=blue!20](0,0)--(1.5,-1)--(3,0)--(3,2)--(1.5,3)--(0,2)--(0,0);
\node at (0,0) {$\bullet$};\node at (1.5,-1) {$\bullet$};\node at (3,0) {$\bullet$};\node at (3,2) {$\bullet$};\node at (1.5,3) {$\bullet$};\node at (0,2) {$\bullet$};

\node[rotate=30] at (2.5,-0.33){$>>$};
\node[rotate=-30] at (2.5,2.33){$>>$};

\node at (0,1){$\bullet$};
\node[rotate=-90] at (0,0.5) {$>$};
\node[rotate=90] at (0,1.5){$>$};

\draw[thick,cyan] (2.5,2.33)--(0,1)--(2.5,-0.33);

\end{scope}

\begin{scope}[xshift=6cm, yshift=0cm,scale=1]

\draw[fill=blue!20](0,0)--(1.5,-1)--(3,0)--(3,2)--(1.5,3)--(0,2)--(0,0);
\node at (0,0) {$\bullet$};\node at (1.5,-1) {$\bullet$};\node at (3,0) {$\bullet$};\node at (3,2) {$\bullet$};\node at (1.5,3) {$\bullet$};\node at (0,2) {$\bullet$};

\node[rotate=30] at (2.5,-0.33){$>>$};
\node[rotate=-30] at (2.5,2.33){$>>$};

\node at (0,1){$\bullet$};
\node[rotate=-90] at (0,0.5) {$>$};
\node[rotate=90] at (0,1.5){$>$};

\draw[thick,cyan] (2.5,2.33)--(0,1)--(2.5,-0.33);

\node[rotate=50,cyan] at (0.5,0.7) {$\bowtie$};
\node[rotate=-50,cyan] at (0.5,1.3) {$\bowtie$};

\end{scope}
\end{tikzpicture}}\]

For $\lambda=-1$ we obtain the following decomposition:

\[\scalebox{0.7}{
\begin{tikzpicture}[>=stealth,scale=0.8]

\begin{scope}[xshift=0cm, yshift=0cm]
\node at (-2,0) {${\rm H}F\tB([\tgamma],-1)\simeq$};
\node (A2) at (0,0) {$k$};\node (A1) at (0,2) {$k$};\node (A3) at (1,1) {$k$};\node (A3') at (1,-1) {$k$};\node (A4) at (2,0) {$k$};\node (A5) at (2,2) {$k$};

\draw[<-](A2)--(A1);
\draw[<-](A1)--(A5);
\draw[<-](A4)--(A5);
\draw[->](A4)--(A3);
\draw[<-](A2)--(A3');

\end{scope}

\begin{scope}[xshift=4cm, yshift=0cm]
\node at (-1,0) {$\oplus$};
\node (A2) at (0,0) {$k$};\node (A1) at (0,2) {$k$};\node (A3) at (1,1) {$k$};\node (A3') at (1,-1) {$k$};\node (A4) at (2,0) {$k$};\node (A5) at (2,2) {$k$};

\draw[<-](A2)--(A1);
\draw[<-](A1)--(A5);
\draw[<-](A4)--(A5);

\draw[->](A4)--(A3');

\draw[<-](A2)--(A3);

\end{scope}

\end{tikzpicture}}\]

These two objects correspond to the tagged arcs $e_Pw$ and $w e_P$ in $\pi_1^{\rm orb}(\Sigma,\cP')$ defined in Corollary \ref{cor::map-band-to-tagged-arcs}.

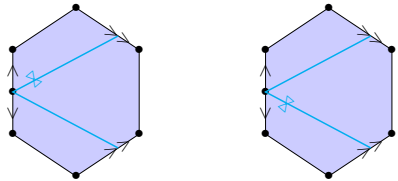
\[\scalebox{0.7}{
\begin{tikzpicture}[>=stealth,scale=0.8]
\begin{scope}[xshift=0cm, yshift=0cm,scale=1]

\draw[fill=blue!20](0,0)--(1.5,-1)--(3,0)--(3,2)--(1.5,3)--(0,2)--(0,0);
\node at (0,0) {$\bullet$};\node at (1.5,-1) {$\bullet$};\node at (3,0) {$\bullet$};\node at (3,2) {$\bullet$};\node at (1.5,3) {$\bullet$};\node at (0,2) {$\bullet$};

\node[rotate=30] at (2.5,-0.33){$>>$};
\node[rotate=-30] at (2.5,2.33){$>>$};

\node at (0,1){$\bullet$};
\node[rotate=-90] at (0,0.5) {$>$};
\node[rotate=90] at (0,1.5){$>$};

\draw[thick,cyan] (2.5,2.33)--(0,1)--(2.5,-0.33);
\node[rotate=-50,cyan] at (0.5,1.3) {$\bowtie$};

\end{scope}

\begin{scope}[xshift=6cm, yshift=0cm,scale=1]

\draw[fill=blue!20](0,0)--(1.5,-1)--(3,0)--(3,2)--(1.5,3)--(0,2)--(0,0);
\node at (0,0) {$\bullet$};\node at (1.5,-1) {$\bullet$};\node at (3,0) {$\bullet$};\node at (3,2) {$\bullet$};\node at (1.5,3) {$\bullet$};\node at (0,2) {$\bullet$};

\node[rotate=30] at (2.5,-0.33){$>>$};
\node[rotate=-30] at (2.5,2.33){$>>$};

\node at (0,1){$\bullet$};
\node[rotate=-90] at (0,0.5) {$>$};
\node[rotate=90] at (0,1.5){$>$};

\draw[thick,cyan] (2.5,2.33)--(0,1)--(2.5,-0.33);

\node[rotate=50,cyan] at (0.5,0.7) {$\bowtie$};

\end{scope}
\end{tikzpicture}}\]

Note that these two arcs are not tagged arcs in the sense of \cite{FST}. Indeed if an arc connects a puncture to itself and does not enclose, on either side, a once punctured monogon, both ends are tagged the same way.  One can also easily check that these modules are not rigid, so are not summands of a cluster-tilting object in the cluster category.

\section*{Acknowledgements}
The results in this paper were presented during the conference on the occasion of Idun Reiten's birthday in Trondheim in May 2017.
We would like to thank the organizers of this conference, and dedicate this paper to Idun Reiten. The first author would like to thank Fran\c cois Dahmani for interesting discussions on free groups and for pointing out reference \cite{Serre}. 
We are grateful to Anna Felikson, Daniel Labardini-Fragoso, Pavel Tumarkin and Yu Zhou for interesting discussions.

\newcommand{\etalchar}[1]{$^{#1}$}
\providecommand{\bysame}{\leavevmode\hbox to3em{\hrulefill}\thinspace}
\providecommand{\MR}{\relax\ifhmode\unskip\space\fi MR }
\providecommand{\MRhref}[2]{%
  \href{http://www.ams.org/mathscinet-getitem?mr=#1}{#2}
}

\end{document}